\documentclass[leqno,11pt]{amsart}
\usepackage{amsmath,amssymb}
\usepackage{color}
\usepackage{comment}
\usepackage{todonotes}

\allowdisplaybreaks
\makeatletter
\long\def\unmarkedfootnote#1{{\long\def\@makefntext##1{##1}\footnotetext{#1}}}
\makeatother

\theoremstyle{plain}
\newtheorem{theorem}{Theorem}[section]
\newtheorem{lemma}[theorem]{Lemma}
\newtheorem{prop}[theorem]{Proposition}

\newtheorem{remark}[theorem]{Remark}

\allowdisplaybreaks

\def\rn{\mathbb R\sp n}
\def\Rn{\mathbb R\sp n}
\def\R{\mathbb R}
\def\N{\mathbb N}

\def\M{\mathcal M}

\def\o{\Omega}

\def\Mpl{\mathcal M_+}

\def\lt{\left}
\def\rt{\right}

\newtoks\by
\newtoks\paper
\newtoks\book
\newtoks\jour
\newtoks\yr
\newtoks\pages
\newtoks\vol
\newtoks\publ
\newtoks\eds
\newtoks\proc
\newtoks\no
\def\ota{{\hbox{???}}}
\def\cLear{\by=\ota\paper=\ota\book=\ota\jour=\ota\yr=\ota
\pages=\ota\vol=\ota\publ=\ota}
\def\endpaper{\the\by, \textit{\the\paper},
{\the\jour} \textbf{\the\vol} (\the\yr), \the\pages.\cLear}
\def\endbook{\the\by, \textit{\the\book}, \the\publ.\cLear}
\def\endprep{\the\by, \textit{\the\paper}, \the\jour.\cLear}
\def\endproc{\the\by, \textit{\the\paper}, \the\publ, \the\pages.\cLear}
\def\name#1#2{#1 #2}
\def\et{ and }

\numberwithin{equation}{section}

\newcommand{\norm}[1]{{\left\vert\kern-0.25ex\left\vert\kern-0.25ex\left\vert #1
    \right\vert\kern-0.25ex\right\vert\kern-0.25ex\right\vert}}

\title[Sobolev embeddings, rearrangement-invariant spaces and Frostman measures]{Sobolev embeddings, rearrangement-invariant spaces \\ and Frostman measures
} 
\numberwithin{equation}{section}

\author{
  Andrea Cianchi,  Lubo\v s Pick, Lenka Slav\'ikov\'a}

\address{
 \it Dipartimento di Matematica e Informatica \lq\lq U. Dini", Universit\`a di Firenze, \it Viale Morgagni 67/A, 50134 Firenze, Italy, e-mail: cianchi@unifi.it}

\address{\it  Department of Mathematical Analysis, Faculty of Mathematics and
Physics,  Charles University, Sokolovsk\'a~83,
186~75 Praha~8, Czech Republic, email: pick@karlin.mff.cuni.cz}

\address{\it Department of Mathematics,
University of Missouri, Columbia, MO 65211,
USA, and Department of Mathematical Analysis, Faculty of Mathematics and
Physics,  Charles University, Sokolovsk\'a~83,
186~75 Praha~8, Czech Republic, \newline
email: slavikova@karlin.mff.cuni.cz}

\date{}

\begin{document}

%
%

\maketitle

\begin{equation*}
\hbox{{\bf Abstract}} \end{equation*}
Sobolev embeddings, of arbitrary order, are considered into function spaces on domains of $\rn$ endowed with measures whose  decay on balls  is dominated by a  power $d$ of their radius. Norms in arbitrary rearrangement-invariant spaces are contemplated. A comprehensive approach is proposed based on the reduction of the relevant $n$-dimensional embeddings to one-dimensional Hardy-type inequalities. Interestingly, the latter inequalities  depend on the involved measure only through the power $d$.
Our results allow for the detection of the optimal  target space in  Sobolev embeddings, for  broad  families of norms, in situations where  customary techniques do not apply. In particular, new embeddings, with augmented  target spaces,  are deduced even for  standard Sobolev spaces.

\medskip
\par\noindent
\begin{equation*}
\hbox{{\bf R\'esum\'e}} \end{equation*}
On consid\`ere des immersions de Sobolev d'ordre quelconque dans des espaces de fonctions sur des domaines de $\rn$ munis des mesures avec une tendance dans les boules qui est domin\'ee par une puissance $d$ du rayon. Des normes dans les espaces arbitraires invariants par r\'earrangements sont permises.
Nous proposons une approche g\'en\'eral bas\'ee sur la r\'eduction des immersions  en dimension $n$  \`a des in\'egalit\'es du type Hardy en dimension un. On souligne que ces in\'egalit\'es d\'ependent de la mesure consid\'er\'ee seulement par le degr\'e de puissance $d$.
Notre r\'esultat permets de d\'etecter l'espace cible optimal dans les immersions de Sobolev, pour une large famille de normes dans des cas o\`u le techniques habituelles ne s'appliquent pas.
En particulier on d\'eduit des nouvelles immersions avec espaces cible augment\'es m\^eme dans le cas d'espace de Sobolev standard.

\unmarkedfootnote {
\par\noindent {\it Mathematics Subject
Classification:} 46E35, 46E30.
\par\noindent {\it Keywords:}  Sobolev inequalities, Frostman measures, Ahlfors regular measures, rearrangement-invariant spaces, Lorentz spaces.
%
%
}

\section{Introduction}\label{intro}

The purpose of the present paper is to offer new reduction principles for an extensive class of Sobolev-type inequalities. The relevant principles enable us to derive   Sobolev inequalities on  $n$-dimensional domains from considerably simpler one-dimensional inequalities for Hardy-type operators. The idea of deducing higher-dimensional  from one-dimensional inequalities, via some sort of symmetrization or rearrangement, is a classical one and has been implemented in a number of questions in diverse fields, such as partial differential equations, theory of function spaces, harmonic analysis. None of the available results is however apt at proving sharp Sobolev embeddings  involving  general norms and measures, also known as trace embeddings in the literature, as those considered by our approach. It can be exploited to  establish original Sobolev inequalities, with optimal (i.e. strongest) target norms, in   diverse classes of function spaces, including Lorentz and Orlicz spaces. In order to maintain our discussion  within a reasonable length, here we limit ourselves to applying our comprehensive results to trace  embeddings for the  standard Sobolev spaces. As a consequence,  the  optimal targets in these embeddings among all rearrangement-invariant spaces with respect to  Ahlfors regular measures are exhibited. This improves classical results by Adams
and by Maz'ya,
 and provides us with a version with measures of  the Br\'ezis-Wainger inequality
  for the limiting case.
 The spin-off for specific Sobolev embeddings embracing other families of function norms will be presented in \cite{CPS_appl}.
\par
The main advances of our contribution are highlighted below. A preliminary  brief review of the status of the art in the area is provided to help grasp their novelty.

\medskip

\par\noindent {\bf First-order Sobolev inequalities with vanishing boundary conditions}. \,
A powerful idea in the proof of a variety of first-order Sobolev inequalities for functions vanishing on the boundary of their domain $\o \subset \rn$, with $n \geq 2$, is that they can be deduced from  their analogues for radially symmetric functions on  balls, via symmetrization. As a consequence, the original inequalities are converted into  parallel  one-dimensional Hardy-type inequalities. This approach has its modern roots in the work of  Moser \cite{Moser}, Aubin \cite{Aubin} and Talenti \cite{Ta}, and has been successfully exploited  in a number of subsequent  contributions. The underlying symmetrization technique rests upon the fact that Lebesgue and, more generally, rearrangement-invariant norms of a function vanishing on the boundary of its domain are invariant under replacement with its radially decreasing symmetral, and that gradient norms from the same class do not increase under the same replacement. The latter property is known as P\'olya-Szeg\"o inequality, and is, in its turn, a consequence of the classical isoperimetric inequality in $\rn$. The close connection between isoperimetric and Sobolev inequalities in quite general settings was discovered in the early sixties of the last century in the seminal researches by Maz'ya \cite{Ma1960}, who also elucidated the equivalence between isocapacitary and Sobolev inequalites \cite{Ma1961}. Special cases were also pointed out by Federer-Fleming \cite{FF}.
Related reduction principles can be found in \cite{CR-1, CR-2}.
\par
An inclusive version of this reduction principle can be stated as follows. Assume that $\o$ has  finite   Lebesgue measure and that $X(\o)$ and $Y(\o)$ are any Banach function spaces endowed with rearrangement-invariant norms, briefly called
  rearrangement-invariant spaces -  Lebesgue, Orlicz or Lorentz spaces, for instance.
 Then
 the Sobolev inequality
\begin{equation}
\label{classical0}
\|u\|_{Y(\o)} \leq C_1 \|\nabla u\|_{X(\o)}
\end{equation}
holds for some constant $C_1$ and for every  function $u \in W^1_0X(\o)$, provided that the  Hardy-type inequality
\begin{equation}
\label{hardy0}
\bigg\|\int _t^1f(s) s^{-1+\frac 1n}\, ds\bigg\|_{Y(0,1)} \leq C_2 \|f\|_{X(0,1)}
\end{equation}
holds for some constant $C_2$ and for every function $f \in X(0,1)$. Here, $W^1_0X(\o)$
denotes
the Sobolev space of those weakly differentiable functions in $\o$, vanishing in a suitable sense on $\partial \o$, such that $|\nabla u|\in X(\o)$, and $X(0,1)$ and $Y(0,1)$ stand for the one-dimensional representation spaces of $X(\o)$ and $Y(\o)$.
 The converse is also true, and hence  inequalities \eqref{classical0} and \eqref{hardy0} are actually equivalent.  The constants $C_1$ and $C_2$ appearing  in these inequalities only depend on each other and on $n$ and $|\o|$, the Lebesgue measure of $\o$. This reduction principle is a key step in the characterization of the optimal target space $Y(\o)$ for a prescribed space $X(\o)$ in inequality \eqref{classical0} \cite{CPS, KP}.

\medskip

\par\noindent {\bf First-order Sobolev inequalities without boundary   conditions}. \,
Assume that $\Omega$ is bounded. The P\'olya-Szeg\"o inequality fails for functions which need not vanish on $\partial \o$. Hence, symmetrization cannot be applied to derive Sobolev-type inequalities for functions, with arbitrary boundary values, that are just required to belong to $W^1X(\o)$, the Sobolev space   of those weakly differentiable functions in $\o$ that, together with their gradient, belong to $X(\o)$. However, inequality \eqref{classical0}, with the norm $\|\nabla u\|_{X(\o)}$ replaced with the full norm $\|u\|_{W^1X(\o)}$, and inequality \eqref{hardy0} are easily seen to be still equivalent, if the domain $\o$ is regular enough, just by an extension argument. Interestingly, a version of this reduction principle holds in more general domains $\o$, with the power weight $r^{-1+\frac 1n}$ in \eqref{hardy0} replaced with the reciprocal of another function, depending on $\o$, and called the isoperimetric function, or isoperimetric profile, of $\Omega$ in the literature. This function, denoted by $I_\o$, appears in the relative isoperimetric inequality in $\o$.  Loosely speaking, $I_\o : (0,1) \to [0, \infty)$ is an increasing function, whose value at $s \in (0, \tfrac 12)$ agrees with the infimum of the perimeter relative to $\o$ among all   subsets of $\o$  whose measure lies in $[s|\o|, \tfrac 12 |\o|]$.
Under some mild assumptions on the decay of $I_\o$ near $0$,  the Sobolev inequality
\begin{equation}
\label{classical}
\|u\|_{Y(\o)} \leq C_1 \|u\|_{W^1X(\o)}
\end{equation}
holds for some constant $C_1$ and for every  function $u \in W^1X(\o)$ if  the  Hardy-type inequality
\begin{equation}
\label{hardyI}
\bigg\|\int _t^1f(s) \frac{ds}{I_\o(s)}\, \bigg\|_{Y(0,1)} \leq C_2 \|f\|_{X(0,1)}
\end{equation}
holds for some constant $C_2$ and for every function $f \in X(0,1)$ \cite{CPS}. The  converse implication depends on the geometry of $\o$, and, loosely speaking, it holds when $\o$ is not too  irregular, a property that is reflected in a sufficiently mild decay of $I_\o$ at $0$.

\medskip


\par\noindent {\bf Higher-order Sobolev inequalities}. \,
The situation is even more delicate when higher-order Sobolev inequalities are in question, namely when we are dealing with  functions endowed with weak derivatives up to some order $m \in \N$ which belong to some rearrangement-invariant space $X(\o)$, a space of functions that will be denoted by
$W^mX(\o)$.  Indeed, if $m>1$, symmetrization methods fail yet in the subspace $W^m_0X(\o)$ of $W^mX(\o)$ of those   functions vanishing, together with their derivatives up to the order $m-1$, on $\partial \o$ (see \cite{Cianchi-Duke, Cianchi_AMPA} for  partial results in this connection when $m=2$).
\par
In \cite{CPS} a sharp iteration method has been developed to derive a  reduction principle for arbitrary-order Sobolev inequalities from the first-order one mentioned above. In case of inequalities involving functions with unrestricted boundary values, namely inequalities of the form
\begin{equation}
\label{classicalm}
\|u\|_{Y(\o)} \leq C_1\|u\|_{W^mX(\o)}
\end{equation}
for
$u \in W^mX(\o)$,  the  relevant reduction principle is driven by a Hardy-type operator which, in general,   involves  a kernel built upon $I_\o$, and not just   a weight function. Specifically, the  resultant Hardy inequality takes the form
\begin{equation}
\label{hardym-kernel}
\bigg\|\int_t\sp1\frac{f(s)}{I_\o(s)}\left(\int_t\sp
s\frac{dr}{I_\o(r)}\right)\sp{m-1}\,ds\bigg\|_{Y(0,1)} \leq
C_2\|f\|_{X(0,1)}\,
\end{equation}
for $f \in X(0,1)$.
In fact, the result of \cite{CPS} applies to a broad class of Sobolev-type inequalities, where the norms on both sides of \eqref{classicalm}  and the isoperimetric function $I_\o$ are taken with respect to measures with a density, i.e. measures which are absolutely continuous with respect to Lebesgue measure.
%
Under an additional assumption on the behavior of $I_\o$ near $0$, the kernel
in the integral operator on the left-hand side of \eqref{hardym-kernel}
 can be replaced by the function $\frac {s^{m-1}}{I_\o(s)^m}$
of the sole variable $s$.
In particular, when $\o$ is  regular enough -- a bounded Lipschitz domain, for instance --  and equipped with Lebesgue measure, one has that $I_\o(s)$ decays like $ s^{1- \frac{1}{n}}$ as $s\to 0^+$, and
 inequality  \eqref{classicalm}
holds for some constant $C_1$ and for every $u \in W^mX(\o)$, if (and only if) the  Hardy-type inequality
\begin{equation}
\label{hardym}
\bigg\|\int _t^1f(s) s^{-1+\frac mn}\, ds\bigg\|_{Y(0,1)} \leq C_2 \|f\|_{X(0,1)}
\end{equation}
holds for some constant $C_2$ and every $f \in X(0,1)$. The same characterization applies to any domain $\o$ with $|\o|<\infty$, provided that the space $W^mX(\o)$ is replaced with
$W^m_0X(\o)$. This special case was earlier obtained, via a different approach, in \cite{KP}.
\\ Reduction principles for  Sobolev type inequalities, for different kinds of norms or domains, are the subject of \cite{ACPS, GMNO, GNO, Holik}. Compactness of Sobolev embeddings is characterized via reduction principles in \cite{CM, Slav}.

\medskip

\par\noindent {\bf New results}. \,
In the present paper we abandon the point of view of linking Sobolev to isoperimetric inequalities, and  pursue the approach to a wider family of Sobolev inequalities, via reduction principles, from a different perspective. A basic version of the inequalities that will be considered, called Sobolev trace inequalities in a broad sense,  has the form
\begin{equation*}
\|u\|_{Y(\o, \mu)}\leq C \|\nabla ^mu\|_{X(\o)}\,
\end{equation*}
for some constant $C$ and every $u \in W^m_0X(\o)$. Here, $\Omega$ is an  open set in $\Rn$, with $|\o|<\infty$ and $Y(\o, \mu)$ and $X(\o)$ are rearrangement-invariant spaces. The crucial novelty now is that the norm in $Y(\o, \mu)$ is taken with respect to a finite Borel measure $\mu$ on $\o$ under the sole piece of information that
\begin{equation}\label{E:sup}
\sup_{x\in \Rn, r>0} \frac{\mu(B_r(x)\cap {\Omega})}{r^d} <\infty
\end{equation}
for some $d \in (0,n]$, where
$B_r(x)$ denotes the ball centered at $x$, with radius $r$. This classical class of measures, called $d$-Frostman measures or $d$-upper Ahlfors regular measures in the literature, enters in a number of questions in measure theory, harmonic analysis, theory of function spaces. This is in fact the class of measures for which standard Sobolev trace embedding theorems by Adams \cite{Adams2,Adams3} and Maz'ya \cite{Mazya543,Mazya548} are established.
Plainly, condition \eqref{E:sup} is fulfilled with $d=n$ if $\mu$ is Lebesgue measure.
\par
Inequalities for functions with unrestricted boundary values, namely inequalities of the form
$$\|u\|_{Y(\o, \mu)} \leq  C \|u\|_{W^mX(\o)}\,$$
for some constant $C$ and every $u \in W^mX(\o)$, are also considered under the same decay assumption on $\mu$ and under suitable regularity assumptions on a bounded domain $\o$.
\par
Our discussion  encompasses Sobolev-type inequalities up to the boundary that read
$$\|u\|_{Y(\overline \o, \mu)}\leq C  \|u\|_{W^mX(\o)}\,$$
for some constant $C$ and every $u \in W^mX(\o)$, and  measures $\mu$, whose support is contained in the whole of $\overline \o$,  satisfying the corresponding condition
\begin{equation}\label{E:supbar}
\sup_{x\in \Rn, r>0} \frac{\mu(B_r(x)\cap \overline{\Omega})}{r^d} <\infty
\end{equation}
for some $d \in (0,n]$. In the case when the measure $\mu$ is supported in $\partial \o$, these inequalities embrace, as special instances, standard boundary trace inequalities and weighted  inequalities on the boundary.
\par
An important trait of our results is that the one-dimensional operators coming into play in the pertaining reduction principles only depend on $n$, $m$ and $d$. In particular, when $\mu$ is Lebesgue measure (or just $d=n$),  the results recalled above are recovered. Trace inequalities over the intersection of $\o$ with $d$-dimensional affine subspaces of $\rn$, established in \cite{CP-Trans} for $d$ in a certain range, are  reproduced as well.
\par
The results to be presented are new even for first-order Sobolev spaces. However, even with the first-order case at disposal, an iteration process as in \cite{CP-Trans, CPS} would not apply to carry it over to the higher-order case. We have thus to resort to a strategy based upon sharp endpoint inequalities, coupled with  $K$-functional techniques from interpolation theory, which is reminiscent of arguments  from \cite{CKP, KP}. Additional difficulties yet arise in the present setting.
\par
To begin with,  whereas
a trace operator  with respect to the measure $\mu$  is well defined on any Sobolev space $W^mX(\o)$,  if $d \in [n-m, n]$, whatever $X(\o)$ is, this need not be the case when $d\in (0, n-m)$. For $d$ in the latter range, the norm in $X(\o)$ has to be strong enough, depending on   $n$, $m$ and $d$, for a trace to exist. This leads to  reduction principles with different features in the two cases, that we shall refer to as \lq\lq fast decaying measures" ($d \in [n-m, n]$) and \lq\lq slowly decaying measures" ($d\in (0, n-m)$).
\par
One of the endpoint estimates needed in
the fast decay regime was missing, and it is established in this paper as a result of independent interest. The corresponding  endpoint estimate in  the slow decay regime follows from a recent result from \cite{KK}, the special case of traces on subspaces having
 earlier been dispensed in  \cite{CP2010}. On the other hand, the piece of information that can be extracted from the interpolation method is less transparent in this case, and the identification of the correct class of admissible rearrangement-invariant norms in the inequalities requires an ad hoc construction.
\par
A further distinction between the two rates of decay of the measures arises in connection with the necessity of the one-dimensional inequalities playing a role in the reduction principles. The one-dimensional inequality to be exhibited for fast decaying measures is an exact extension of \eqref{hardym}, and  is always necessary, in analogy with the case of Sobolev inequalities with Lebesgue measure, provided that $\mu$ also admits a lower bound with the same exponent $d$ at least at one point.
The reduction principle for slowly decaying measures requires some additional one-dimensional inequality. Moreover, the pair of one-dimensional inequalities prescribed in this case cannot be necessary
in absence of extra information on the measure.
\par Indeed, a major outcome of our analysis is a discussion of the existence of a necessary and sufficient reduction principle enabling one to characterize the optimal target space $Y(\o, \mu)$, or $Y(\overline \o, \mu)$, in the inequalities displayed above. This characterization is actually provided for fast decaying mesures.
By contrast,  in the slow decaying regime  examples of Sobolev inequalties will be produced where the measures have the same exact power type decay   on balls, but the optimal target spaces are different.
\par Prototypical  examples of $d$-Frostman measures,  exhibiting extremal features with regard to   Sobolev trace embeddings, are the measure $\mu_1$, defined as
\begin{equation}\label{mu1}
d\mu_1(x)=\frac{dx}{|x-x_0|^{n-d}}
\end{equation}
for some $x_0 \in \o$, and, when
$d \in \N$, the measure $\mu_2$ given by
\begin{equation}\label{mu2} \mu_2= \mathcal H^{d}|_{\Omega_d},
\end{equation}
where $\Omega _d$ denotes the intersection of $\Omega$ with a $d$-dimensional affine subspace of $\mathbb R^n$.
\par The measure $\mu_2$ is, in view of Sobolev inequalities, the worst possible measure fulfilling condition \eqref{E:sup} for some given $d$. In fact,
Sobolev inequalities associated with $\mu_2$ typically admit the weakest target norms among all measures fulfilling condtion \eqref{E:sup}  for some integer $d\in (0, n-m)$. They correspond to traces in the classical sense of   \lq\lq restrictions" to hyperplanes. The measure $\mu_2$   will indeed be called into play to prove the sharpness of one of our endpoint results.
\par On the other hand, the measure $\mu_1$ is a distinguished member of the special class of $d$-Frostman measures that admit a radially decreasing density  with respect to Lebesgue measure. Measures from this class will be
 shown to allow
 for stronger optimal target norms than those associated with an arbitrary $d$-Frostman measure in certain Sobolev embeddings for the space $W^mX(\o)$, when $d \in (0,n-m)$.
 Sobolev inequalities with a weight as in \eqref{mu1} have been extensively studied in the literature. For instance, when $X(\o)=L^p(\o)$, $Y(\o,\mu)=L^q(\o, \mu_1)$ and $m=1$, the weighted Sobolev inequality in question coincides with a special instance of the Caffarelli-Kohn-Nirenberg inequality, that has attracted the interest of researchers over the years in connection with the existence and description of its extremals if $\o= \rn$ -- see e.g. \cite{ABCMP, CW, DEFT, DEL, FS}.
\par As mentioned above,  just to give the flavor of the conclusions that can be reached by our approach, it will be implemented   to exhibit the optimal rearrangement-invariant target norm $Y(\o, \mu)$, for an arbitrary $d$-Frostman measure $\mu$, in the inequality
\begin{equation}\label{sobintro}
\|u\|_{Y(\o, \mu)} \leq C \|\nabla\sp{m} u\|_{L^p(\o)}
\end{equation}
for $u \in W^{m,p}_0(\o)$. Parallel inequalities with $W^{m,p}_0(\o)$ replaced by  $W^{m,p}(\o)$, or $\o$ replaced by $\overline \o$, are considered as well.
Analogous problems, where the $L^p(\o)$ norm  on the right-hand side of inequality \eqref{sobintro} is replaced by a more general Orlicz or Lorentz norm, entail additional new tools in dealing with the associated one-dimensional inequalities, and  are treated  in the separate contribution  \cite{CPS_appl}.

%
%


\section{Background}\label{back}

Throughout the paper, the relation $\lq\lq \lesssim "$ between two positive expressions means that the former is bounded by the latter, up to a multiplicative constant depending on quantities to be specified.
The relation $\lq\lq \approx "$  between two expressions means that   they are bounded by each
other  up to multiplicative constants depending on quantities to be specified.
\par
Let $(\mathcal R,\nu)$ be a $\sigma$-finite non-atomic measure space. We denote by
$\M(\mathcal R , \nu)$
 the set of all $\nu$-measurable functions
on~$\mathcal R$ taking values in $[-\infty,\infty]$. We
also define $\Mpl(\mathcal R,\nu)=\{\phi \in\M(\mathcal R,\nu)\colon \phi \geq 0 \
\nu\textup{-a.e. on}\ \mathcal R\}$ and
$\M_0(\mathcal R,\nu)=\{\phi \in\M(\mathcal R,\nu)\colon \phi\ \textup{is finite}\
\nu\textup{-a.e. on}\ \mathcal R\}$. When $\mathcal R\subset \rn$ and $\nu$ is Lebesgue measure, we denote $\M(\mathcal R , \nu)$ simply by $\M(\mathcal R)$, and similarly for $\Mpl(\mathcal R,\nu)$ and $\M_0(\mathcal R,\nu)$.

Given a~function $\phi \in \M(\mathcal R , \nu)$, its \textit{non-increasing
rearrangement} $\phi\sp*_\nu : [0, \infty ) \to [0, \infty]$  is defined as
$$
\phi\sp*_\nu(t)=\inf\{{\varrho\in\mathbb R}:\,\nu\left(\{x\in \mathcal R
:\,|\phi(x)|>\varrho\}\right)\leq t\}\quad \text{for $t\in [0,\infty)$}.
$$
When $\nu$ is Lebesgue measure, we omit the subscript $\nu$, and just write $\phi^*$ instead of $\phi^*_\nu$.
The \textit{maximal non-increasing rearrangement} of $\phi$, denoted by $\phi^{**} _\nu : (0, \infty) \to [0, \infty ]$, is defined by
$$
\phi^{**}_\nu(t) = \frac{1}{t}\int _0^t \phi^*_\nu(s)\, ds \quad \text{for $t\in (0,\infty)$}.
$$
The function $\phi^{**}_\nu$ is also non-increasing and  $\phi^{*}_\nu(t) \leq \phi^{**}_\nu(t)$   for every $t\in (0, \infty)$. The operation $\phi\mapsto \phi\sp{**}_\nu$ is subadditive in the sense that
\begin{equation}\label{subadd}
\int_0\sp t(\phi + \psi)^{*}_\nu(s)\,ds \leq \int_0\sp t \phi^{*}_\nu(s)\,ds +
\int_0\sp t \psi ^{*}_\nu(s)\,ds\quad \text{for $t\in [0,\infty)$}
\end{equation}
for every $\phi,\psi\in\mathcal M_+(\mathcal R , \nu)$. Although the operation  $\phi\mapsto \phi\sp{*}$ is not  subadditive, yet
\begin{equation}\label{gen5}
(\phi + \psi)^{*}_\nu(t) \leq  \phi^{*}_\nu(t/2) +  \psi ^{*}_\nu(t/2)\quad \text{for $t\in [0,\infty)$}
\end{equation}
for every $\phi,\psi\in\mathcal M_+(\mathcal R , \nu)$.
 Two measurable functions $\phi$ and $\psi$ (possibly defined on   different measure spaces) are said to be \textit{equimeasurable} if
$\phi\sp*_\nu=\psi\sp*_\nu$.

A fundamental property of rearrangements is the \textit{Hardy-Littlewood inequality}, which asserts that
\begin{equation}\label{E:HL}
\int _{\mathcal R} |\phi(x) \psi(x)| d\nu (x) \leq \int _0\sp\infty  \phi^*_\nu(t) \psi^*_\nu(t)\, dt
\end{equation}
 for every $\phi, \psi \in\M(\mathcal R , \nu)$.

Assume that either $L=1$ or $L=\infty$. A functional
 $\|\cdot\|_{X(0,L)}{:\Mpl(0,L)\to[0,\infty]}$ is called a
\textit{function norm} if, for all $f$, $g \in \Mpl(0,L)$, all
$\{f_n\}_{n\in\N}\subset{\Mpl(0,L)}$, and every $\lambda {\in[0,\infty)}$:
\begin{itemize}
\item[(P1)]\quad $\|f\|_{X(0,L)}=0$ if and only if $f=0$ a.e.;
$\|\lambda f\|_{X(0,L)}= \lambda \|f\|_{X(0,L)}$; \par\noindent \quad
$\|f+g\|_{X(0,L)}\leq \|f\|_{X(0,L)}+ \|g\|_{X(0,L)}$;
\item[(P2)]\quad $ f \le g$ a.e.\  implies $\|f\|_{X(0,L)}
\le \|g\|_{X(0,L)}$;
\item[(P3)]\quad $f_n \nearrow f$ a.e.\
implies $\|f_n\|_{X(0,L)} \nearrow \|f\|_{X(0,L)}$;
\item[(P4)]\quad $\|\chi _E\|_{X(0,L)}<\infty$ if $|E| < \infty$;
\item[(P5)]\quad if $|E|< \infty$, there exists a constant
 $C$ depending on $E$  such that $\int_E f(t)\,dt \le C
\|f\|_{X(0,L)}$.
\end{itemize}
Here, $E$ denotes a measurable set in $(0,L)$, and  $\chi_E$ stands for its characteristic function.
If, in addition,
\begin{itemize}
\item[(P6)]\quad $\|f\|_{X(0,L)} = \|g\|_{X(0,L)}$ whenever $f\sp* = g\sp *$,
\end{itemize}
we say that $\|\cdot\|_{X(0,L)}$ is a
\textit{rearrangement-invariant function norm}.

The \textit{associate function norm}  $\|\cdot\|_{X'(0,L)}$ of a function norm $\|\cdot\|_{X(0,L)}$ is  defined as
$$
\|f\|_{X'(0,L)}=\sup_{\begin{tiny}
                        \begin{array}{c}
                       {g\in{\Mpl(0,L)}}\\
                        \|g\|_{X(0,L)}\leq 1
                        \end{array}
                      \end{tiny}}
\int_0\sp{L}f(t)g(t)\,dt
$$
for $ f\in\Mpl(0,L)$.
Note that
\begin{equation}\label{X''}
\|\cdot \|_{(X')'(0,L)} = \|\cdot \|_{X(0,L)}.
\end{equation}

Given $\alpha\in[1,\infty)$ and a rearrangement-invariant function norm $\|\cdot \|_{X(0,L)}$, the functional $\|\cdot\|_{X^{\{\alpha\}}(0,L)}$, defined as
\begin{equation}\label{Zp}
\|f\|_{X^{\{\alpha\}}(0,L)}=
\|f^{\alpha} \|_{X(0,L)} ^{\frac 1{\alpha}}
\end{equation}
for $f \in \Mpl(0,L)$, is also a rearrangement-invariant function norm (see e.g.~\cite{MaPe}). The same is true for the functional  $\|\, \cdot \, \|_{X^{\langle\alpha\rangle}(0,1)}$ defined as
\begin{equation}\label{Yalpha}
\|f\|_{X^{\langle\alpha\rangle}(0,L)}=\big\|\left(\left(f\sp{\alpha}\right)\sp{**}\right)\sp{\frac{1}{\alpha}}\big\|_{{X}(0,L)}
\end{equation}
for   $f \in \Mpl(0,L)$ (see e.g.~\cite{hanicka}). One can show that, if $L< \infty$, then there exists a constant $C$ depending on $\alpha$, $L$ and $\|\cdot \|_{X(0,L)}$   such that
\begin{equation}\label{gen35}
\|f \|_{X(0,L)} \leq C \|f\|_{X^{\{\alpha\}}(0,L)}
\end{equation}
for $f \in \Mpl(0,L)$. Moreover,
\begin{equation}\label{gen36}
\|f \|_{X(0,L)} \leq  \|f\|_{X^{\langle\alpha\rangle}(0,L)}
\end{equation}
for $f \in \Mpl(0,L)$.
\par
Let
 $\|\cdot\|_{X(0,L)}$ be   a rearrangement-invariant function norm and let $(\mathcal R, \nu)$ be a measure space as above. Assume that either $\nu (\mathcal R)=\infty$ and $L=\infty$, or  $\nu (\mathcal R)<\infty$ and $L=1$. Then the space $X(\mathcal R , \nu)$ is
defined as the collection of all  functions  $\phi \in\M(\mathcal R , \nu)$
such that the quantity
\begin{equation}\label{norm}
\|\phi\|_{X(\mathcal R,\nu)}=\begin{cases}  \|\phi\sp*_\nu(t)\|_{X(0,\infty)} & \quad \hbox{if $\nu (\mathcal R)= \infty$}
\\
\|\phi\sp*_\nu(\nu(\mathcal R)t)\|_{X(0,1)} & \quad \hbox{if $\nu (\mathcal R)< \infty$,}
\end{cases}
\end{equation}
is finite. The space $X(\mathcal R , \nu)$ is a Banach space, endowed
with the norm given by \eqref{norm}. If
$\mathcal R \subset \rn$ and $\nu$ is Lebesgue measure, we denote
$X(\mathcal R,\nu)$ simply by $X(\mathcal R)$.
 The space $X(0,L)$ is called the
\textit{representation space} of $X(\mathcal R , \nu)$.

The \textit{associate
space}  $X'(\mathcal R , \nu)$
of   a~rearrangement-invariant~space $X(\mathcal R , \nu)$ is
 the
rearrangement-invariant space   built upon the
function norm $\|\cdot\|_{X'(0,L)}$. By \eqref{X''},
$X''(\mathcal R , \nu)=X(\mathcal R , \nu)$.  Hence, any rearrangement-invariant
space $X(\mathcal R , \nu)$ is always the associate space of another
rearrangement-invariant space, namely $X'(\mathcal R , \nu)$. Furthermore,
the \textit{H\"older inequality}
\[
\int_{\mathcal R}|\phi(x)\psi(x)|\,d\nu(x)\leq\|\phi\|_{X(\mathcal R ,
\nu)}\|\psi\|_{X'(\mathcal R , \nu)}
\]
holds  for every $\phi$ and $\psi$ in $\M(\mathcal R , \nu)$.

Let $X(\mathcal R , \nu)$ and $Y(\mathcal R , \nu)$ be rearrangement-invariant\
spaces. We write $X(\mathcal R , \nu) \to Y(\mathcal R , \nu)$ to denote that
$X(\mathcal R , \nu)$ is continuously embedded into $Y(\mathcal R , \nu)$, in the sense that
there exists a  constant $C$ such that
$\|\phi\|_{Y(\mathcal R , \nu)}\leq C\|\phi\|_{ X(\mathcal R , \nu)}$ for every $\phi\in\M(\mathcal R, \nu)$.
Note that the  embedding $X(\mathcal R , \nu) \to Y(\mathcal R , \nu)$ holds if
and only if there exists a  constant $C$ such that
$\|f\|_{Y(0,L)}\leq C\|f\|_{X(0,L)}$ for every $f\in\Mpl(0,L)$.
A property of function norms ensures that
$$
\hbox{$X(\mathcal R , \nu) \subset Y(\mathcal R , \nu)$ \qquad if and only if
\qquad $X(\mathcal R , \nu) \to Y(\mathcal R , \nu)$}.
$$
Moreover,  for any rearrangement-invariant spaces $X(\mathcal R , \nu)$ and
$Y(\mathcal R , \nu)$,
\begin{equation}\label{emb}
\hbox{ $X(\mathcal R , \nu) \rightarrow Y(\mathcal R , \nu)$ \quad  if and only if \quad
$Y'(\mathcal R , \nu) \rightarrow X'(\mathcal R , \nu)$,}
\end{equation}
with the same embedding constants.

Given any $s>0$, let $E_s$ be the  \textit{dilation operator} defined at $f \in \M(0, \infty)$ as
$$
  (E_sf)(t)=
  f(t/s) \quad \text{for $t\in (0,\infty)$,}
$$
and at
$f\in \M(0,1)$ by
$$
  (E_sf)(t)=\begin{cases}
  f(t/s)\quad&\textup{if}\ 0<t\leq s,\\
  0&\textup{if}\ s<t<1.
  \end{cases}
$$
The operator $E_s$
is bounded on any rearrangement-invariant~space $X(0,L)$, with norm
not exceeding $\max\{1, s\}$.

\textit{Hardy's lemma} tells us that if $f_1, f_2 \in \M_+(0,L)$,
and
\begin{equation}\label{E:hardy-lemma}
\int _0^t f_1(s) ds \leq \int _0^t f_2(s) ds \quad \text{for $t\in (0,L)$}, \quad \text{then} \quad
\int _0^L f_1(t) g(t)\, dt \leq \int _0^L f_2(t)g(t)\, dt
\end{equation}
for every non-increasing function $g : (0,L) \rightarrow [0, \infty
]$. A consequence of this result is the \textit{Hardy--Littlewood--P\'olya
principle} which asserts that, if   $f,g\in\Mpl(0,L)$
and
\begin{equation}\label{E:HLP}
\int_0\sp tf\sp*(s)\,ds\leq \int_0\sp tg\sp*(s)\,ds \quad \text{for $t\in (0,L)$},
\quad
\text{then} \quad
\|f\|_{X(0,L)}\leq \|g\|_{X(0,L)}
\end{equation}
for every rearrangement-invariant function norm $\|\cdot\|_{X(0,L)}$.
\\
If  $\nu (\mathcal R ) < \infty$, then
\begin{equation}\label{l1linf}
L^\infty (\mathcal R , \nu ) \to X(\mathcal R , \nu) \to L^1(\mathcal R , \nu )
\end{equation}
for every rearrangement-invariant space
$X(\mathcal R , \nu)$.

We say that an operator $T$ defined on $\Mpl(0,L)$ and taking values in $\Mpl(0,L)$ is bounded
between  two rearrangement-invariant spaces $X(0,L)$ and $Y(0,L)$,
and we  write
\begin{equation}\label{Tembed}
T:X(0,L)\to Y(0,L),
\end{equation}
if the quantity
$$
\|T\|=\sup\left\{\|Tf\|_{Y(0,L)}:\ f\in X(0,L)\cap\Mpl(0,L),\
\|f\|_{X(0,L)}\leq 1\right\}
$$
is finite. Such a quantity will be called the norm of $T$. The space
$Y(0,L)$  will be called optimal, within a certain class, in
\eqref{Tembed} if, whenever $Z(0,L)$ is another
rearrangement-invariant space, from the same class, such that
$T:X(0,L)\to Z(0,L)$, then $Y(0,L)\to Z(0,L)$. Equivalently,
the corresponding function norm $\| \cdot \|_{Y(0,L)}$ will be said
to be optimal in \eqref{Tembed} in the relevant class.

Assume that $T, T'$ are operators defined on $\Mpl(0,L)$ and taking values in $\Mpl(0,L)$ such that
\begin{equation}\label{E:duo}
\int_0\sp L Tf(t)g(t)\,dt=\int_0\sp Lf(t)T'g(t)\,dt
\end{equation}
for every $f,g\in\M_+(0,L)$. Let $X(0,L)$ and $Y(0,L)$ be
rearrangement-invariant spaces. A simple argument involving Fubini's
theorem and the definition of the associate norm shows that
\begin{equation}\label{E:novabis}
T:X(0,L)\to Y(0,L)\quad \textup{if and only if} \quad T':Y'(0,L)\to
X'(0,L)
\end{equation}
and
$
 \|T\|=\|T'\|
$,
see e.g. \cite[Lemma~8.1]{CPS}.

Consider a~pair $(X_0,X_1)$ of Banach spaces of real-valued functions defined on a  measure space $(\mathcal R,\nu)$ as above.
Their
\textit{$K$--functional} is defined, for each $\phi \in X_0+X_1$, by
\begin{equation}\label{gen1}
  K(t,\phi;X_0,X_1)=\inf_{\phi=\phi_1+\phi_2}\left(\|\phi_1\|_{X_0}+t\|\phi_2\|_{X_1}\right)
\quad \text{for $t\in (0,\infty)$}.
\end{equation}
The   following result  is an easy consequence of
definition \eqref{gen1}. Let $(X_0,X_1)$ and
$(Y_0,Y_1)$  be two   pairs of Banach spaces of real-valued functions defined on  measure spaces $(\mathcal R,\nu)$ and $(\mathcal S,\mu)$, respectively. Let $T$ be a~sublinear operator defined on $X_0+X_1$ and taking values in $\M(\mathcal S,\mu)$. This means  that  there exists a constant $C$ such that $|T(\phi + \psi)|\leq C(|T \phi|+|T\psi|)$ and $|T(\lambda \phi)|= |\lambda| |T\phi|$ for every $\phi, \psi \in X_0+X_1$ and  $\lambda\in\R$. Assume that   $T:X_0\to Y_0$ and $T:X_1\to Y_1$.
Then there exists a constant $C'$, depending only on $C$ and on the
norms of $T$ between $X_0$ and $Y_0$ and between $X_1$ and $Y_1$,
such that
\begin{equation}\label{K}
  K(t,T\phi;Y_0,Y_1)\leq C' K(C't,\phi;X_0,X_1)
  \quad \text{for every $\phi\in X_0+X_1$ and $t\in(0,\infty)$}.
\end{equation}
A classical interpolation theorem by
Calder\'on \cite[Chapter~3, Theorem~2.12]{BS}  asserts that, if $(\mathcal R , \nu)$ is a  measure space as above and $T$ is a
sublinear operator such that $T: L^1(\mathcal R, \nu )\to L^1(\mathcal R , \nu)$ and $T:
L^\infty(\mathcal R, \nu )\to L^\infty(\mathcal R, \nu )$, then
\begin{equation}\label{gen2} T: X(\mathcal R, \nu )\to X(\mathcal R, \nu )
\end{equation}
for every rearrangement-invariant space $X(\mathcal R, \nu )$. Moreover, the  norm of $T$ in \eqref{gen2} depends only on the
norms of $T$ in $L^1(\mathcal R, \nu )$ and in $L^\infty(\mathcal R, \nu )$, and on the constant $C$ appearing in the definition of sublinear operator.
\par
A basic example of a~function norm is the \textit{Lebesgue norm}
$\|\cdot\|_{L\sp p(0,L)}$, defined as usual for $p\in[1,\infty]$.

An important generalization of Lebesgue norms is constituted by the two-parameter Lorentz norms. Assume that $0<p,q\le\infty$. We define the functional
$\|\cdot\|_{L\sp{p,q}(0,L)}$  by
$$
\|f\|_{L\sp{p,q}(0,L)}=
\left\|t\sp{\frac{1}{p}-\frac{1}{q}}f^*(t)\right\|_{L\sp q(0,L)}
$$
for  $f \in {\Mpl(0,L)}$. Here, and in what follows,  we use the convention that $\frac1{\infty}=0$.
If either $1<p<\infty$
and $1\leq q\leq\infty$, or $p=q=1$, or $p=q=\infty$,
 then $\|\cdot\|_{L\sp{p,q}(0,L)}$ is equivalent to a~rearrangement-invariant function norm, and
\begin{equation}\label{E:lorentz-assoc}
(L\sp{p,q})'(0,L)=L\sp{p',q'}(0,L).
\end{equation}
We further define the functional $\|\cdot\|_{L\sp{(p,q)}(0,L)}$ as
$$
\|f\|_{L\sp{(p,q)}(0,L)}=
\left\|t\sp{\frac{1}{p}-\frac{1}{q}}f^{**}(t)\right\|_{L\sp q(0,L)}
$$
for  $f \in {\Mpl(0,L)}$. If either $0< p<\infty$
and $1\leq q\leq\infty$, or $p=q=\infty$, then
$\|\cdot\|_{L\sp{(p,q)}(0,L)}$ is a~rearrangement-invariant function
norm (see e.g.~\cite[Theorem~9.7.5]{PKJF}). The norms
$\|\cdot\|_{L\sp{p,q}(0,L)}$ and  $\|\cdot\|_{L\sp{(p,q)}(0,L)}$ are
called \textit{Lorentz function norms}, and the corresponding spaces
$L\sp{p,q}(\mathcal R,\nu)$ and $L\sp{(p,q)}(\mathcal R,\nu)$ are called
\textit{Lorentz spaces}.

The following inclusion relations between Lorentz spaces hold:
\begin{equation}\label{i}
L\sp{p,p}(0,L)=L\sp p(0,L)\quad \text{for $p\in[1,\infty]$;}
\end{equation}
\begin{equation}\label{ii}
L\sp{p,q}(0,L)\to L\sp {p,r}(0,L)\quad \text{for $p\in[1,\infty]$ and $1\leq q\leq r\leq
\infty$;}
\end{equation}
\begin{equation}\label{iii}
L\sp{(p,q)}(0,L)\to L\sp{p,q}(0,L)\quad \text{for $p,q\in[1,\infty]$;}
\end{equation}
\begin{align}
&\textup{if either}\ p\in(1,\infty)\ \textup{and}\ 1\leq q\leq\infty,\
\textup{or}\ p=q=\infty,\label{iv}\\
&\textup{then}\
L\sp{(p,q)}(0,L)=L\sp{p,q}(0,L)\ \textup{up to equivalent norms. \notag}
\end{align}

Suppose now that $0<p,q\le\infty$ and $\alpha \in\R$. We define the
functional $\|\cdot\|_{L\sp{p,q;\alpha}(0,1)}$ by
\begin{equation}\label{E:1.18}
\|f\|_{L\sp{p,q;\alpha}(0,1)}=
\left\|t\sp{\frac{1}{p}-\frac{1}{q}}\log \sp
\alpha\lt(\tfrac{e}{t}\rt) f^*(t)\right\|_{L\sp q(0,1)}
\end{equation}
for  $f \in {\Mpl(0,1)}$. For suitable choices of the parameters  $p,q, \alpha$, the functional
$\|\cdot\|_{L\sp{p,q;\alpha}(0,1)}$ is equivalent to a~function
norm. If this is the  case, $\|\cdot\|_{L\sp{p,q;\alpha}(0,1)}$  is
called  \textit{Lorentz--Zygmund function norm}, and the
corresponding space $L\sp{p,q;\alpha}(R,\nu)$ is called a
\textit{Lorentz--Zygmund space}. A detailed study of
Lorentz-Zygmund spaces can be found in~\cite{glz}
or~\cite{EOP} -- see also \cite[Chapter~9]{PKJF}. It follows from~\cite[Theorem~4]{Sa} that
\begin{equation}\label{v}
(L\sp{\infty,q;-1})'(0,1)=L\sp{(1,q')}(0,1)\qquad \hbox{for $q\in(1,\infty)$.}
\end{equation}

The $K$--functional for pairs of Lorentz spaces $L\sp{p,q}(\mathcal R, \nu)$
is given, up to equivalence, by the \textit{Holmstedt's
formulas} \cite[Theorem~4.1]{H}. They assert  that, if either
$p_0=q_0=1$, or  $1< p_0<p_1<\infty$ and $1\leq q_0,q_1<\infty$, and $\alpha$ is given by
$\frac1\alpha=\frac{1}{p_0}-\frac{1}{p_1}$, then
\begin{equation}\label{E:fx2.12}
  K(t,\phi;L\sp{p_0,q_0}(\mathcal R, \nu),L\sp{p_1,q_1}(\mathcal R, \nu))\approx
  \left(\int_{0}\sp{t\sp\alpha}\left[s\sp{\frac1{p_0}-\frac1{q_0}}
  \phi\sp*_\nu(s)\right]\sp{q_0}\,ds\right)\sp{\frac1{q_0}}
  +
  t\left(\int_{t\sp\alpha}\sp{\infty}\left[s\sp{\frac1{p_1}-\frac1{q_1}}
  \phi\sp*_\nu(s)\right]\sp{q_1}\,ds\right)\sp{\frac1{q_1}}
\end{equation}
for $t\in (0, \infty)$, up to multiplicative constants depending on $p_0, p_1, q_0, q_1$. Furthermore, if either $p_0=q_0=1$, or $1<p_0<\infty$
and $1\leq q_0<\infty$, then
\begin{equation}\label{E:fx2.13}
  K(t,\phi;L\sp{p_0,q_0}(\mathcal R, \nu),L\sp{\infty}(\mathcal R, \nu))\approx
 \left(\int_{0}^{t\sp{p_0}}
 \left[s\sp{\frac{1}{p_0}-\frac{1}{q_0}}\phi\sp*_\nu(s)\right]\sp{q_0}\,ds\right)\sp{\frac1{q_0}}\quad \text{for $t\in (0,\infty)$},
\end{equation}
 up to multiplicative constants depending on $p_0$ and  $q_0$.
\par
An open set  $\Omega$ in $\rn$ is said to have the \textit{cone property} if
there exists a finite cone $\Lambda$ such that each point in $\Omega$ is the vertex of a finite cone contained in
$\Omega$ and congruent to $\Lambda$.
An open set  $\Omega$ is called a \textit{Lipschitz domain} if each point of $\partial \Omega$ has a neighborhood
$\mathcal U $ such that $\Omega \cap \mathcal U$ is  the
subgraph of a Lipschitz continuous function of $n-1$ variables.

 Let $ m \in \N$
and let $X(\Omega )$ be a rearrangement-invariant space. We define
the $m$-th order Sobolev type space
 $W^m X(\Omega )$ as
\begin{equation}\label{E:W}
W^m X(\Omega ) =   \big\{u: \hbox{$u$ is $m$-times weakly
differentiable in $\Omega$, and $|\nabla ^k u| \in X(\Omega )$ for
$k=0, \dots , m$}\big\},
\end{equation}
equipped with the norm
$$\|u\|_{W^m X(\Omega )} = \sum _{k=0}^{m} \|\nabla
^k u\|_{X(\Omega  )}.$$ Here, $\nabla ^k u$ denotes the vector of
all $k$-th order weak derivatives of $u$, and $\|\nabla
^k u\|_{X(\Omega  )}$ is an abridged notation for $\| |\nabla
^k u | \|_{X(\Omega  )}$. In particular, $\nabla ^0
u$ stands for $u$, and $\nabla ^1 u$ will also be simply denoted by
$\nabla u$.  We shall also denote by $D^mu$ the vector whose components are the union of the components of $\nabla ^k u$ with $k=0, \dots , m$.
\\ We define the subspace $W^m_0X(\o)$ of $W^mX(\o)$ as the collection of those functions from $W^mX(\o)$ whose continuation to $\rn$  by $0$ outside $\o$ is an $m$-times weakly differentiable function in $\rn$.
The spaces $W^m X(\Omega )$ and $W^m_0 X(\Omega )$ are  Banach spaces. If $|\Omega|<\infty$, then, thanks to a general form of the Poincar\'e  inequality for rearrangement-invariant spaces   \cite[Lemma 4.2]{CP1998}, the space $W^m_0X(\o)$ can also be equivalently normed  by the functional $ \|\nabla\sp m u\|_{X(\Omega)}$ for $u\in W^m_0 X(\Omega )$.
Thus,
\begin{equation}\label{gen21}
\|u\|_{W^m X(\Omega )} \approx  \|D
^m u\|_{X(\Omega  )} \approx  \|\nabla
^m u\|_{X(\Omega  )}
\end{equation}
for $u \in W^m_0 X(\Omega )$.
Observe that, still under the assumption that $|\Omega|<\infty$ and thanks to the second embedding in \eqref{l1linf}, the continuation to $\rn$  by $0$ outside $\o$  of any function in $W^m_0 X(\Omega )$ belongs, in particular, to $W^{m,1}(\rn)$.
\\
By $C^\infty_0(\Omega)$  we denote the space of all infinitely times differentiable functions with compact support in $\Omega$, and by $C(\Omega)$ and $C(\overline \Omega)$ the spaces of continuous functions in $\Omega$ and $\overline \Omega$, respectively. Observe that the closure of $C^\infty_0(\Omega)$ in $W^m X(\Omega )$ is contained in $W^m_0 X(\Omega )$. The reverse inclusion holds under suitable  assumptions on the domain $\Omega$ and on the   norm in  $X(\o)$.
\par
As far as the K-functional for pairs of Sobolev spaces is
concerned,   if $\Omega$ is a bounded Lipschitz domain, then by
the result of \cite{DS},
the reiteration
theorem~\cite[p.~311]{BS} and Holmstedt's formulas one has what follows.  If either
$p_0=q_0=1$, or  $1< p_0<p_1<\infty$ and $1\leq q_0,q_1<\infty$,
then
\begin{align}\label{E:fx2.14}
  K(t,u;W\sp{m}L\sp{p_0,q_0}(\Omega ),W\sp{m}L\sp{p_1,q_1}(\Omega ))
  \approx
  &\left(\int_{0}\sp{t\sp\alpha}\left[s\sp{\frac1{p_0}-\frac1{q_0}}
  |D\sp mu|\sp*(s)\right]\sp{q_0}\,ds\right)\sp{\frac1{q_0}}\\
  &+
  t\left(\int_{t\sp\alpha}\sp{\infty}\left[s\sp{\frac1{p_1}-\frac1{q_1}}
  |D\sp mu|\sp*(s)\right]\sp{q_1}\,ds\right)\sp{\frac1{q_1}}\quad \text{for $t\in (0,\infty)$},\nonumber
\end{align}
 up to multiplicative constants depending on $\Omega$ and on $p_0, p_1, q_0, q_1$,
where
$
\frac1{\alpha}=\frac{1}{p_0}-\frac{1}{p_1}$.
Furthermore, if either $p_0=q_0=1$, or $1<p_0<\infty$ and $1\leq
q_0<\infty$, then
\begin{equation}\label{gen41}
  K(t,u;W\sp{m}L\sp{p_0,q_0}(\Omega ),W\sp{m}L\sp{\infty}(\Omega ))
  \approx
  \left(\int_{0}^{t\sp{p_0}}
 \left[s\sp{\frac{1}{p_0}-\frac{1}{q_0}}|D\sp mu|\sp*(s)\right]\sp{q_0}\,ds\right)\sp{\frac1{q_0}}\quad \text{for $t\in (0,\infty)$,}
\end{equation}
 up to multiplicative constants depending on $p_0$ and  $q_0$.
\\ Equations \eqref{E:fx2.14}  and \eqref{gen41} continue to hold in any open set $\o \subset \rn$, with $|\o|< \infty$, provided that $W^m $ is replaced by  $\ W^m_0$ in all occurences on their left-hand sides.

\section{Endpoint embeddings}\label{boder}

This section is devoted to Sobolev embeddings with measures
in borderline situations. These correspond to the weakest possible domain Sobolev norms for trace operators, with respect to the involved measures, to exist. In their turn, the endpoint  embeddings that will be established provide us with the strongest possible target norms. Apart from their own interest, they are pivotal for the development of our method.
\par
Assume that $\mu$ is a $d$-Frostman measure on $\o$, namely a finite Borel measure satisfying condition \eqref{E:sup} for some $d\in (0, n]$.
Let  $\|\cdot\|_{X(0,1)}$ and $\|\cdot\|_{Y(0,1)}$ be rearrangement-invariant function norms, and let $m \in \mathbb N$. We say that  $$T_\mu : W^mX(\o) \to Y(\o , \mu)$$ is a trace operator if it is a linear bounded operator such that $T_\mu u = u$ in $\o$ whenever $u\in W^mX(\o) \cap C(\o)$. An analogous definition applies if $W^mX(\o)$ is replaced by $W^m_0X(\o)$.
\\ Similarly, if $\mu$ is a $d$-Frostman measure on $\overline \o$, namely   a Borel measure satisfying condition \eqref{E:supbar},   we say that  $$T_\mu : W^m X(\o) \to Y(\overline \o , \mu)$$ is a trace operator if it is a linear bounded operator such that $T_\mu u = u$ in $\overline \o$ whenever $u\in W^mX(\o) \cap C(\overline \o)$.
\par In what follows,  a trace embedding
\begin{equation}\label{gen20}
W^mX(\o) \to Y(\o , \mu),
\end{equation}
or, equivalently, a trace inequality
\begin{equation}\label{gen8}\|u\|_{Y(\o , \mu)}\leq C \|u\|_{W^mX(\o)}
\end{equation}
has to be interpreted in the sense that there exists a trace operator $T_\mu : W^mX(\o) \to Y(\o , \mu)$ as defined above, and that inequality \eqref{gen8} holds, where  $u$ is a simplified notation for $T_\mu u$ on the  left-hand side. Embeddings and inequalities involving $W^m_0X(\o)$ or $Y(\overline \o , \mu)$ have to be understood analogously. The space $Y(\o , \mu)$ is said to be optimal in embedding \eqref{gen20} within a certain family of spaces if, whenever embedding \eqref{gen20} holds with $Y(\o , \mu)$ replaced by another space $Z(\o , \mu)$ from the same family, then $Y(\o , \mu) \to Z(\o , \mu)$. Equivalently, the norm $\|\cdot \|_{Y(\o , \mu)}$, or the function norm $\|\cdot \|_{Y(0, L)}$,  is said to be optimal in inequality \eqref{gen8}.

\medskip
\par
We begin our discussion by noticing that the only non-trivial case in view of our purposes is when
\begin{equation}\label{m<n}
m<n.
\end{equation}
 Indeed, if $m \geq n$, then classically $W^{m,1}_0(\Omega) \to C(\Omega)$. Hence, if $\Omega$ has finite Lebesgue measure, then, by \eqref{l1linf},  $W^m_0X(\o) \to W^{m,1}_0(\Omega) \to Y(\o , \mu)$   for every rearrangement-invariant norms $\|\cdot\|_{X(0,1)}$ and $\|\cdot\|_{Y(0,1)}$ and any Frostman measure $\mu$. The situation is analogous, under the  regularity assumptions to be imposed on $\o$, if
 trace embeddings involving $W^mX(\o)$ or  $Y({\overline \Omega},\mu)$ are considered.  We shall thus  restrict our attention to the case when \eqref{m<n} is in force.

As already mentioned in Section \ref{intro},  the  threshold $n-m$ for the exponent $d$ naturally arises when dealing with $m$-th order Sobolev trace embeddings for $d$-Frostman measures.  The corresponding theory  indeed exhibits different features depending on whether $d$ exceeds this threshold or not. Our analysis of the pertaining borderline inequalities is accordingly split in two subsections.

\subsection{Case $n-m \leq d \leq n$}\label{borderfast}

 Here, we assume that
\begin{equation}\label{super}
d \in [n-m,  n]\,.
\end{equation}
Under  condition \eqref{super}, a trace operator is classically well defined on  $W^{m,1}_0(\Omega)$ if $|\Omega|<\infty$.
Hence, it is well defined in any  Sobolev space $W^m_0X(\Omega)$, whatever the rearrangement-invariant space $X(\Omega)$ is, since, by \eqref{l1linf}, $W^m_0X(\Omega) \to W^{m,1}_0(\Omega)$.
The situation is similar, under the regularity assumptions that will be imposed  on $\o$, if
 trace embeddings of $W^mX(\o)$ on $\o$ or $\overline \o$ are in question.

Our first result provides us with an optimal rearrangement-invariant target space  for trace embeddings of $W^{m,1}_0(\Omega)$ or $W^{m,1}(\Omega)$.

\begin{theorem}{\rm{\bf [Case $n-m \leq d \leq  n$: optimal trace embedding for $W^{m,1}(\Omega)$]}}\label{P:endpoint}
Let $\Omega$ be an open set with finite Lebesgue measure in $\rn$, $n\geq 2$, let $m\in \N$, with $m<n$, and let $d \in [n-m, n]$.
\\
\textup{(}i\textup{)} Let
$\mu$ be a finite Borel measure on $\Omega$ fulfilling \eqref{E:sup}. Then
\begin{equation}\label{E:endpoint_estimate0}
\|u\|_{L^{\frac{d}{n-m},1}(\Omega,\mu)} \leq C \|\nabla ^m u\|_{L^1(\Omega)}
\end{equation}
for some constant $C$ and every $u\in W^{m,1}_0(\Omega)$.
\\
\textup{(}ii\textup{)}  Assume, in addition, that  $\Omega$ is  bounded  and has the cone property. Let
$\mu$ be a Borel measure on $\o$ fulfilling \eqref{E:sup}. Then
\begin{equation}\label{E:endpoint_estimate}
\|u\|_{L^{\frac{d}{n-m},1}(\Omega,\mu)} \leq C \|u\|_{W^{m,1}(\Omega)}
\end{equation}
for some constant $C$ and every $u\in W^{m,1}(\Omega)$.
\\
\textup{(}iii\textup{)} Assume  that $\Omega$ is a bounded Lipschitz domain.  Let
$\mu$ be a Borel measure on $\overline \o$ fulfilling \eqref{E:supbar}. Then
\begin{equation}\label{E:endpoint_estimatebar}
\|u\|_{L^{\frac{d}{n-m},1}(\overline{\Omega},\mu)} \leq C \|u\|_{W^{m,1}(\Omega)}
\end{equation}
for some constant $C$ and every $u\in W^{m,1}(\Omega)$.
\end{theorem}

The result of Theorem \ref{P:endpoint} is sharp. Indeed, assume that  the measure $\mu$ decays exactly like $r^d$ on balls of radius $r$, at least around one point of $\o$, in the sense that there exist $x_0\in \o$  and $R>0$ such that
\begin{equation}\label{E:inf}
\inf_{r\in (0,R]} \frac{\mu(B_r(x_0)\cap {\Omega})}{r^d} >0\,.
\end{equation}
Then
the Lorentz norm $\|\cdot\|_{L^{\frac{d}{n-m},1}(\Omega,\mu)}$ on the left-hand sides of inequalities \eqref{E:endpoint_estimate0}--\eqref{E:endpoint_estimate} is the strongest admissible among all rearrangement-invariant norms. The same conclusion holds with regard to the norm $\|\cdot\|_{L^{\frac{d}{n-m},1}(\overline \Omega,\mu)}$ in inequality \eqref{E:endpoint_estimatebar}, provided that
\begin{equation}\label{E:infbar}
\inf_{r\in (0,R]} \frac{\mu(B_r(x_0)\cap {\overline \Omega})}{r^d} >0\,
\end{equation}
for some $x_0\in \overline \o$  and $R>0$. This is the content of the next proposition.

\begin{prop}\label{sharpness fast}{\rm{\bf [Sharpness of Theorem \ref{P:endpoint}]}}
Let $\Omega$ be an open set  in $\rn$, with $n\geq 2$, as in the respective parts \textup{(}i\textup{)}, \textup{(}ii\textup{)} and \textup{(}iii\textup{)} of Theorem~\ref{P:endpoint}. Let  $m  \in \N$ be such that $m<n$, and let $d \in [n-m, n]$.
\\
\textup{(}i\textup{)} Let $\mu $ be a finite Borel measure on $\Omega$ fulfilling conditions \eqref{E:sup} and \eqref{E:inf}. Then the target norm $\|\cdot\|_{L^{\frac{d}{n-m},1}(\Omega,\mu)}$ is optimal in  inequalities \eqref{E:endpoint_estimate0} and \eqref{E:endpoint_estimate}.
\\
\textup{(}ii\textup{)}
 Let $\mu $ be a finite Borel measure on $\overline \Omega$ fulfilling conditions \eqref{E:supbar} and \eqref{E:infbar}. Then the target norm $\|\cdot\|_{L^{\frac{d}{n-m},1}(\overline \Omega,\mu)}$ is optimal in  inequality   \eqref{E:endpoint_estimatebar}.
\end{prop}

Proposition \ref{sharpness fast} is a special case of Theorem~\ref{T:main-example}, which is stated and proved in Section \ref{S:last}.

\begin{proof}[Proof of Theorem \ref{P:endpoint}]
Consider part~(i). Set
\begin{equation}\label{normmu}
\|\mu\|_d=\sup_{x\in \Rn, r>0} \frac{\mu(B_r(x)\cap {\Omega})}{r^d}.
\end{equation}
Let $\nu$ be the Borel measure on $\rn$ defined as
\begin{equation}
\label{extmeas}
\nu (E)= \mu (E \cap \Omega)
\end{equation}
for every Borel set $E\subset \rn$. We shall prove that
\begin{equation}\label{E:whole_rn}
\|u\|_{L^{\frac{d}{n-m},1}(\Rn,\nu)} \lesssim \|\mu\|_d^{\frac{n-m}{d}} \int_{\Rn} |\nabla^m u|\,dx
\end{equation}
for every function $u\in C^\infty_0(\Rn)$.  The constants in the relations ``$\lesssim$'' and ``$\approx$'' in this proof depend only on $m,n,d,\Omega$ and $\mu$.
\\ In order to establish inequality \eqref{E:whole_rn}, we distinguish between the cases when $m=1$ and $m>1$.
First, assume that $m=1$.  By \cite[Theorem 1.4.3]{Mabook},
$$
\|u\|_{L^{\frac{d}{n-1}}(\Rn,\nu)} \lesssim \|\mu\|_d^{\frac{n-1}{d}} \int_{\Rn} |\nabla u|\,dx
$$
for every function $u \in W^{1,1}(\rn)$.
Since $L^{\frac{d}{n-1}}(\Rn, \nu) \rightarrow L^{\frac{d}{n-1},\infty}(\Rn,\nu)$,
one as well has that
\begin{equation}\label{E:weak_estimate}
\|u\|_{L^{\frac{d}{n-1},\infty}(\Rn,\nu)} \lesssim \|\mu\|_d^{\frac{n-1}{d}} \int_{\Rn} |\nabla u|\,dx
\end{equation}
for every function $u \in W^{1,1}(\rn)$.  A truncation method introduced by Maz'ya (see also~\cite{MP} for applications to Lorentz-type norms) can be used to show that inequality \eqref{E:weak_estimate} implies, in fact, the stronger estimate
\begin{equation}\label{E:stronger_estimate}
\|u\|_{L^{\frac{d}{n-1},1}(\Rn,\nu)} \lesssim \|\mu\|_d^{\frac{n-1}{d}} \int_{\Rn} |\nabla u|\, dx
\end{equation}
for every function $u \in W^{1,1}(\rn)$. In order to prove inequality~\eqref{E:stronger_estimate}, note that inequality~\eqref{E:weak_estimate} can be rewritten in the form
\begin{equation}\label{E:weak_type_reformulated}
\sup_{\varrho >0}\varrho \,\nu(\{x\in \Rn: |u(x)|\geq \varrho\})^{\frac{n-1}{d}}
\lesssim \|\mu\|_d^{\frac{n-1}{d}} \int_{\Rn} |\nabla u|\,dx.
\end{equation}
Fix $u\in C^\infty_0(\Rn)$, and set
$$
t_k=2^{1-k}\nu (\operatorname{supp} u) \quad \hbox{and} \quad a_k=u^*_\nu(t_k) \quad \hbox{for $k\in \N$,}
$$
where $\lq\lq \operatorname{supp} u"$ stands for support of $u$.
Given $a$ and $b$, with $0<a<b<\infty$, let $\varphi_a^b\colon \R\to[0,\infty)$ be the function defined as
$$
\begin{cases}
\varphi_a^b(s)=0 &\hbox{if $s \leq a$,}\\
\varphi(s)=s-a  &\hbox{if $a < s < b$,}\\
\varphi_a^b(s)=b-a  &\hbox{if $ s\geq  b$.}
\end{cases}
$$
Inasmuch as $\varphi_a^b$ is Lipschitz continuous, standard properties of Sobolev functions ensure that $\varphi_a^b (|u|) \in W^{1,1}(\rn)$, and $|\nabla (\varphi_a^b (|u|) ) |= \chi _{\{ a<|u|<b\}} |\nabla u|$ a.e. in $\rn$.  Thus, an  application of inequality~\eqref{E:weak_type_reformulated} to the function $\varphi_a^b (|u|)$, with $\rho=b-a$, yields
\begin{multline}\label{E:truncation}
(b-a) \nu(\{x\in \Rn: |u(x)| \geq b\})^{\frac{n-1}{d}}
=(b-a) \nu(\{x\in \Rn: \varphi_a^b (|u(x)|)  \geq b-a\})^{\frac{n-1}{d}}
\\ \lesssim \|\mu\|_d^{\frac{n-1}{d}} \int_{\Rn} |\nabla (\varphi_a^b  |u|)|\,dx
\approx \|\mu\|_d^{\frac{n-1}{d}} \int_{\{x\in \Rn: a<|u(x)|<b\}} |\nabla u|\,dx.
\end{multline}
Since $t_k\leq \nu(\{x\in \Rn: |u(x)|\geq a_k\})$, inequality~\eqref{E:truncation}, applied with $a=a_k$ and $b=a_{k+1}$, in its turn tells us that
\begin{equation}\label{gen10}
t_{k+1}^{\frac{n-1}{d}} (a_{k+1}-a_k)
\lesssim \|\mu\|_d^{\frac{n-1}{d}} \int_{\{x\in \Rn: a_k<|u(x)|<a_{k+1}\}} |\nabla u|\,dx
\end{equation}
for $k \in \N$.
On adding inequalities \eqref{gen10} as $k$ ranges in $\N$, we obtain that
$$
\sum_{k=1}^\infty t_{k+1}^{\frac{n-1}{d}} (a_{k+1}-a_k)
\lesssim \|\mu\|_d^{\frac{n-1}{d}} \sum_{k=1}^\infty \int_{\{x\in \Rn: a_k<|u(x)|<a_{k+1}\}} |\nabla u|\, dx
\lesssim \|\mu\|_d^{\frac{n-1}{d}} \int_{\Rn} |\nabla u|\,dx.
$$
Since $a_1=0$, the latter inequality can be rewritten as
$$
\sum_{k=2}^\infty t_k^{\frac{n-1}{d}} a_k - 2^{\frac{1-n}{d}}\sum_{k=1}^\infty t_k^{\frac{n-1}{d}} a_k
=\left(1-2^{\frac{1-n}{d}}\right) \sum_{k=1}^\infty t_k^{\frac{n-1}{d}} a_k
\leq C \|\mu\|_d^{\frac{n-1}{d}} \int_{\Rn} |\nabla u|\,dx.
$$
Thus,
\begin{align*}
  \int_0^{\nu (\operatorname{supp} u)} u_\nu^{*}(t)t^{\frac{n-1}{d}-1}\,dt
  & = \sum_{k=1}^\infty \int_{t_{k+1}}^{t_{k}} u_\nu^{*}(t)t^{\frac{n-1}{d}-1}\,dt
          \leq \sum_{k=1}^\infty u_\nu^{*}(t_{k+1})\int_{t_{k+1}}^{t_{k}} t^{\frac{n-1}{d}-1}\,dt
           \\
  & = \frac{d}{n-1} \sum_{k=1}^\infty u_\nu^{*}(t_{k+1})\Big(t_k^{\frac{n-1}{d}}-t_{k+1}^{\frac{n-1}{d}}\Big)
          \leq \frac{d}{n-1} \sum_{k=1}^\infty u_\nu^{*}(t_{k+1})t_k^{\frac{n-1}{d}}
           \\
  & \leq \frac{d \, 2^{\frac{n-1}{d}}}{n-1} \sum_{k=1}^\infty u_\nu^{*}(t_{k+1})t_{k+1}^{\frac{n-1}{d}}
    \leq \frac{C\, d \, 2^{\frac{n-1}{d}}}{(n-1) (1-2^{\frac{1-n}{d}})} \|\mu\|_d^{\frac{n-1}{d}} \int_{\Rn} |\nabla u|\,dx,
\end{align*}
and~\eqref{E:stronger_estimate} follows.
\\
Let us next assume that $m>1$. For any function $u\in C^\infty_0(\Rn)$, one has that
\begin{equation}\label{gen60}
|u(x)|\lesssim \int_{\Rn} \frac{|\nabla^{m-1} u(y)|}{|x-y|^{n-m+1}}\,dy\quad \text{for a.e. $x\in \Rn$},
\end{equation}
see e.g. \cite[Theorem~1.1.10/2]{Mabook}.
By equation \eqref{E:lorentz-assoc} and Fubini's theorem,
\begin{align*}
\|u\|_{L^{\frac{d}{n-m},1}(\Rn,\nu)}
&\lesssim \left\|\int_{\Rn} \frac{|\nabla^{m-1} u(y)|}{|x-y|^{n-m+1}}\,dy \right\|_{L^{\frac{d}{n-m},1}(\Rn,\nu)}
\approx \sup_{v \not\equiv 0} \frac{\int_{\Rn} |v(x)| \int_{\Rn} \frac{|\nabla^{m-1} u(y)|}{|x-y|^{n-m+1}}\,dy \,d\nu(x)}{\|v\|_{L^{\frac{d}{d-n+m},\infty}(\Rn,\nu)}}\\
&= \sup_{v \not\equiv 0} \frac{\int_{\Rn} |\nabla^{m-1} u(y)| \int_{\Rn} \frac{|v(x)|}{|x-y|^{n-m+1}}\,d\nu(x) \,dy}{\|v\|_{L^{\frac{d}{d-n+m},\infty}(\Rn,\nu)}} \approx \sup_{v \not\equiv 0} \frac{\int_{\Rn} |\nabla^{m-1} u(y)| w(y)\,dy}{\|v\|_{L^{\frac{d}{d-n+m},\infty}(\Rn,\nu)}},
\end{align*}
\textit{\(\)}where $w : \rn \to [0, \infty)$ is the function defined as $w(y)=\int_{\Rn} \frac{|v(x)|}{|x-y|^{n-m+1}}\,d\nu(x)$ for $y \in \rn$. In order to conclude,  it remains to prove that
$$
\| \nabla^{m-1} u \|_{L^1(\Rn,w)}
\lesssim \|\mu\|_d^{\frac{n-m}{d}} \|v\|_{L^{\frac{d}{d-n+m},\infty}(\Rn,\nu)} \int_{\Rn} |\nabla^m u|\, dx
$$
for every function $v \in L^{\frac{d}{d-n+m},\infty}(\Rn,\nu)$.
By the result of the present  theorem for $m=1$, $d=n-1$ and $d\mu(x)=w(x)dx$, which has been proved above,   it suffices to show that
$$
\sup_{z\in \Rn, r>0} \frac{1}{r^{n-1}}\int_{B(z,r)}w(x)\,dx
\lesssim \|\mu\|_d^{\frac{n-m}{d}} \|v\|_{L^{\frac{d}{d-n+m},\infty}(\Rn,\nu)}
$$
for every function $v \in L^{\frac{d}{d-n+m},\infty}(\Rn,\nu)$. Here, $B(z,r)$ is an alternate notation for $B_r(z)$.
Fix $z\in \Rn$ and $r>0$.  By Fubini's theorem and equation \eqref{E:lorentz-assoc},
\begin{align*}
\int_{B(z,r)}w(x)\,dx
&=\int_{B(z,r)} \int_{\Rn} \frac{|v(x)|}{|x-y|^{n-m+1}}\,d\nu(x)\,dy
=\int_{\Rn} |v(x)| \int_{B(z,r)} \frac{\,dy}{|x-y|^{n-m+1}}\,d\nu(x)\\
&\leq \|v\|_{L^{\frac{d}{d-n+m},\infty}(\Rn,\nu)} \left\|\int_{B(z,r)} \frac{\,dy}{|x-y|^{n-m+1}} \right\|_{L^{\frac{d}{n-m},1}(\Rn,\nu)}.
\end{align*}
Now,
\begin{align*}
&\left\|\int_{B(z,r)} \frac{\,dy}{|x-y|^{n-m+1}} \right\|_{L^{\frac{d}{n-m},1}(\Rn,\nu)}\\
&\leq \left\|\chi_{\{x\in \Rn: |x-z|\leq 2r\}} \int_{B(z,r)} \frac{\,dy}{|x-y|^{n-m+1}} \right\|_{L^{\frac{d}{n-m},1}(\Rn,\nu)} +\left\|\chi_{\{x\in \Rn: |x-z|> 2r\}} \int_{B(z,r)} \frac{\,dy}{|x-y|^{n-m+1}} \right\|_{L^{\frac{d}{n-m},1}(\Rn,\nu)}\\
&\lesssim r^{m-1} \big\|\chi_{\{x\in \Rn: |x-z|\leq 2r\}}\big\|_{L^{\frac{d}{n-m},1}(\Rn,\nu)} + \left\|\chi_{\{x\in \Rn: |x-z|> 2r\}} \frac{r^n}{|x-z|^{n-m+1}} \right\|_{L^{\frac{d}{n-m},1}(\Rn,\nu)}.
\end{align*}
Note that  we have made use of inequality \eqref{E:HL} in estimating the first norm in the last inequality, and of the fact that  $|x-y|\approx |x-z|$ if  $|x-z|>2r$ and  $y\in B(z,r)$
in estimating the second norm in the same inequality. One has that
$$
r^{m-1} \|\chi_{\{x\in \Rn: |x-z|\leq 2r\}}\|_{L^{\frac{d}{n-m},1}(\Rn,\nu)}\lesssim r^{m-1} \nu(B(z,2r))^{\frac{n-m}{d}}
\approx r^{m-1} (\mu(B(z,2r)\cap \Omega))^{\frac{n-m}{d}}
\lesssim r^{n-1} \|\mu\|_d^{\frac{n-m}{d}}.
$$
Moreover,
\begin{align*}
&\left\|\chi_{\{x\in \Rn: |x-z|> 2r\}} \frac{r^n}{|x-z|^{n-m+1}} \right\|_{L^{\frac{d}{n-m},1}(\Rn,\nu)}
\\
&\approx r^n \int_0^\infty \nu(\{x\in \Rn: 2r<|x-z|<\varrho^{-\frac{1}{n-m+1}}\})^{\frac{n-m}{d}}\,d\varrho\\
&\leq  r^n \int_0^{\frac{1}{(2r)^{n-m+1}}} \nu(B(z,\varrho^{-\frac{1}{n-m+1}}))^{\frac{n-m}{d}}\,d\varrho = r^n \int_0^{\frac{1}{(2r)^{n-m+1}}} \mu(B(z,\varrho^{-\frac{1}{n-m+1}})\cap \Omega)^{\frac{n-m}{d}}\,d\varrho\\
&\lesssim r^n \|\mu\|_d^{\frac{n-m}{d}} \int_0^{\frac{1}{(2r)^{n-m+1}}} \varrho^{-\frac{n-m}{n-m+1}}\,d\varrho\lesssim r^{n-1} \|\mu\|_d^{\frac{n-m}{d}}.
\end{align*}
Inequality~\eqref{E:whole_rn} is thus established for every function $u\in C^\infty_0(\rn)$.
Now, let  $u \in W^{m,1}_0(\o)$. As observed in Section \ref{back}, the continuation $\overline u$ of $u$ to $\rn$ by $0$ outside $\Omega$ belongs to $W^{m,1}(\rn)$. Thus, the function $\overline u$   can be approximated in $W^{m,1}(\rn)$ by a sequence
of functions $\{\overline u_k\} \subset C^\infty _0(\rn)$ in such a way that $\overline u_k \to \overline u$ at every Lebesgue point of $\overline u$. In particular, $\overline u_k \to u$ at every point where $u$ is continuous, and hence everywhere in $\o$ if $u$ is continuous in $\Omega$. An application of inequality \eqref{E:whole_rn}, with $u$ replaced by $\overline u_k-\overline u_m$,  ensures that $\{\overline u_k\}$ is a Cauchy sequence in
$L^{\frac{d}{n-m},1}(\Omega,\mu)$. One can define $T_\mu u$ as the limit of this sequence in $L^{\frac{d}{n-m},1}(\Omega,\mu)$, since such a limit is easily seen to be independent of the approximating sequence for $\overline u$. Also, $T_\mu$ turns out to be a bounded linear operator  on $W^{m,1}_0(\o)$, which agrees with the identity on $W^{m,1}_0(\o)\cap C(\o)$. The proof of part (i) is complete.
\\
Let us  now deal with part (ii). Any bounded domain $\o$ with the cone property can be decomposed into a finite union of bounded Lipschitz domains $\{\o_j\}$, with $j=1, \dots , J$ (see, e.g.,~\cite[Lemma~1.1.1/9]{Mabook}). Thus, each open set  $\o_j$ is an extension domain, in the sense of \cite[Chapter 6, Section~3.3]{stein}, from $W^{m,1}(\o_j)$ into $W^{m,1}_0(B_j)$ for a suitable open ball $B_j$ such that $\overline \o_j \subset B_j$ .
 This implies that, for each $j$, there exists a bounded linear operator $\mathcal E_j : W^{m,1}(\Omega_j) \to W^{m,1}_0(B_j)$ such that $\mathcal E_j u=u$ in $\Omega_j$ for every $u\in W^{m,1}(\Omega_j)$. By part (i), if $u\in W^{m,1}(\o)\cap C^\infty (\o)$, then
\begin{align}
\label{july1}
\|u\|_{L^{\frac{d}{n-m},1}(\Omega,\mu)}
    & \leq \sum _{j=1}^J\|u\|_{L^{\frac{d}{n-m},1}(\Omega_j,\mu)} = \sum _{j=1}^J \|T_{\nu}(\mathcal E_j u)\|_{L^{\frac{d}{n-m},1}(\Omega_j,\nu)} \leq \sum _{j=1}^J \|T_{\nu}(\mathcal E_j u)\|_{L^{\frac{d}{n-m},1}(B_j,\nu)}
            \\
    \nonumber & \lesssim  \sum _{j=1}^J\|\nabla ^m\mathcal E_j u\|_{L^{1}(B_j)}
\lesssim \sum _{j=1}^J\|u\|_{W^{m,1}(\o_j)} \lesssim\|u\|_{W^{m,1}(\o)}.
\end{align}
On the other hand, every function $u \in W^{m,1}(\o)$ can be approximated by a sequence $\{u_k\}\subset  W^{m,1}(\o)\cap C^\infty (\o)$ in such a way that, in addition, $u_k\to u$ at every Lebesgue point of $u$. The conclusion now follows as at the end of the proof of part (i).
\\
Finally, consider part (iii). If $\o$ is a bounded Lipschitz domain, then it is an extension domain in the sense specified above. Let $\mathcal E : W^{m,1}(\o) \to W^{m,1}_0(B)$ be an extension operator, for a suitable ball $B \supset \overline \o$. Denote by $\nu$ the Borel measure defined in $B$ as
$$\nu (E) = \mu (E\cap \overline \o)$$
for every Borel set $E\subset B$. The measure $\nu$ satisfies condition \eqref{E:supbar}. Suppose, for the time being,  that $u \in W^{m,1}(\Omega)\cap C(\overline \o)$. Inasmuch as $|\partial \Omega|=0$,   we may assume that $\mathcal Eu=u$ in $\overline \o$, whence, in particular, $\mathcal Eu \in C(\overline \o)$.
Therefore, by part (i),
\begin{align}
\label{july2}
\| u\|_{L^{\frac{d}{n-m},1}(\overline \o,\mu)}  = \| T_{\nu}(\mathcal E u)\|_{L^{\frac{d}{n-m},1}(\overline \o,\nu)} = \|T_{\nu}(\mathcal E u)\|_{L^{\frac{d}{n-m},1}(B,\nu)}  \lesssim\|\nabla ^m \mathcal E u\|_{L^1(B)} \lesssim \|u\|_{W^{m,1}(\o)}.
\end{align}
Next, since $\Omega$ is a bounded Lipschitz domain (a \lq\lq continuous domain" would in fact suffice), given any  function $u \in W^{m,1}(\o)$, by \cite[Theorem~1.1.6/2]{Mabook}, there exists a sequence of functions $\{u_k\}\subset C^\infty_0(\rn)$ such that $u_k \to u$ in $W^{m,1}(\o)$. An application of inequality \eqref{july2} with $u$ replaced by $u_k - u_m$, for $k, m \in \mathbb N$, tells us that $\{u_k\}$, restricted to $\overline \o$, is a Cauchy sequence in $L^{\frac{d}{n-m},1}(\overline \o,\mu)$. One can define $T_\mu u$ as the limit of this sequence, whence inequality \eqref{july2}, namely \eqref{E:endpoint_estimatebar}, holds with  $u$ replaced by $T_\mu u$.
It is easily verified, on making use of inequality \eqref{july2} again, that this limit is independent of the sequence $\{u_k\}$ approximating $u$. Moreover, the operator $T_\mu$ is linear and bounded. In particular, if $u \in C(\overline \o)$, then the sequence $\{u_k\}\subset C^\infty_0(\rn)$ that approximates $u$ in $W^{m,1}(\o)$ can be chosen in such a way that $u_k \to u$ pointwise in $\overline \o$ (see e.g. the proof of \cite[Theorem~1.1.6/2]{Mabook}). Hence, $T_\mu u = u$ in $\overline \Omega$ whenever $u \in  W^{m,1}(\o)\cap C(\overline \o)$.
\end{proof}

\subsection{Case $0 < d < n-m$}\label{borderslow}

In this subsection, we assume that
\begin{equation}\label{sub}
 d \in (0, n-m).
\end{equation}
The situation is now substantially different from that discussed in Subsection \ref{borderfast},
since, under \eqref{sub}, a trace operator with respect to a $d$-Frostman measure  need not even be defined in  $W^m_0X(\Omega)$ or  $W^mX(\Omega)$.
The existence of such an operator
is guaranteed if  the space $X(\Omega)$ is included in  $L^{\frac{n-d}m,1}(\Omega)$, and hence $W^m_0X(\Omega)\to W^m_0L^{\frac{n-d}m,1}(\Omega)$ and $W^mX(\Omega)\to W^mL^{\frac{n-d}m,1}(\Omega)$. The optimality of the space
$L^{\frac{n-d}m,1}(\Omega)$, among all rearrangement-invariant spaces, for a trace to be well defined, is shown in \cite{CP2010} when $d \in \N$ and the measure is $\mu_2$ as defined in \eqref{mu2}.
\par
A precise trace embedding for the space $W^m_0L^{\frac{n-d}m,1}(\Omega)$, or $W^mL^{\frac{n-d}m,1}(\Omega)$, follows from
a recent result of \cite{KK}   (which was earlier proved  in~\cite{CP-Trans} in the case of the measure $\mu_2$), and is stated in the next theorem.

\begin{theorem} {\rm{\bf [Case $0< d <n-m$: optimal trace embedding for $W\sp{m}L\sp{\frac{n-d}{m},1}(\Omega)$]}}  \label{L:endpointsub}
Let $\Omega$ be an open set with finite Lebesgue measure in $\rn$, $n\geq 2$, let $m\in \N$, with $m<n$, and let $d \in (0, n-m)$.
\\
\textup{(}i\textup{)} Let
$\mu$ be a finite Borel measure on $\Omega$ fulfilling \eqref{E:sup}. Then
\begin{equation}\label{E:endpointsub_estimate0}
\|u\|_{L\sp{\frac{n-d}{m}}(\Omega,\mu)}
\leq C
\|\nabla ^m u\|_{L\sp{\frac{n-d}{m},1}(\Omega)}
\end{equation}
for some constant $C$ and every $u\in W\sp{m}_0L\sp{\frac{n-d}{m},1}(\Omega)$.
\\
\textup{(}ii\textup{)} Assume, in addition,  that $\Omega$ is  bounded   and has  the cone property.  Let
$\mu$ be a Borel measure on $\o$ fulfilling \eqref{E:sup}. Then
\begin{equation}\label{E:endpointsub_estimate}
\|u\|_{L\sp{\frac{n-d}{m}}(\Omega,\mu)}
\leq C
\|u\|_{W\sp{m}L\sp{\frac{n-d}{m},1}(\Omega)}
\end{equation}
for some constant $C$ and every $u\in W\sp{m}L\sp{\frac{n-d}{m},1}(\Omega)$.
\\
\textup{(}iii\textup{)} Assume that $\Omega$ is  a bounded Lipschitz domain.
Let
$\mu$ be a Borel measure on $\overline \o$ fulfilling \eqref{E:supbar}. Then
\begin{equation}\label{E:endpointsub_estimatebar}
\|u\|_{L\sp{\frac{n-d}{m}}(\overline{\Omega},\mu)}
\leq C
\|u\|_{W\sp{m}L\sp{\frac{n-d}{m},1}(\Omega)}
\end{equation}
for some constant $C$ and every $u\in W\sp{m}L\sp{\frac{n-d}{m},1}(\Omega)$.
\end{theorem}

Our next aim is to show that the range spaces in Theorem~\ref{L:endpointsub} cannot be improved under the sole assumption \eqref{E:sup} or \eqref{E:supbar}.
 More precisely,  Proposition \ref{sharpness} below demonstrates that, at least when $d \in \N$, there exist  measures $\mu$  satisfying condition~\eqref{E:sup} for which  the target space $L\sp{\frac{n-d}{m}}(\Omega,\mu)$ 
 is optimal   among all rearrangement-invariant spaces. An example in this connection is provided by the measure   $\mu_2$   given by \eqref{mu2}.
Incidentally, let us point out    that  this conclusion is in sharp contrast with all other known results about optimal target spaces in Sobolev embeddings, where the optimal target space is always of Lorentz type, with second exponent different from the first one.

\begin{prop}\label{sharpness}{\rm{\bf [Sharpness of Theorem \ref{L:endpointsub}]}}
Let $\Omega$ be an open set  in $\rn$, $n\geq 2$, as in the respective parts \textup{(}i\textup{)} and \textup{(}ii\textup{)} of Theorem~\ref{L:endpointsub}.
Let  $m, d \in \N$ be such that $m<n$ and  $d \in (0, n-m)$. Let $\mu_2 $ be the measure defined by \eqref{mu2}. Then the target norm $\|\cdot\|_{L\sp{\frac{n-d}{m}}(\Omega,\mu_2)}$ in inequalities  \eqref{E:endpointsub_estimate0} and \eqref{E:endpointsub_estimate} is optimal among all rearrangement-invariant norms.
\end{prop}

Unlike the conclusions for $d\in [n-m, n]$ of Theorem \ref{P:endpoint}, which
are optimal for any measure $\mu$ satisfying \eqref{E:sup} and \eqref{E:inf}, or \eqref{E:supbar} and \eqref{E:infbar}, the results  for $d\in (0, n-m)$ of Theorem \ref{L:endpointsub}
can yet be augmented if extra information
 on $\mu$, besides \eqref{E:sup} and  \eqref{E:inf}, or \eqref{E:supbar} and \eqref{E:infbar},
 is available.
%
%
This is the case, for instance, when the measure $\mu$
 is absolutely continuous with respect to Lebesgue measure, with a radially decreasing density with respect to some point $x_0 \in \Omega$. Namely, when
\begin{equation}\label{E:mu-radial}
d\mu (x) = g(|x-x_0|) dx
\end{equation}
for some non-increasing function $g \colon (0, \infty) \to [0, \infty)$. Indeed, the
space $L^{\frac{n-d}m}$ can be replaced by the strictly smaller space $L^{\frac{n-d}m, 1}$
in inequalities \eqref{E:endpointsub_estimate0}--\eqref{E:endpointsub_estimatebar}.

\begin{prop}{\rm{\bf [Improved trace embedding for $W\sp{m}L\sp{\frac{n-d}{m},1}(\Omega)$ for radially decreasing densities]}}\label{P:radial-improvement}
Let $\Omega$ be an open set  with finite Lebesgue measure in $\Rn$,  $n\geq 2$, let $m\in \N$ with $m<n$, and let $d \in (0, n-m)$.
\\
\textup{(}i\textup{)} Let $\mu$ be a finite Borel measure on $\Omega$ fulfilling \eqref{E:sup} and having the form \eqref{E:mu-radial} for  some  $x_0\in\Omega$. Then
\begin{equation}\label{E:radial-improvement-0}
\|u\|_{L\sp{\frac{n-d}{m},1}(\Omega,\mu)}
\leq C
\|\nabla ^m u\|_{L\sp{\frac{n-d}{m},1}(\Omega)}
\end{equation}
for some constant $C$ and every $u\in W\sp{m}_0L\sp{\frac{n-d}{m},1}(\Omega)$.
\\
\textup{(}ii\textup{)} Assume, in addition, that  $\Omega$ is bounded and has  the cone property. Let $\mu$ be a Borel measure on $\o$ fulfilling \eqref{E:sup}
and having the form \eqref{E:mu-radial} for  some  $x_0\in\Omega$.
 Then
\begin{equation}\label{E:radial-improvement}
\|u\|_{L\sp{\frac{n-d}{m},1}(\Omega,\mu)}
\leq C
\|u\|_{W\sp{m}L\sp{\frac{n-d}{m},1}(\Omega)}
\end{equation}
for some constant $C$ and every $u\in W\sp{m}L\sp{\frac{n-d}{m},1}(\Omega)$.
\\
\textup{(}iii\textup{)}  Assume that $\Omega$ is a bounded Lipschitz domain. Let $\mu$ be a Borel measure on $\overline \o$ fulfilling \eqref{E:sup}
and having the form \eqref{E:mu-radial} for  some  $x_0\in \overline \Omega$.  Then
\begin{equation}\label{E:radial-improvement-bar}
\|u\|_{L\sp{\frac{n-d}{m},1}(\overline{\Omega},\mu)}
\leq C
\|u\|_{W\sp{m}L\sp{\frac{n-d}{m},1}(\Omega)}
\end{equation}
for some constant $C$ and every $u\in W\sp{m}L\sp{\frac{n-d}{m},1}(\Omega)$.
\\ Moreover, if, in addition, $\mu$ satisfies condition \eqref{E:inf} in cases \textup{(}i\textup{)} and \textup{(}ii\textup{)}, or \eqref{E:infbar} in case \textup{(}iii\textup{)}, then the norm  $\|\,\cdot \,\|_{L\sp{\frac{n-d}{m},1}}$ is optimal  in the pertaining inequality.
\end{prop}

A proof of Proposition~\ref{P:radial-improvement} will be presented as an application    of Theorem~\ref{T:radial}, Section \ref{low}.

\smallskip

\par
As mentioned in Section \ref{intro}, the picture described above shows that,  when $d \in (0, n-m)$, the couple of conditions \eqref{E:sup} and \eqref{E:inf}, or \eqref{E:supbar} and \eqref{E:infbar}, is not sufficient to characterize the optimal target in Sobolev trace emebddings. For instance,  both the measure $\mu_1$, given by \eqref{mu1}, and the measure $\mu_2$, defined as in \eqref{mu2},
%
%
%
%
satisfy  \eqref{E:sup} and \eqref{E:inf}. However, by Proposition \ref{sharpness}, the space $L^{\frac{n-d}{m}}(\Omega,\mu_2)$ is optimal in the embedding
%
\begin{equation}
    \label{gen62}
W^mL^{\frac{n-d}{m},1}(\Omega) \to L^{\frac{n-d}{m}}(\Omega,\mu_2)
\end{equation}
for the measure $\mu_2$, whereas, by Proposition \ref{P:radial-improvement}, the stronger  embedding
%
\begin{equation}
    \label{gen63}
W^mL^{\frac{n-d}{m},1}(\Omega) \to L^{\frac{n-d}{m},1}(\Omega,\mu_1)
\end{equation}
holds for the measure $\mu_1$, the space  $L^{\frac{n-d}{m},1}(\Omega,\mu_1)$ being optimal in this case.

%
%

\begin{proof}[Proof of Theorem \ref{L:endpointsub}] Consider part (i). Let $\nu$ be the measure on $\rn$ defined as in \eqref{extmeas}. Analogously to inequality \eqref{gen60}, one has that
\begin{equation}\label{riesz}
|u(x)|
\lesssim I_m(|\nabla\sp m u|)(x)
\quad \text{for $x \in \rn$,}
\end{equation}
for every function $u \in C^\infty _0(\rn)$, where
 $I_m$ denotes the Riesz potential operator of order $m$ in $\rn$. The constants in the relation ``$\lesssim$''   in this proof depend only on $m,n,d,\Omega$ and $\mu$.
By inequality \eqref{riesz} and~\cite[Theorem~1.2]{KK},
\begin{equation}\label{nov1}
\|u\|_{L\sp{\frac{n-d}{m}}(\Rn, \nu)}
\lesssim
\|I_m(|\nabla\sp m u|)|\|_{L\sp{\frac{n-d}{m}}(\Rn, \nu)} \lesssim \|\nabla\sp m u\|_{L\sp{\frac{n-d}{m},1}(\Rn)}
\end{equation}
for every  $u \in C^\infty _0(\rn)$.
Since the space $C^\infty _0(\rn)$ is dense in $W\sp{m}L\sp{\frac{n-d}{m},1}(\Rn)$, an extension and approximation argument as in the proof of Theorem \ref{P:endpoint},  part (i), tells us that a trace operator $T_\mu$ is well defined, and that inequality \eqref{E:endpointsub_estimate0} holds for every $u \in W^m_0L\sp{\frac{n-d}{m},1}(\Omega)$.
\\ Consider next part (ii).
As recalled above, the assumption that $\Omega$ be a bounded domain with the cone property ensures that it can be decomposed into
into a~finite union of bounded Lipschitz -- and hence extension -- domains $\{\Omega _j\}$, with $j=1, 	\dots , J$.  By \cite[Theorem 4.1]{CR},  there exist balls $B_j \supset \overline{\Omega _j}$ and bounded extension operators  $\mathcal E_j : W\sp{m}L\sp{\frac{n-d}{m},1}(\Omega_j)\to W\sp{m}_0L\sp{\frac{n-d}{m},1}(B_j)$.
Hence, owing to part (i),
if $u \in W\sp{m}L\sp{\frac{n-d}{m},1}(\Omega) \cap C^\infty (\Omega)$, then
\begin{align}\label{nov3}
\|u\|_{{L\sp{\frac{n-d}{m}}}(\Omega,{\mu})}
\leq
\sum_{j=1}\sp{J}
\|u\|_{{L\sp{\frac{n-d}{m}}}(\Omega_j,{\mu})}
&=
\sum_{j=1}\sp{J}
\|T_{\nu}(\mathcal E _j u)\|_{L\sp{\frac{n-d}{m}}(\Omega_j,\nu)}
\leq
\sum_{j=1}\sp{J}
\|T_{\nu}(\mathcal E _j u)\|_{L\sp{\frac{n-d}{m}}(B_j,\nu)} \\ \nonumber
&\lesssim
\sum_{j=1}\sp{J}
\|\nabla\sp m \mathcal E _j u\|_{L\sp{\frac{n-d}{m},1}(B_j)}
\lesssim
\|u\|_{W\sp mL\sp{\frac{n-d}{m},1}(\Omega)}.
\end{align}
Any function in $W\sp mL\sp{\frac{n-d}{m},1}(\Omega)$ can be approximated in norm and a.e. in $\Omega$ by  a sequence of functions in $
W\sp{m}L\sp{\frac{n-d}{m},1}(\Omega) \cap C^\infty (\Omega)$, as shown by an adaptation of the classical result for standard Sobolev spaces. Hence, inequality \eqref{E:endpointsub_estimate} follows.
\\ As far as part (iii) is concerned, one can again make use of an extension operator  in order to exploit  part (i), and argue along the same lines as in the proof  of Theorem \ref{P:endpoint}, part (iii). Here, an argument that justifies the existence of a trace operator  makes use of approximation  of functions in $W\sp{m}L\sp{\frac{n-d}{m},1}(\Omega)$ by  the restriction to $\Omega$ of sequences of functions from $C^\infty _0(\rn)$.  Such an approximation is possible since $\Omega$ is a bounded Lipschitz domain, the proof being analogous to that for standard Sobolev spaces as in \cite[Theorem~1.1.6/2]{Mabook}.
\end{proof}

In the proof of Proposition~\ref{sharpness} we shall make use of  Lemma \ref{L:operator} below. Given $n\in \N$ and $d\in(0,n)$, let us set
\begin{equation}\label{E:C}
  \Lambda_{n,d}=\{f\in\M_+(0,1): \text{$f$ is non-increasing and $\|f\|_{L^{n-d}(0,1)}\leq 1$}\}.
\end{equation}

\begin{lemma}\label{L:operator}
Let $n$, $m$, $d\in \N$, with $n\geq 2$ and $d<n-m$. Assume that $R \in (0,1)$. Let $H$ be the operator which maps any function    $f\in \Lambda_{n,d}$ into the function  defined as
$$
Hf(t)= \begin{cases} \displaystyle \chi_{(0,R)}(t)\int_t^R \int_{t_1}^R \dots \int_{t_{m-1}}^R \frac{f(t_m)}{t_m^m}\,dt_m \dots dt_1 &  \quad \text{if $m\geq 2$,}
\\  \\  \displaystyle
\chi_{(0,R)}(t) \int_t^R\frac{f(t_1)}{t_1}\,dt_1 & \quad \text{if $m=1$,}
\end{cases}
$$
for  $t\in (0,1)$.
\\
\textup{(}i\textup{)}  The function $Hf$ is $m$-times weakly differentiable in $(0,1)$,
and there exists a constant $C=C(n)$ such that
\begin{equation}\label{E:derivative_estimate}
\big|(Hf)^{(i)}(t)\big|\leq C t^{-i-\frac{1}{n-d}} \quad \text{for a.e.  $t\in (0,1)$},
\end{equation}
for every $f\in \Lambda_{n,d}$, and
 $i=0,1,\dots,m$. Here,  $(Hf)^{(i)}$ stands for the $i$-th-order derivative of $Hf$, and $H^{(0)}f=Hf$.
\\
\textup{(}ii\textup{)} One has that
\begin{equation}\label{gen15}
\sup_{f\in \Lambda_{n,d}} \|Hf\|_{L^{n-d}(0,1)}<\infty.
\end{equation}
\\
\textup{(}iii\textup{)} If $\|\cdot\|_{X(0,1)}$ is a rearrangement-invariant function norm such that
\begin{equation}\label{gen17}
\sup_{f\in \Lambda_{n,d}} \|Hf\|_{X(0,1)}<\infty,
\end{equation}
then
\begin{equation}\label{gen16}
L^{n-d}(0,1) \to X(0,1).
\end{equation}
\end{lemma}

\begin{proof} Consider part \textup{(i)}.
One has that
$$
(Hf)^{(i)}(t)=(-1)\sp{i}\int_t^R \int_{t_{i+1}}^R \dots \int_{t_{m-1}}^R \frac{f(t_m)}{t_m^m}\,dt_m \dots dt_{i+1} \quad \quad \text{for $i=1,\dots,m-2$},
$$
$$
(Hf)^{(m-1)}(t)=(-1)\sp{m-1}\int_t^R \frac{f(t_m)}{t_m^m}\,dt_m 
\quad \hbox{
and} \quad
(Hf)^{(m)}(t)=(-1)\sp{m}\frac{f(t)}{t^m}
$$
for a.e. $t\in (0,R)$, and $(Hf)^{(i)}$ vanishes elsewhere.
Inequality~\eqref{E:derivative_estimate}  hence follows,  since
\begin{equation*}
1\geq \int_0^t f^{n-d}(s)\,ds \geq tf^{n-d}(t) \quad\text{for $t\in (0,1)$}.
\end{equation*}
\par\noindent
As far as part \textup{(ii)} is concerned,  by  Fubini's theorem, if $f\in \Lambda_{n,d}$, then
\begin{equation}\label{E:fubini}
Hf(t)=\frac{\chi_{(0,R)}(t)}{(m-1)!} \int_t^R \frac{f(s)}{s^m}(s-t)^{m-1}\,ds
\leq \frac{\chi_{(0,R)}(t)}{(m-1)!} \int_t^R \frac{f(s)}{s}\,ds \quad\text{for $t\in (0,1)$}.
\end{equation}
Thus,
\begin{align*}
\sup_{f\in \Lambda_{n,d}} \|Hf\|_{L^{n-d}(0,1)}
&\lesssim \sup_{f\in \Lambda_{n,d}} \left\|\chi_{(0,R)}(t)\int_t^R \frac{f(s)}{s}\,ds\right\|_{L^{n-d}(0,1)}\\
& = \sup_{f\in \Lambda_{n,d}}\,\, \sup_{\|g\|_{L^{\frac{n-d}{n-d-1}}(0,1)}\leq 1} \int_0^R g(t) \int_t^R \frac{f(s)}{s}\,ds\,dt\\
&= \sup_{f\in \Lambda_{n,d}}\, \sup_{\|g\|_{L^{\frac{n-d}{n-d-1}}(0,1)}\leq 1} \int_0^R \frac{f(s)}{s} \int_0^s g(t)\,dt\,ds\\
&\lesssim \,\, \sup_{\|f\|_{L^{n-d}(0,1)}\leq 1} \,\sup_{\|g\|_{L^{\frac{n-d}{n-d-1}}(0,1)}\leq 1} \int_0^R f^*(s) g^{**}(s)\,ds\\
&\lesssim  \,\, \sup_{\|f\|_{L^{n-d}(0,1)}\leq 1} \,\sup_{\|g\|_{L^{\frac{n-d}{n-d-1}}(0,1)}\leq 1} \int_0^1 f^*(s) g^{**}(s)\,ds\\
&\lesssim \,\, \sup_{\|g\|_{L^{\frac{n-d}{n-d-1}}(0,1)}\leq 1} \|g^{**}\|_{L^{\frac{n-d}{n-d-1}}(0,1)}
\lesssim \,\, \sup_{\|g\|_{L^{\frac{n-d}{n-d-1}}(0,1)}\leq 1} \|g^{*}\|_{L^{\frac{n-d}{n-d-1}}(0,1)} =1.
\end{align*}
In the present proof, the constants in the relations $\lq\lq \lesssim "$ and $\lq\lq \approx " $   depend only on $m,n,d$ and $R$. Note that the last inequality holds owing to property \eqref{iv}. Equation \eqref{gen15} is thus established.
\par\noindent
Finally, we deal with part \textup{(iii)}. Let $f\in \Lambda_{n,d}$. Then
\begin{align*}
Hf(t)&=\frac{\chi_{(0,R)}(t)}{(m-1)!} \int_t^R \frac{f(s)}{s^m}(s-t)^{m-1}\,ds
\geq  \frac{\chi_{(0,\frac{R}{2})}(t)}{(m-1)!} \int_{2t}^R \frac{f(s)}{s^m}(s-t)^{m-1}\,ds\\
&\geq \frac{\chi_{(0,\frac{R}{2})}(t)}{2^{m-1}(m-1)!} \int_{2t}^R \frac{f(s)}{s}\,ds \quad\text{for $t\in (0,1)$}.
\end{align*}
Therefore, assumption \eqref{gen17} implies  that
\begin{equation}\label{E:sup_finite}
\sup_{f\in \Lambda_{n,d}} \left\|\chi_{(0,\frac{R}{2})}(t) \int_{2t}^R \frac{f(s)}{s}\,ds\right\|_{X(0,1)} <\infty.
\end{equation}
By property \eqref{X''}, Fubini's theorem, property \eqref{iv}, and the boundedness of the dilation operator on rearrangement-invariant spaces,
\begin{align}\label{E:associate}
\sup_{f\in \Lambda_{n,d}}
    &\left\|\chi_{(0,\frac{R}{2})}(t) \int_{2t}^R \frac{f(s)}{s}\,ds\right\|_{X(0,1)}
        =\sup_{f\in \Lambda_{n,d}} \sup_{\|g\|_{X'(0,1)}\leq 1}
        \int_0^{\frac{R}{2}} g^*(t) \int_{2t}^R \frac{f(s)}{s}\,ds\,dt
            \\ \nonumber
    &=\sup_{\|f\|_{L^{n-d}(0,1)}\leq 1} \sup_{\|g\|_{X'(0,1)}\leq 1}
        \int_0^R \frac{f^*(s)}{s} \int_0^{\frac{s}{2}} g^*(t)\,dt\,ds
            =\frac{1}{2} \sup_{ \|g\|_{X'(0,1)}\leq 1}
        \left\|\chi_{(0,R)}(s) g^{**}(s/2)\right\|_{L^{\frac{n-d}{n-d-1}}(0,1)}
        \\ \nonumber
    & \approx \sup_{\|g\|_{X'(0,1)}\leq 1}
        \|\chi_{(0,\frac{R}{2})} g^{**}\|_{L^{\frac{n-d}{n-d-1}}(0,1)}\\
    &\geq \frac{R}{2} \sup_{\|g\|_{X'(0,1)}\leq 1}
        \|g^{**}\|_{L^{\frac{n-d}{n-d-1}}(0,1)}
             \approx \nonumber \sup_{\|g\|_{X'(0,1)}\leq 1} \|g\|_{L^{\frac{n-d}{n-d-1}}(0,1)}.
\end{align}
Notice that the last but one inequality holds by the monotonicity of the function $g^{**}$.
 Combining equations~\eqref{E:sup_finite} and~\eqref{E:associate} tells us that  $X'(0,1) \to L^{\frac{n-d}{n-d-1}}(0,1)$, or, equivalently, that \eqref{gen16} holds.
\end{proof}

\begin{proof}[Proof of Proposition~\ref{sharpness}]
Assume, without loss of generality, that $\Omega_d=\Omega \cap \{(y,0): y \in \mathbb R^d, 0 \in \mathbb R^{n-d}\}$ and that  $0\in  \Omega$.
Denote  by    $\omega_d$ the Lebesgue measure of the $d$-dimensional unit ball in $\R^d$, and set $\gamma=\omega_d/\mu_2(\Omega)$.  Let $\Lambda_{n,d}$ be the set defined by \eqref{E:C}, and let $H$ be the operator introduced in Lemma~\ref{L:operator}.  Fix any $\kappa\in(n,\infty)$. Let $R\in (0,1)$,  $\alpha\in(0,\frac{1}{2m(n-d)})$ and   $f\in \Lambda_{n,d}$.  Define the set $E \subset \Rn$ as
$$
E=\left\{(y,z)\in \mathbb R^d \times \mathbb R^{n-d}: 0<\gamma|y|^d<R, |z|^{n-d}<|y|^{\kappa} Hf(\gamma |y|^d)^{n-d}\right\}.
$$
By Lemma~\ref{L:operator}, part \textup{(i)}, applied with $i=0$, we deduce that
\begin{equation}
    \label{gen64}
|z|^{n-d} <|y|^{\kappa} Hf(\gamma |y|^d)^{n-d} \lesssim (\gamma |y|^{d})^{\frac{\kappa}{d}-1}<R^{\frac{\kappa}{d}-1}\quad \hbox{for   $(y,z)\in E$.}
\end{equation}
Hence, if  $R$ is chosen small enough, then $\overline{E} \subseteq \Omega$ for every $f \in \Lambda_{n,d}$. Throughout this proof, constants either  explicitly appearing, or involved in the relations ``$\lesssim$'' and ``$\approx$'', depend only on $m,n,d,\Omega,\mu, \kappa$ and $R$.  In particular, all constants in this proof are independent of $f$ and  $\alpha$. In equation \eqref{gen64}  the involved constants  are, however, also independent of $R$.
%
Define the function $u: \rn \to [0, \infty)$ as
\begin{equation*}
u(y,z)= Hf(\gamma |y|^d)^{m(1-\alpha(n-d))} \big[Hf(\gamma |y|^d)^{\alpha(n-d)} |y|^{\kappa \alpha}-|z|^{(n-d)\alpha}\big]^m\chi_{E}(y,z) \quad \textup{for $(y,z)\in \mathbb R^d \times \mathbb R^{n-d}$.}
\end{equation*}
One can verify that the support of $u$ is contained in $\Omega$ and that $u$ is $m$-times weakly differentiable in $\Omega$.
To justify this assertion, one can in particular exploit the fact that   the function $u$, as well as all the terms arising by differentiating $u$ up to the order $m-1$, contain the factors $Hf(\gamma |y|^d)$ and $[Hf(\gamma |y|^d)^{\alpha(n-d)} |y|^{\kappa \alpha}-|z|^{(n-d)\alpha}]$ raised to positive powers, and that $Hf(\gamma |y|^d) \rightarrow 0$ whenever $\gamma |y|^d \rightarrow R$. Also, the verification that $u$ is actually $m$-times weakly differentiable across the singular set of $u$ can be accomplished by making use of the very definition of weak derivative, removing from $\Omega$ a neighbourhood  $U_r$ of $\mathbb R^d \times \{0\}$ with radius $r$, and showing that the integrals on $\partial U_r$ appearing after making use of the divergence theorem over $\Omega \setminus U_r$ approach $0$ as $r\to 0^+$.
\\
Let us define $B=\{y\in \mathbb R^d:~\gamma |y|^d<R\}$ and let $T_{\mu_2}$ denote the trace operator with respect to the measure $\mu_2$  defined as in \eqref{mu2}. Then,
\begin{equation}\label{E:trace_u}
T_{\mu_2} u(y,z)=Hf(\gamma |y|^d)^{m} |y|^{\kappa \alpha m} \chi_{B}(y) \quad\text{for $(y,z)\in \Omega_d$}.
\end{equation}
\\ We shall show that
\begin{equation}\label{gen52}
\|\nabla^m u\|_{L^{\frac{n-d}{m},1}(\Omega)} \leq C
\end{equation}
for some constant $C$. To begin with,
observe that
\begin{align*}
u(y,z)= Hf(\gamma |y|^d)^{m(1-\alpha(n-d))} \sum_{\ell=0}^{m} \binom{m}{\ell} (-1)^{m-\ell} Hf(\gamma |y|^d)^{\alpha \ell(n-d)} |y|^{\kappa l \alpha} |z|^{(n-d)(m-\ell)\alpha} \chi_{E}(y,z)
\end{align*}
 for $(y,z) \in \o$.
Next, for
each $\ell = 0,1,\dots, m$ define the functions
\begin{align*}
& S_\ell : \mathbb R^d \to [0, \infty) \quad \quad \quad  \hbox{as} \quad S_\ell(y)=Hf(\gamma |y|^d)^{m-\alpha(n-d)(m-\ell)} |y|^{\kappa \ell \alpha} \chi_{B}(y) \quad \quad\hbox{for $y\in \mathbb R^d$,}\\
&    T_\ell : \mathbb R^{n-d} \to [0, \infty) \quad \quad \quad \hbox{as} \quad T_\ell(z)=|z|^{(n-d)(m-\ell)\alpha} \quad\quad  \hbox{for $z\in \mathbb R^{n-d}$,}\\
&    F_\ell : \mathbb R^{n} \to [0, \infty) \quad \quad \quad \hbox{as} \quad F_\ell(y,z)=S_\ell(y) T_\ell(z)  \quad \quad\hbox{for $(y,z)\in \mathbb R^d \times \mathbb R^{n-d}$.}
\end{align*}
On denoting  by $\nabla_y^{j}$ the   $j$-th order gradient operator   with respect to the variables $y$, for $j=0,\dots, m$, and using the notation $\nabla_z^{j}$ with an analogous meaning,  we have that
\begin{align}\label{E:F}
|\nabla^m u|(y,z) \lesssim \sum_{\ell=0}^m |\nabla^m F_\ell|(y,z) \chi_E (y,z)
\lesssim \sum_{\ell=0}^m \sum_{j=0}^m |\nabla^j_y S_\ell|(y) |\nabla_z^{m-j} T_\ell|(z) \chi_E(y,z) \quad \hbox{for a.e. $(y,z) \in \o$.}
\end{align}
If $j\leq m$ and $\ell<m$, then
\begin{equation}\label{E:T}
|\nabla_z^{m-j} T_\ell|(z) \lesssim \alpha |z|^{(n-d)(m-\ell)\alpha-m+j} \quad\text{for a.e. $z\in \mathbb R^{n-d}$}.
\end{equation}
\\
In order to estimate $|\nabla^j_y S_\ell|$ for each $j=0,1,\dots,m$ and $\ell =0,1,\dots,m$, we set
\begin{align*}
I=\{(a_0, \dots,a_j,\beta):~&a_0\in \R, a_1,\dots,a_j\in \mathbb N_0, \, a_1 + 2 a_2  + \dots + j a_j\leq j,\\
&a_0+a_1+\dots +a_j=m-\alpha (m-\ell)(n-d),\, \beta=\kappa \ell \alpha + d(a_1+2a_2+\dots+j a_j)-j \},
\end{align*}
and observe that
\begin{multline}\label{E:S}
|\nabla^j_y S_\ell|(y) \\ \lesssim \chi_{\{y\in \mathbb R^d:~\gamma |y|^d<|{\rm supp} f| \}}(y)\sum_{(a_0,\dots,a_j,\beta)\in I} |y|^\beta Hf(\gamma |y|^d)^{a_0} (Hf)^{(1)}(\gamma |y|^d)^{a_1} \dots (Hf)^{(j)}(\gamma |y|^d)^{a_j}\\
\lesssim \chi_{\{y \in \mathbb R^d:~\gamma |y|^d<|{\rm supp} f|\}}(y) |y|^{\kappa \ell\alpha -\frac{nj}{n-d}} Hf(\gamma |y|^d)^{m-\alpha (m-\ell)(n-d)-j} \quad\text{for a.e.  $y\in \mathbb R^{n-d}$},
\end{multline}
where the first inequality can be proved by induction with respect to $j$ and the second inequality holds thanks to Lemma~\ref{L:operator}, part \textup{(i)}. The very definition of the set $E$ ensures that
\begin{equation}\label{E:definition_M}
\left(\frac{|z|}{Hf(\gamma |y|^d)}\right)^{n-d}\leq |y|^{\kappa}\leq \left(\frac{R}{\gamma}\right)^{\frac{\kappa}{d}} \quad\text{for $(y,z)\in E$.}
\end{equation}
From equations ~\eqref{E:F},~\eqref{E:T},~\eqref{E:S} and~\eqref{E:definition_M} one can deduce that
\begin{align}\label{E:gradient}
|\nabla^m u|(y,z)
&\lesssim \alpha \sum_{\ell=0}^{m-1} \sum_{j=0}^{m-1} |y|^{\kappa \ell\alpha -\frac{nj}{n-d}} \left(\frac{Hf(\gamma |y|^d)}{|z|}\right)^{-(n-d)(m-\ell)\alpha+m-j}\chi_E(y,z)\\
\nonumber
& \quad +\sum_{\ell=0}^{m} |y|^{\kappa \ell \alpha - \frac{nm}{n-d}} \left(\frac{|z|}{Hf(\gamma |y|^d)}\right)^{(n-d)(m-\ell)\alpha} \chi_E(y,z)\\
\nonumber
&\lesssim \alpha \sum_{\ell=0}^{m-1} \sum_{j=0}^{m-1} |y|^{-\frac{nj}{n-d}} \left(\frac{Hf(\gamma |y|^d)}{|z|}\right)^{-(n-d)(m-\ell)\alpha+m-j}\chi_E(y,z)\\
\nonumber
&\quad +\sum_{\ell=0}^{m} |y|^{- \frac{nm}{n-d}} \left(\frac{|z|}{Hf(\gamma |y|^d)}\right)^{(n-d)(m-\ell)\alpha} \chi_E(y,z)\\
\nonumber
&\lesssim \alpha \sum_{\ell=0}^{m-1} \sum_{j=0}^{m-1} \left(\frac{Hf(\gamma |y|^d)}{|z|}\right)^{-(n-d)(m-\ell)\alpha+m-j(1-\frac{n}{\kappa})}\chi_E(y,z)\\
\nonumber
&\quad +\sum_{l=0}^{m} |y|^{- \frac{nm}{n-d}+\alpha \kappa (m-\ell)} \chi_E(y,z)\\
\nonumber
&\lesssim \alpha \sum_{\ell=0}^{m-1} \sum_{j=0}^{m-1} \left(\frac{Hf(\gamma |y|^d)}{|z|}\right)^{-(n-d)(m-\ell)\alpha+m-j(1-\frac{n}{\kappa})}\chi_E(y,z) +|y|^{- \frac{nm}{n-d}} \chi_E(y,z)
\end{align}
for a.e. $(y,z) \in \o$.
Notice that the terms with $\ell=m$ and $j<m$ are not included in the above estimate since $T_m(z)=1$ for $z \in \mathbb R^{n-d}$, and therefore $|\nabla^{m-j}_z T_m|(z)=0$ whenever $j<m$.
Given any
$j=0,1,\dots,m-1$ and $\ell=0,1,\dots,m-1$,  we have that
\begin{align*}
&\left|\left\{(y,z)\in \mathbb R^d \times \mathbb R^{n-d}: \alpha \left(\frac{Hf(\gamma |y|^d)}{|z|}\right)^{-(n-d)(m-\ell)\alpha+m-j(1-\frac{n}{\kappa})}\chi_E(y,z)>\varrho \right\}\right|\\
&=\left|\left\{(y,z)\in \mathbb R^d \times \mathbb R^{n-d}: |z|^{n-d} < \chi_{B}(y) \left(\frac{\alpha}{\varrho}\right)^{\frac{n-d}{-(n-d)(m-\ell)\alpha+m-j(1-\frac{n}{\kappa})}} Hf(\gamma |y|^d)^{n-d} \right\}\right|\\
&\lesssim \left(\frac{\alpha}{\varrho}\right)^{\frac{n-d}{-(n-d)(m-\ell)\alpha+m-j(1-\frac{n}{\kappa})}} \int_{B} Hf(\gamma |y|^d)^{n-d}\,dy\\
&\approx\left(\frac{\alpha}{\varrho}\right)^{\frac{n-d}{-(n-d)(m-\ell)\alpha+m-j(1-\frac{n}{\kappa})}} \int_0^R Hf(s)^{n-d}\,ds
\lesssim \left(\frac{\alpha}{\varrho}\right)^{\frac{n-d}{-(n-d)(m-\ell)\alpha+m-j(1-\frac{n}{\kappa})}} \quad \hbox{for $\varrho \in (0, \infty)$,}
\end{align*}
where the last inequality holds thanks to Lemma~\ref{L:operator}, part \textup{(ii)}. Hence,
\begin{align}\label{E:rearrangement1}
&\left(\alpha \left(\frac{Hf(\gamma |y|^d)}{|z|}\right)^{-(n-d)(m-\ell)\alpha+m-j(1-\frac{n}{\kappa})}\chi_E(y,z)\right)^*(t)\\
\nonumber
&\lesssim \alpha t^{\frac{(n-d)(m-\ell)\alpha-m+j(1-\frac{n}{\kappa})}{n-d}} \chi_{(0,|E|)}(t)
\lesssim \alpha t^{\alpha-\frac{m}{n-d}} \chi_{(0,|
E|)}(t) \quad\text{for $t\in(0,\infty)$}.
\end{align}
Furthermore, by Lemma~\ref{L:operator}, part \textup{(i)},
\begin{align*}
&\left|\left\{(y,z)\in E: |y|^{-\frac{nm}{n-d}}>\varrho\right\}\right|
\leq \left|\left\{(y,z)\in B \times \mathbb R^{n-d}: |y|^{-\frac{nm}{n-d}}>\varrho, |z|^{n-d}<Hf(\gamma |y|^d)^{n-d} |y|^\kappa \right\}\right|\\
&\lesssim \left|\left\{(y,z)\in \mathbb R^d \times \mathbb R^{n-d}: |y|^d<\varrho^{-\frac{(n-d)d}{nm}}, |z|^{n-d}\lesssim |y|^{\kappa-d} \right\}\right|\\
&\lesssim \int_{\{x\in \mathbb R^d: |x|^d<\varrho^{-\frac{(n-d)d}{nm}}\}} |y|^{\kappa-d}\,dy
\lesssim \int_0^{C\varrho^{-\frac{(n-d)d}{nm}}} s^{\frac{\kappa}{d}-1}\,ds
\lesssim \varrho^{-\frac{(n-d)\kappa}{nm}}\quad\text{for $\varrho\in(0,\infty)$,}
\end{align*}
for some constant $C$.
Hence,
\begin{equation}\label{E:rearrangement2}
\left(|y|^{-\frac{nm}{n-d}}\chi_E(y,z)\right)^*(t)\lesssim t^{-\frac{nm}{(n-d)\kappa}}\chi_{(0,|E|)}(t) \quad\text{for $t\in(0,\infty)$}.
\end{equation}
A combination of~\eqref{E:gradient},~\eqref{E:rearrangement1} and~\eqref{E:rearrangement2} thus yields
$$
|\nabla^m u|^*(t)\lesssim \left(\alpha t^{\alpha-\frac{m}{n-d}} + t^{-\frac{nm}{(n-d)\kappa}}\right) \chi_{(0,|\Omega|)}(t) \quad\text{for $t\in(0,\infty)$,}
$$
whence equation \eqref{gen52} follows.
\\ Since $u$ is compactly supported   in $\Omega$, we infer from \eqref{gen52}, via a general Poincar\'e inequality (see, e.g.,~\cite[Lemma 4.2]{CP1998}),   that $u \in W^m_0L^{\frac{n-d}{m},1}(\Omega)$, and
\begin{equation}\label{E:boundedness}
\|u\|_{W^mL^{\frac{n-d}{m},1}(\Omega)}\leq C
\end{equation}
for some constant $C$.
\\
Now, assume that $\|\cdot \|_{Z(0,1)}$ is a rearrangement-invariant  function norm such that
$$
W^mL^{\frac{n-d}{m},1}(\Omega) \rightarrow Z(\Omega, \mu_2).
$$
Thanks to~\eqref{E:trace_u} and~\eqref{E:boundedness},
$$
\left\|Hf(\gamma |y|^d)^{m} |y|^{\kappa \alpha m} \chi_{B}(y) \chi_{\Omega_d}(y,z)\right\|_{Z(\Omega, \mu_2)} \leq C
$$
for some constant $C$. Hence, passing to limit as $\alpha \to 0^+$ and making use of    property (P3) of the definition of rearrangement-invariant function norm imply that
$$
\left\|Hf(\gamma |y|^d)^{m} \chi_{B}(y)\chi_{\Omega_d}(y,z)\right\|_{Z(\Omega, \mu_2)} \leq C.
$$
Since
\begin{align*}
\left\|Hf(\gamma |y|^d)^{m} \chi_{B}(y)\chi_{\Omega_d}(y,z)\right\|_{Z(\Omega, \mu_2)}
=\left\|Hf(t)^{m} \chi_{(0,R)}(t)\right\|_{{Z}(0, 1)} =\left(\left\|Hf(t) \chi_{(0,R)}(t)\right\|_{{Z}^{\{m\}}(0, 1)} \right)^m,
\end{align*}
the assumption of Lemma~\ref{L:operator}, part \textup{(iii)}, is fulfilled with $\|\cdot\|_{X(0,1)}$ replaced by $\|\cdot\|_{{Z}^{\{m\}}(0, 1)}$. Thus,   $L^{n-d}(0, 1) \to {Z}^{\{m\}}(0, 1)$. This embedding can be written as
$$
\|f^m\|_{{Z}(0, 1)}\lesssim \|f^m\|_{L^{\frac{n-d}{m}}(0, 1)}
$$
for $f\in \mathcal M_+(0, 1)$,
which, in its turn, implies  that $L^{\frac{n-d}{m}}(\Omega, \mu_2) \to Z(\Omega, \mu_2)$, thus proving the optimality of the  target norm  $\|\cdot \|_{L^{\frac{n-d}{m}}(\Omega, \mu_2)}$ in~\eqref{E:endpointsub_estimate} among all rearrangement-invariant norms.
\end{proof}

\section{Main results -- fast decaying measures}\label{high}

%

This section deals with $d$-Frostman measures $\mu$  for $d\in [n-m, n]$. A sharp embedding in this range of values of $d$, for the endpoint space $W^{m,1}(\Omega),$ has been established in Theorem \ref{P:endpoint}. This result, coupled with a classical embedding at an apposite endpoint, is exploited here to derive a reduction principle to one-dimensional Hardy-type inequalities for  Sobolev embeddings involving arbitrary rearrangement-invariant norms. Moreover, the Hardy inequalities in question are shown to be fully equivalent to the Sobolev embeddings if the power $d$ in the decay of the measure $\mu$ on balls is sharp at least   at one point. The reduction principle provides us with a key tool for the identification of the optimal rearrangement-invariant target space in the relevant Sobolev embeddings.


\begin{theorem}{\rm{\bf [Reduction principle: case $n-m \leq d \leq n$]}}\label{T:sufficiency}
Let $\Omega$ be an open set with finite Lebesgue measure in $\rn$, $n\geq 2$, let $m\in \N$, with $m<n$, and let $d \in [n-m, n]$. Let $\|\cdot\|_{X(0,1)}$ and $\|\cdot\|_{Y(0,1)}$ be rearrangement-invariant function norms. Assume that there exists a constant
$C_1$ such that
\begin{equation}\label{E:assumption}
\left\|\int_{t^{\frac{n}{d}}}^1 f(s) s^{-1+\frac{m}{n}}\,ds\right\|_{Y(0,1)} \leq C_1 \|f\|_{  X(0,1)}
\end{equation}
for every nonnegative non-increasing function $f\in  X(0,1)$.
\\
\textup{(}i\textup{)} Let $\mu$ be a finite  Borel measure on $\Omega$ fulfilling \eqref{E:sup}. Then
\begin{equation}\label{E:4.1-i}
\|u\|_{Y({\Omega},\mu)} \leq C_2\|\nabla ^m u\|_{X(\Omega)}
\end{equation}
for some constant $C_2$ and every $u\in W^{m}_0X(\Omega)$.
\\
\textup{(}ii\textup{)} Assume, in addition, that $\Omega$ is bounded  and has the cone property. Let  $\mu$ be a Borel measure on $\o$ fulfilling \eqref{E:sup}.  Then
\begin{equation}\label{E:4.1-ii}
\|u\|_{Y({\Omega},\mu)} \leq C_2 \|u\|_{W^mX(\Omega)}
\end{equation}
for some constant $C_2$ and every $u\in W^{m}X(\Omega)$.
\\
\textup{(}iii\textup{)} Assume that $\Omega$ is a bounded Lipschitz domain.  Let $\mu$ be a Borel measure on $\overline{\Omega}$ fulfilling \eqref{E:supbar}.  Then
\begin{equation}\label{E:4.1-iii}
\|u\|_{Y(\overline{\Omega},\mu)} \leq C_2 \|u\|_{W^mX(\Omega)}
\end{equation}
for some constant $C_2$ and every $u\in W^mX(\Omega)$.
\end{theorem}

\begin{remark}\label{gen50} {\rm One can show that inequality \eqref{E:assumption} for  nonnegative and non-increasing functions  holds if and only if it  just holds for nonnegative functions -- see \cite[Corollary 9.8]{CPS}.}
\end{remark}

The necessity of condition \eqref{E:assumption}, under asumption \eqref{E:inf} or \eqref{E:infbar}, is the subject of the next result. Let us stress that such a condition is necessary   for  any $d \in (0,n]$, and not just for  $d \in [n-m, n]$.


\begin{theorem}{\rm{\bf [Necessity]}}\label{T:necessity}
Let $\Omega$ be an open set with finite Lebesgue measure    in $\Rn$, $n\geq 2$, let $m\in \N$  with $m<n$, and let $d\in (0,n]$. Let $\|\cdot\|_{X(0,1)}$ and $\|\cdot\|_{Y(0,1)}$ be rearrangement-invariant function  norms.
\\
\textup{(}i\textup{)} Assume that $\mu$ is a finite Borel measure on $\Omega$ fulfilling \eqref{E:sup} and \eqref{E:inf}. If either inequality \eqref{E:4.1-i} or inequality \eqref{E:4.1-ii} is satisfied for some constant $C_2$,  then there exists a constant $C_1$  such that inequality \eqref{E:assumption} holds.
\\
\textup{(}ii\textup{)} Assume that $\mu$ is a finite Borel measure on $\overline\Omega$ fulfilling \eqref{E:supbar} and \eqref{E:infbar}. If inequality \eqref{E:4.1-iii} is satisfied for some constant $C_2$, then there exists a constant $C_1$  such that inequality \eqref{E:assumption} holds.
\end{theorem}

Theorems \ref{T:sufficiency} and \ref{T:necessity} enable us to exhibit the optimal  rearrangement-invariant target norms in  inequalities~\eqref{E:4.1-i}--\eqref{E:4.1-ii}, or~\eqref{E:4.1-iii},   for any measure $\mu$ fulfilling \eqref{E:sup} and \eqref{E:inf}, or \eqref{E:supbar} and \eqref{E:infbar}, with $d \in [n-m,n]$. The  optimal space is built upon the  function norm $\|\cdot \|_{X\sp m_{d,n}(0,1)}$  obeying
\begin{equation}\label{E:trace_opt_norm}
\|f\|_{(X\sp m_{d,n})'(0,1)}
  =\bigg\|t\sp{-1+\frac mn}\int _0^{t^{\frac d n}} f\sp{*}(s)ds\bigg\|_{X'(0,1)}
\end{equation}
for    $f\in \mathcal \Mpl(0,1)$.

\begin{theorem}{\rm{\bf [Optimal range space]}}\label{T:optimal} Let $\Omega$ be an open set with finite Lebesgue measure  in $\Rn$, $n\geq 2$, let $m\in \N$  with
$m<n$, and let $d \in [n-m, n]$. Let $\|\cdot\|_{X(0,1)}$ be a~rearrangement-invariant function  norm.
\\
\textup{(}i\textup{)} Assume that $\mu$ is a finite Borel measure on ${\Omega}$ fulfilling condition \eqref{E:sup}.
Then
\begin{equation}\label{E:optimal emb-i}
W^m_0X(\Omega) \to X\sp m_{d,n}(\Omega, \mu).
\end{equation}
\\
\textup{(}ii\textup{)} Assume, in addition, that $\Omega$ is bounded and has the cone property. Let
$\mu$ be a Borel measure on ${\Omega}$  fulfilling condition \eqref{E:sup}.
Then
\begin{equation}\label{E:optimal emb-ii}
W^mX(\Omega) \to X\sp m_{d,n}(\Omega, \mu).
\end{equation}
\\
\textup{(}iii\textup{)} Assume that $\Omega$ is a bounded Lipschitz domain.  Let
$\mu$ be a Borel measure on $\overline{\Omega}$ fulfilling condition \eqref{E:supbar}.
Then
\begin{equation}\label{E:optimal emb-iii}
W^mX(\Omega) \to X\sp m_{d,n}(\overline \Omega, \mu).
\end{equation}
Moreover,
 if, in addition, $\mu$ satisfies condition \eqref{E:inf} in cases \textup{(}i\textup{)} and \textup{(}ii\textup{)}, or \eqref{E:infbar} in case \textup{(}iii\textup{)}, then the target space is optimal,  in each case, among all rearrangement-invariant spaces in the pertaining embedding.
\end{theorem}

\bigskip

Let us point out that, in the special case when $\mu$ is Lebesgue measure, and hence $d=n$, Theorems \ref{T:sufficiency}, \ref{T:necessity} and \ref{T:optimal} overlap with results form \cite{CPS}.

\begin{proof}[Proof of Theorem~\ref{T:sufficiency}]
Consider first part (iii).  Theorem~\ref{P:endpoint}, part (iii), yields
\begin{equation}\label{E:endpoint1}
W\sp{m,1}(\Omega)\to L^{\frac{d}{n-m},1}(\overline{\Omega},\mu).
\end{equation}
Furthermore, since $\Omega$ is a~Lipschitz domain, one has that $W^mL^{\frac{n}{m},1}(\Omega) \to W^1L^{n,1}(\Omega) \to C(\overline \Omega)$, where the first embedding goes back to \cite{Oneil, Peetre} and the second one holds by a result of \cite{Stein} and the fact that $\Omega$ is an extension domain. Thus,
if $u\in W^mL^{\frac{n}{m},1}(\Omega)$, then $u$ is continuous in $\overline{\Omega}$, and
\begin{equation}\label{E:endpoint2}
\|u\|_{L^\infty(\overline{\Omega},\mu)} = \|u\|_{L^\infty(\Omega)} \lesssim\|u\|_{W^mL^{\frac{n}{m},1}(\Omega)}.
\end{equation}
Here, and throughout this proof, constants either explicitly appearing, or involved in the relations  $\lq\lq \lesssim"$ and $\lq\lq \approx"$,   depend on $n$, $m$, $d$, $\Omega $ and $\mu (\Omega)$.
Inequality \eqref{E:endpoint2} tells us that
\begin{equation}\label{E:right-trace-embedding}
W\sp mL\sp{\frac{n}{m},1}(\Omega)\to L\sp{\infty}(\overline{\Omega},\mu).
\end{equation}
 Inequalities~\eqref{E:endpoint1} and~\eqref{E:endpoint2} imply, via inequality \eqref{K} applied to the trace operator $T_\mu$, that there exists a constant $C$ such that
$$
K(t,u;L^{\frac{d}{n-m},1}(\overline{\Omega},\mu), L^\infty(\overline{\Omega},\mu))
\lesssim K(Ct,u; W^{m,1}(\Omega), W^mL^{\frac{n}{m},1}(\Omega)) \quad\text{for $t\in(0,\infty)$,}
$$
and for $u\in W^{m,1}(\Omega)$. Fix any such function $u$. By formulas~\eqref{E:fx2.14} and~~\eqref{E:fx2.13}, the latter inequality takes the form
\begin{equation}\label{gen3}
\int_0^{t^{\frac{d}{n-m}}} s^{-1+\frac{n-m}{d}} u^*_\mu(s)\,ds
\lesssim\int_0^{ct^{\frac{n}{n-m}}} s^{-\frac{m}{n}} \int_s^{\infty} |D^m u|^*(r) r^{-1+\frac{m}{n}}\,dr\,ds \quad\text{for $t\in(0,\infty)$,}
\end{equation}
for some constant $c$.
Via an estimate for the right-hand side of inequality~\eqref{gen3} as in the proof of~\cite[Theorem 4.2]{KP}, one can deduce that
\begin{align*}
\int_0^t s^{-1+\frac{n-m}{d}} u^*_\mu(\mu(\overline{\Omega}) s)\,ds
\lesssim \int_0^{t} s^{-1+\frac{n-m}{d}} \int_{\frac{s^{\frac{n}{d}}}{c}}^{1} |D^m u|^*(|\Omega|r) r^{-1+\frac{m}{n}}\,dr\,ds \quad\text{for $t\in(0,\infty)$,}
\end{align*}
for some constant  $c$. Here, we have made use of the fact that $|D^m u|^*$ vanishes outside $(0, |\Omega|)$.
Hence, by property~\eqref{E:hardy-lemma},
\begin{equation}\label{E:HLP-appl}
\int_0^1  h(t) t^{-1+\frac{n-m}{d}} u^*_\mu(\mu(\overline{\Omega}) t)\,dt
\lesssim \int_0^1  h(t) t^{-1+\frac{n-m}{d}} \int_{\frac{t^{\frac{n}{d}}}{c}}^1 |D^m u|^*(|\Omega| r) r^{-1+\frac{m}{n}}\,dr  \,dt,
\end{equation}
for any non-increasing function $h: (0,1) \rightarrow [0,\infty)$.
Given $f\in\Mpl(0,1)$, define the function $Sf \in \Mpl(0,1)$ as
$$
Sf(t)=t^{-1+\frac{n-m}{d}} \sup_{t<s<1} s^{1-\frac{n-m}{d}} f^*(s) \quad\text{for $t\in(0,1)$},
$$
and choose
$$
h(t)=\sup_{t<s<1} s^{1-\frac{n-m}{d}} f^*(s) \quad\text{for $t\in(0,1)$}
$$
in \eqref{E:HLP-appl}. This results in the inequality
\begin{equation}\label{E:S-1}
\int_0^1 Sf(t) u^*_\mu(\mu(\overline{\Omega})t) \,dt
\lesssim \int_0^1 Sf(t) \bigg(\int_{\frac{t^{\frac{n}{d}}}{c}}^1 |D^m u|^*(|\Omega| s) s^{-1+\frac{m}{n}}\,ds\bigg)  \,dt.
\end{equation}
Next, define the functional $\|\cdot\|_{X_\sigma(0,1)}$ on $\Mpl(0,1)$ by
\begin{equation}\label{E:X-sigma}
\|f\|_{X_\sigma(0,1)}=\bigg\|t^{-1+\frac{m}{n}} \int_0^{t^{\frac{d}{n}}} f^*(s)\,ds\bigg\|_{X'(0,1)}
\end{equation}
for $f\in \mathcal \Mpl(0,1)$. Since $-1+\frac{m}{n}+\frac{d}{n}\geq0$, it follows from~\cite[Proposition~3.1]{CP-Trans} that the functional $\|\cdot\|_{X_\sigma(0,1)}$ is equivalent to a rearrangement-invariant function norm. Define the operators $T$ and $T'$ by
\begin{equation*}
  Tf(t)=\int_{t^{\frac{n}{d}}}^1 f(s) s^{-1+\frac{m}{n}}\,ds  \quad \hbox{and}  \quad T'f(t)=t^{-1+\frac{m}{n}} \int_0^{t^{\frac{d}{n}}} f(s)\,ds
 \quad\text{for $t\in(0,1)$,}
\end{equation*}
for $f\in \mathcal \Mpl(0,1)$. Fubini's theorem ensures that equation~\eqref{E:duo} holds with this choice of $T$ and $T'$. Hence, equation~\eqref{E:novabis} is satisfied as well. Thus, since assumption~\eqref{E:assumption} implies that $T\colon X(0,1)\to Y(0,1)$, one  has that   $T'\colon Y'(0,1)\to X'(0,1)$, namely,
\begin{equation*}
\bigg\|t^{-1+\frac{m}{n}} \int_0^{t^{\frac{d}{n}}} f(s)\,ds\bigg\|_{X'(0,1)} \lesssim \|f\|_{Y'(0,1)}
\end{equation*}
for $f\in \mathcal \Mpl(0,1)$. In particular, inasmuch as  $\|\cdot\|_{Y'(0,1)}$ is a rearrangement-invariant function norm,
\begin{equation*}
\bigg\|t^{-1+\frac{m}{n}} \int_0^{t^{\frac{d}{n}}} f^{*}(s)\,ds\bigg\|_{X'(0,1)} \lesssim \|f\|_{Y'(0,1)}
\end{equation*}
for $f\in \mathcal \Mpl(0,1)$. The very definition of $\|\cdot\|_{X_\sigma(0,1)}$ implies that
$Y'(0,1)\to X_{\sigma}(0,1)$, whence,  by~\eqref{emb},
\begin{equation}\label{E:embedding}
X'_\sigma(0,1) \to Y(0,1).
\end{equation}
Owing to
 the definition of  associate function norm, to the pointwise inequality $f^*\leq Sf$ for $f\in \Mpl(0,1)$, to inequality~\eqref{E:S-1} and to H\"older's inequality, we get
\begin{align}\label{E:estimate-by-Sf}
\|u^*_\mu(\mu(\overline{\Omega}) t)\|_{X'_\sigma(0,1)}
&=\sup_{\|f\|_{X_\sigma(0,1)}\leq 1} \int_0^1 u^*_\mu(\mu(\overline{\Omega}) t) f^*(t)\,dt
\leq \sup_{\|f\|_{X_\sigma(0,1)}\leq 1} \int_0^1 u^*_\mu(\mu(\overline{\Omega}) t) Sf(t)\,dt\\
\nonumber
&\lesssim \sup_{\|f\|_{X_\sigma(0,1)}\leq 1} \int_0^1 \bigg(\int_{\frac{t^{\frac{n}{d}}}{c}}^1 |D^m u|^*(|\Omega|s) s^{-1+\frac{m}{n}}\,ds\bigg) Sf(t)\,dt\\
\nonumber
&\lesssim \sup_{\|f\|_{X_\sigma(0,1)}\leq 1} \bigg\|\int_{\frac{t^{\frac{n}{d}}}{c}}^1 |D^m u|^*(|\Omega|s) s^{-1+\frac{m}{n}}\,ds\bigg\|_{X'_\sigma(0,1)} \|Sf\|_{X_\sigma(0,1)}.
\end{align}
We claim that
\begin{equation}\label{E:boundedness_S}
\|Sf\|_{X_\sigma(0,1)} \lesssim\|f\|_{X_\sigma(0,1)}
\end{equation}
for $f\in \Mpl(0,1)$.
Indeed, by~\cite[Theorem 3.8]{KP},
$$
(Sf)^{**}(t) \lesssim(Sf^{**})(t) \quad\text{for $t\in(0,1)$}.
$$
Hence,
\begin{align*}
\|Sf\|_{X_\sigma(0,1)}
&=\left\|t^{-1+\frac{m+d}{n}} (Sf)^{**}(t^{\frac{d}{n}})\right\|_{X'(0,1)} \lesssim\left\|t^{-1+\frac{m+d}{n}} (Sf^{**})(t^\frac{d}{n})\right\|_{X'(0,1)} \approx \bigg\|\sup_{t^{\frac{d}{n}}<s<1} s^{1-\frac{n-m}{d}} f^{**}(s)\bigg\|_{X'(0,1)}\\
&\approx\left\|\sup_{t<s<1} s^{-1+\frac{d}{n}+\frac{m}{n}}f^{**}(s^{\frac{d}{n}})\right\|_{X'(0,1)} \approx\bigg\|\sup_{t<s<1} s^{-1+\frac{m}{n}} \int_0^{s^{\frac{d}{n}}} f^*(r)\,dr\bigg\|_{X'(0,1)}\\
&\approx\left\|\sup_{t<s<1} s^{-1+\frac{m}{n}} \int_0^s f^*(r^{\frac{d}{n}}) r^{-1+\frac{d}{n}}\,dr\right\|_{X'(0,1)} .
\end{align*}
By \cite[Theorem 3.9]{KP}, the last norm does not exceed  a constant times
$$\bigg\|s^{-1+\frac{m}{n}}\int_0^s f^*(r^{\frac{d}{n}}) r^{-1+\frac{d}{n}}\,dr\bigg\|_{X'(0,1)},$$
and the latter norm  agrees with
$\|f\|_{X_\sigma(0,1)}$. Hence, inequality
~\eqref{E:boundedness_S} follows.
\\ Thanks to inequalities~\eqref{E:estimate-by-Sf},~\eqref{E:boundedness_S}, \eqref{E:X-sigma} and to   the boundedness of the dilation operator on rearrangement-invariant spaces,
\begin{align*}
\left\|u^*_\mu(\mu(\overline{\Omega})t)\right\|_{X'_\sigma(0,1)}
\lesssim \left\|\int_{t^{\frac{n}{d}}}^1 |D^m u|^*(|\Omega|s) s^{-1+\frac{m}{n}}\,ds\right\|_{X'_\sigma(0,1)}\lesssim \left\| |D^m u|^*(|\Omega|t)\right\|_{X(0,1)}
= \left\|D^m u \right\|_{X(\Omega)}.
\end{align*}
Hence, by~\eqref{E:embedding},
$$
\|u\|_{Y(\overline{\Omega},\mu)}
= \left\|u^*_\mu(\mu(\overline{\Omega})t)\right\|_{Y(0,1)}
\lesssim\left\|u^*_\mu(\mu(\overline{\Omega}) t)\right\|_{X'_\sigma(0,1)}
\lesssim\left\|D^m u\right\|_{X(\Omega)}
\approx\left\| u\right\|_{W^mX(\Omega)}.
$$
Notice that the last inequality holds thanks to equation \eqref{gen21}.
Inequality {\eqref{E:4.1-iii}} is thus established.

\smallskip
\par\noindent
Concerning part (ii), as recalled in the proof of Theorem \ref{P:endpoint}, since $\Omega$ is a bounded domain with the cone property, it can be decomposed into a finite union  of bounded Lipschitz domains $\{\o_j\}$, with $j=1, \dots , J$. Let us denote by $u_j$ the restriction of $u$ to $\o_j$, for $j=1, \dots , J$. On applying inequality \eqref{gen3} to each  function $u_j$, and adding the resultant inequalities tell us that there exists a constant $c$ such that
\begin{equation}\label{gen4}
\int_0^{t^{\frac{d}{n-m}}} s^{-1+\frac{n-m}{d}}\sum _{j=1}^J (u_j)^*_\mu(s)\,ds
\lesssim \int_0^{ct^{\frac{n}{n-m}}} s^{-\frac{m}{n}} \int_s^{\infty}\sum _{j=1}^J |D^m u_j|^*(r) r^{-1+\frac{m}{n}}\,dr\,ds \quad\text{for $t\in(0,\infty)$}.
\end{equation}
By an iteration of inequality \eqref{gen5}, there exists a constant $c=c(J)$ such   that
\begin{equation}\label{gen6}
u_\mu^*(t) \leq \Big(\sum _{j=1}^J u\chi_{\o_j}\Big)_\mu ^*(t) =
  \Big(\sum _{j=1}^J u_j\Big)_\mu ^*(t) \leq \sum _{j=1}^J (u_j)^*_\mu (c t) \quad\text{for $t\in(0,\infty)$}.
\end{equation}
On the other hand,
\begin{equation}\label{gen7}
\sum _{j=1}^J |D^m u_j|^*(t) \leq J |D^m u|^*(t) \quad\text{for $t\in(0,\infty)$}.
\end{equation}
Owing to inequalities \eqref{gen4}--\eqref{gen7}, inequality  \eqref{gen3} continues to hold, for a  suitable  constant $C$, also in the present case. The remaining part of the proof is completely analogous to that  of assertion (iii), and will be omitted.

\smallskip
\par\noindent
The proof of the part (i) is analogous.
\end{proof}

\begin{proof}[Proof of Theorem~\ref{T:necessity}]
We shall provide details about  part  (ii). Part (i) can be proved analogously. Assume, without loss of generality, that the point $x_0$ appearing in \eqref{E:inf} agrees with $0$, and   denote by $B_r$ the ball centered at $0$,  with radius equal to $r$. Given any nonnegative  function $f\in X(0,\omega_n R^n)$, define the $m$-times weakly differentiable, non-increasing function $g: (0, \infty) \to [0, \infty)$ as
$$
g(t)=\begin{cases}\displaystyle  \chi_{(0,\omega_n R^n)}(t) \int_t^{\omega_n R^n} \int_{r_1}^{\omega_n R^n} \dots \int_{r_{m-1}}^{\omega_n R^n} f(r_m) r_m^{-m+\frac{m}{n}} \,dr_m \dots \,dr_1 & \quad\text{if $m \geq 2$},
\\ \\  \displaystyle
\chi_{(0,\omega_n R^n)}(t) \int_t^{\omega_n R^n}  f(r_1) r_1^{-1+\frac{1}{n}} \,dr_1   & \quad\text{if $m = 1$},
\end{cases}
$$
for {$t \in (0, \infty) $},
and the $m$-times weakly differentiable function $u: \rn \to [0, \infty)$ as $u(x)=g(\omega_n |x|^n)$ for $x\in \Rn$. Set $r(\varrho)=(g^{-1}(\varrho)/\omega_n)^{\frac{1}{n}}$ for $\varrho\in(0,\infty)$, where $g^{-1}$ stands for a (generalized) inverse of $g$. Then,
\begin{align*}
\mu(\{x\in \overline{\Omega}: u(x) >\varrho\})
&=\int_{\overline{\Omega}} \chi_{\{x\in \overline{\Omega}:~u(x)>\varrho\}}(x)\,d\mu(x)
=\int_{\overline{\Omega}} \chi_{\{x\in \overline{\Omega}:~g(\omega_n |x|^n) >\varrho\}}(x)\,d\mu(x)\\
&
=\int_{\overline{\Omega}} \chi_{\left\{x\in \overline{\Omega}:~|x|<r(\varrho)\right\}}(x)\,d\mu(x)
=\mu\left(\overline{\Omega} \cap B_{r(\varrho)}\right) \quad \hbox{for $\varrho \in (0, \infty)$.}
\end{align*}
Since $g(t)=0$ whenever $t\geq \omega_n R^n$, we have that $g^{-1}(\varrho)\leq \omega_n R^n$, and therefore $r(\varrho)\leq R$ for every $\varrho>0$. Consequently, by assumption~\eqref{E:infbar},
\begin{equation}\label{gen65}
\mu(\{x\in \overline{\Omega}: u(x) >\varrho\})
=\mu\left(\overline{\Omega} \cap B_{r(\varrho)}\right)
\gtrsim(g\sp{-1}(\varrho))^{\frac{d}{n}} \quad\text{for $\varrho\in(0,\infty)$}.
\end{equation}
Here, and throughout this proof, the relation $\lesssim$ holds up to constants depending on $n$, $m$, $d$ and $\mu$.
Inequality \eqref{gen65} implies that
$$
u^*_\mu(t)\geq  g(ct^{\frac{n}{d}}) \quad \hbox{for $t \in (0, \infty)$,}
$$
for some constant $c$ depending on the same quantities.
Hence, by the boundedness of the dilation operator in rearrangement-invariant spaces,
\begin{equation}\label{E:heart}
\|u\|_{Y(\overline{\Omega},\mu)}
= \|u^*_\mu(\mu(\overline{\Omega}) t)\|_{Y(0,1)}
\gtrsim \|g(t^{\frac{n}{d}})\|_{Y(0,1)}.
\end{equation}
Via an iterated use of Fubini's theorem one can show that
\begin{equation}\label{gen26}
g(t^{\frac nd}) \gtrsim \chi_{(0,\omega_n R^n/2)}(t^{\frac{n}{d}}) \int_{ct^{\frac nd}}^{1}  f(\omega_n R^n r) r^{-1+\frac{m}{n}} \,dr \qquad  \hbox{for $t \in (0, \infty)$,}
\end{equation}
whence
\begin{equation}\label{E:1}
\|g(t^{\frac{n}{d}})\|_{Y(0,1)}
\gtrsim \left\|\int_{t^{\frac{n}{d}}}^1 f(\omega_n R^n s) s^{-1+\frac{m}{n}}\,ds\right\|_{Y(0,1)}.
\end{equation}
On the other hand, by~\cite[inequality (4.20)]{CP-Trans},
\begin{equation}\label{E:2}
\|u\|_{W^mX(\Omega)} \lesssim \|f(\omega_n R^n t)\|_{X(0,1)}.
\end{equation}
Combining assumption~\eqref{E:4.1-iii} with inequalities~\eqref{E:heart} -- \eqref{E:2} completes the proof.
\end{proof}

\begin{proof}[Proof of Theorem~\ref{T:optimal}]
One can verify that
 the functional $f\mapsto \|f\|_{(X\sp m_{d,n})'(0,1)}$ is equivalent to a~rearrangement-invariant function norm. Moreover, by
\cite[Proposition~3.1]{CP-Trans},
\[
\left\|\int_{t\sp{\frac nd}}\sp1f(s) s\sp{-1+\frac mn}\,ds\right\|_{X\sp m_{d,n}(0,1)}
\leq
\|f\|_{X(0,1)}
\]
for every nonnegative function $f \in X(0,1)$.
Thus, by Theorem~\ref{T:sufficiency}, inequalities~\eqref{E:4.1-i},~\eqref{E:4.1-ii} and~\eqref{E:4.1-iii} hold, with $Y=X\sp m_{d,n}$, in cases (i), (ii), (iii), respectively. Consequently, embeddings~\eqref{E:optimal emb-i},~\eqref{E:optimal emb-ii} and~\eqref{E:optimal emb-iii} hold in the corresponding cases.
\\
It remains to prove that, under assumption \eqref{E:inf} or \eqref{E:infbar},
 the target space $X\sp m_{d,n}$ is optimal among rearrangement-invariant spaces. To this purpose, suppose that one of the embeddings~\eqref{E:optimal emb-i},~\eqref{E:optimal emb-ii} and~\eqref{E:optimal emb-iii} holds with $\|\cdot\|_{X\sp m_{d,n}(0,1)}$ replaced by some rearrangement-invariant function norm $\|\cdot\|_{Z(0,1)}$. Then, by Theorem~\ref{T:necessity},  inequality \eqref{E:assumption} holds with $Y(0,1)=Z(0,1)$, namely
\[
\left\|\int_{t^{\frac{n}{d}}}^1 f(s) s^{-1+\frac{m}{n}}\,ds\right\|_{Z(0,1)} \lesssim \|f\|_{X(0,1)}
\]
for every nonnegative function $f \in X(0,1)$.
The argument showing that~\eqref{E:assumption} implies~\eqref{E:embedding} in the proof of Theorem~\ref{T:sufficiency} tells us that
\begin{equation*}
X'_\sigma(0,1) \to Z(0,1),
\end{equation*}
where $\|\cdot\|_{X'_\sigma(0,1)}$ is defined as in~\eqref{E:X-sigma}. Since   $(X\sp m_{d,n})'(0,1)=X_\sigma(0,1)$, this yields $X\sp m_{d,n}(0,1)\to Z(0,1)$, thus establishing the optimality of the space $X\sp m_{d,n}$.
\end{proof}

%

\section{Main results -- slowly decaying measures}\label{low}

We are concerned here with $d$-Frostman measures for $d \in (0, n-m)$. As in the fast decaying regime considered in the previous section, our main purpose is to obtain reduction principles for Sobolev trace embeddings to Hardy-type inequalities. Though still necessary in the present range of values of $d$ (for measures decaying exactly like the power $d$ at some point) in view of Theorem \ref{T:necessity},   the Hardy inequality \eqref{E:assumption} is not anymore sufficient for the embedding of $W^{m}X(\o)$ into $Y(\o, \mu)$, or $Y(\overline \o, \mu)$. For instance, if $\mu_2$ is given by \eqref{mu2}, then inequality \eqref{E:assumption} can be shown to hold for $X(0,1)=Y(0,1)=L^{\frac{n-d}m, 1}(0,1)$, and yet $W^mL^{\frac{n-d}m, 1} (\o)$ is not embedded into $ L^{\frac{n-d}m, 1}(\o, \mu_2)$.
\par What renders the problem even more delicate now is the fact that sufficient conditions, depending only on $d \in (0, n-m)$, for trace embeddings of $W^{m}X(\o)$,
which are also necessary,   cannot be given. This is demonstrated by  examples \eqref{gen62} and \eqref{gen63} exhibited in Subsection \ref{borderslow}. They
 involve the same Sobolev domain space, and measures in the target space fulfilling the same couple of conditions \eqref{E:sup} and \eqref{E:inf}, but have a different optimal target norm.
\par In the situation at hand, we propose two versions of reduction principles, that can be well suited to deal with different classes of rearrangement-invariant norms. Their statements share  additional features with respect to Theorem \ref{T:sufficiency}. First, an extra Hardy-type inequality besides \eqref{E:assumption} is required, the latter being not sufficient on its own. Second, the target function norm $\|\cdot\|_{Y(0,1)}$ is a priori assumed  to be, in some proper sense, at least as strong as the Lebesgue function norm $\|\cdot\|_{L^{\frac{n-d}m}(0,1)}$, which, as shown by Theorem \ref{L:endpointsub} and Proposition \ref{sharpness},  naturally appears in the  weakest possible  Sobolev trace embedding for $W^mX(\o)$.
\par
As a final result in this section, we show that, under the additional structure assumption that the measure has a radially decreasing density with respect to Lebesgue measure, the conclusions are the same as in the case when $d\in [n-m, n]$. Namely,  the Hardy inequality \eqref{E:assumption} is necessary and sufficient for the embedding of $W^{m}X(\o)$ into $Y(\o, \mu)$, or $Y(\overline \o, \mu)$, and the optimal target norm is hence the one given by Theorem \ref{T:optimal}.  A borderline case  of this results has already been stated in Proposition \ref{P:radial-improvement}.

\smallskip

\par In the first version of our general reduction principle, which is the subject of the next theorem, the a priori assumption on the target function norm amounts to requiring that it has the form $\|\cdot\|_{Y^{\langle{\frac{n-d}m}\rangle}(0,1)}$, defined as in \eqref{Yalpha}, for some rearrangement-invariant function norm $\|\cdot\|_{Y (0,1)}$. Note that the choice $\|\cdot\|_{Y (0,1)}$=$\|\cdot\|_{L^1 (0,1)}$, the weakest rearrangement-invariant function norm, yields exactly $\|\cdot\|_{Y^{\langle{\frac{n-d}m}\rangle}(0,1)}= \|\cdot\|_{L^{\frac{n-d}m}(0,1)}$, up to equivalent norms. This follows, e.g., from  \cite[Theorem 4.1]{CPSS}.
%
%

\begin{theorem}{\rm{\bf [Reduction principle: case $0<d < n-m$. First version]}}\label{T:sufficiency-sub-main}
Let $\Omega$ be an open set with finite Lebesgue measure in $\rn$, $n\geq 2$, let $m\in \N$, with $m<n$, and  let $d \in (0,n-m)$. Let $\|\cdot\|_{X(0,1)}$ and $\|\cdot\|_{Y(0,1)}$ be rearrangement-invariant function  norms. Assume that there exists a constant
$C_1$ such that
\begin{equation}\label{E:two-operators1}
\left\|\int_{t^{\frac{n}{d}}}^1 f(s) s^{-1+\frac{m}{n}}\,ds\right\|_{Y(0,1)}  \leq C_1 \|f\|_{  X(0,1)}
\end{equation}
and
\begin{equation}\label{E:two-operators2} \left\| t\sp{-\frac{m}{n-d}}\int_0\sp{t\sp{\frac{n}{d}}}f(s)s\sp{-1+\frac{m}{n-d}}\,ds  \right\|_{Y(0,1)} \leq C_1 \|f\|_{  X(0,1)}
\end{equation}
for every nonnegative non-increasing function $f\in  X(0,1)$.
\\
\textup{(}i\textup{)} Let $\mu$ be a finite Borel measure on $\Omega$ fulfilling \eqref{E:sup}. Then
\begin{equation}\label{E:sobolev_embeddingsub-i}
\|u\|_{Y^{\langle{\frac{n-d}m}\rangle}({\Omega},\mu)} \leq C_2 \|\nabla\sp mu\|_{X(\Omega)}
\end{equation}
for some constant $C_2$ and every $u\in W^{m}_0X(\Omega)$,
\\
\textup{(}ii\textup{)} Assume, in addition,  that $\Omega$ is bounded and  has the cone property. Let   $\mu$ be a Borel measure on $\o$ fulfilling \eqref{E:sup}. Then
\begin{equation}\label{E:sobolev_embeddingsub-ii}
\|u\|_{Y^{\langle{\frac{n-d}m}\rangle}({\Omega},\mu)} \leq C_2 \|u\|_{W^mX(\Omega)}
\end{equation}
for some constant $C_2$ and every $u\in W^{m}X(\Omega)$.
\\
\textup{(}iii\textup{)} Assume that  $\Omega$ is a bounded Lipschitz domain. Let $\mu$ be a Borel measure on $\overline{\Omega}$ fulfilling \eqref{E:supbar}.  Then
\begin{equation}\label{E:sobolev_embeddingsub-iii}
\|u\|_{Y^{\langle{\frac{n-d}m}\rangle}(\overline{\Omega},\mu)} \leq C_2 \|u\|_{W^mX(\Omega)}
\end{equation}
for some constant $C_2$ and every $u\in W^mX(\Omega)$.
\\ In particular, the norm $\|\, \cdot \,\|_{Y^{\langle{\frac{n-d}m}\rangle}}$ can be replaced by the norm $\|\, \cdot \,\|_{Y}$ in inequalities \eqref{E:sobolev_embeddingsub-i}--\eqref{E:sobolev_embeddingsub-iii}.
\end{theorem}

The  alternative form of our reduction principle is stated in the following result. Besides the   Hardy-type operator appearing in inequalities \eqref{E:assumption} and  \eqref{E:two-operators1}, an unconventional operator involving the product of powers of two Hardy-type operators comes now  into play.
The fact that only rearrangement-invariant target function norms,  which are not weaker  than $\|\cdot \|_{L^{\frac{n-d}m}(0,1)}$, are allowed is now prescribed in a different fashion. Specifically, they are required  to have the form
$ \|\cdot\|_{Z^{\{\frac{n-d}m\}}(0,1)}
$, defined as in \eqref{Zp},
for some rearrangement-invariant function norm $\|\cdot \|_{Z(0,1)}$. Observe that, again, the choice $\|\cdot\|_{Z (0,1)}$=$\|\cdot\|_{L^1 (0,1)}$ results in $ \|\cdot\|_{Z^{\{\frac{n-d}m\}}(0,1)}= \|\cdot\|_{L^{\frac{n-d}m}(0,1)}$.

\begin{theorem}{\rm{\bf [Reduction principle: case $0<d<n-m$. Second version]}}\label{T:product-of-operators}
Let $\Omega$ be an open set with finite Lebesgue measure in $\rn$, $n\geq 2$, let $m\in \N$, with $m<n$, and let $d \in (0,n-m)$. Let $\|\cdot\|_{X(0,1)}$ and $\|\cdot\|_{Y(0,1)}$ be rearrangement-invariant function  norms. Assume that there exists  a constant $C_1$ such that
\begin{equation}\label{E:assumptionsub1}
\left\|\int_{t^{\frac{n}{d}}}^1 f(s) s^{-1+\frac{m}{n}}\,ds\right\|_{Z^{\{\frac{n-d}m\}}(0,1)} \leq C_1 \|f\|_{X(0,1)}
\end{equation}
and
\begin{equation}\label{E:assumptionsub2}
\left\|\bigg(\int_{t^{\frac{n}{d}}}^1 f(s) s^{-1+\frac{m}{n}}\,ds\bigg)^{\frac m{n-d}}   \bigg(t^{-\frac m{n-d}}\int _0^{t^{\frac n d}} f(s) s^{-1+\frac m{n-d}}\, ds\bigg)^{1-\frac m{n-d}}\right\|_{Z^{\{\frac{n-d}m\}}(0,1)} \leq C_1 \|f\|_{X(0,1)}
\end{equation}
for every nonnegative non-increasing function $f\in X(0,1)$.
\\
\textup{(}i\textup{)} Let $\mu$ be a finite Borel measure on $\Omega$ fulfilling \eqref{E:sup}. Then there exists a constant $C_2$ such that
\begin{equation}\label{E:sobolev-alternate-1}
\|u\|_{Z^{\{\frac{n-d}m\}}({\Omega},\mu)} \leq C_2 \|\nabla ^m u\|_{X(\Omega)}
\end{equation}
for every $u\in W^{m}_0X(\Omega)$.
\\
\textup{(}ii\textup{)} Assume, in additon, that $\Omega$ is bounded and  has the cone property. Let $\mu$ be a Borel measure on $\o$ fulfilling \eqref{E:sup}. Then there exists a constant $C_2$ such that
\begin{equation}\label{E:sobolev-alternate-2}
\|u\|_{Z^{\{\frac{n-d}m\}}({\Omega},\mu)} \leq C_2 \|u\|_{W^mX(\Omega)}
\end{equation}
for every $u\in W^{m}X(\Omega)$.
\\
\textup{(}iii\textup{)} Assume that  $\Omega$ is a bounded Lipschitz domain. Let $\mu$ be a Borel measure on $\overline{\Omega}$ fulfilling \eqref{E:supbar}.  Then there exists a constant $C_2$ such that
\begin{equation}\label{E:sobolev-alternate-3}
\|u\|_{Z^{\{\frac{n-d}m\}}(\overline{\Omega},\mu)} \leq C_2 \|u\|_{W^mX(\Omega)}
\end{equation}
for every $u\in W^mX(\Omega)$.
\end{theorem}

Our last statement concerns the enhanced results mentioned above  for Frostman measures with a radially decreasing density.

\begin{theorem}{\rm{\bf [Improved reduction principle for radially decreasing densities]}}\label{T:radial}
Let $\Omega$ be an open {set} with finite Lebesgue measure in $\rn$, $n\geq 2$, let $m\in \N$, with $m<n$, and $d \in (0, n-m)$. Let $\|\cdot\|_{X(0,1)}$ and $\|\cdot\|_{Y(0,1)}$ be rearrangement-invariant {function} norms. Assume that there exists a constant
$C_1$ such that
\begin{equation}\label{E:assumption-radial}
\left\|\int_{t^{\frac{n}{d}}}^1 f(s) s^{-1+\frac{m}{n}}\,ds\right\|_{Y(0,1)} \leq C_1 \|f\|_{  X(0,1)}
\end{equation}
for every nonnegative non-increasing function $f\in  X(0,1)$.
\\
\textup{(}i\textup{)} If $\mu$ is a finite Borel measure on $\o$ fulfilling \eqref{E:sup}
and having the form \eqref{E:mu-radial} for  some  $x_0\in\Omega$,
 then
\begin{equation}\label{E:radial-general-0}
\|u\|_{Y({\Omega},\mu)} \leq C_2\|\nabla ^m u\|_{X(\Omega)}
\end{equation}
for some constant $C_2$ and every $u\in W^{m}_0X(\Omega)$.
\\
\textup{(}ii\textup{)}
Assume, in addition, that  $\Omega$ is  bounded and  has the cone property. If $\mu$ is a Borel measure on $\o$ fulfilling \eqref{E:sup}
and having the form \eqref{E:mu-radial} for  some  $x_0\in\Omega$,
then
\begin{equation}\label{E:radial-general}
\|u\|_{Y({\Omega},\mu)} \leq C_2 \|u\|_{W^mX(\Omega)}
\end{equation}
for some constant $C_2$ and every $u\in W^{m}X(\Omega)$.
\\
\textup{(}iii\textup{)}
Assume that  $\Omega$ is a bounded Lipschitz domain. If $\mu$ is a Borel measure on $\o$ fulfilling \eqref{E:supbar}
and having the form \eqref{E:mu-radial} for  some  $x_0\in \overline \Omega$,
 then
\begin{equation}\label{E:radial-general-bar}
\|u\|_{Y(\overline{\Omega},\mu)} \leq C_2 \|u\|_{W^mX(\Omega)}
\end{equation}
for some constant $C_2$ and every $u\in W^mX(\Omega)$.
\\ Conversely, assume that any of properties  \textup{(}i\textup{)}, \textup{(}ii\textup{)}, \textup{(}iii\textup{)} holds. If, in addition, $\mu$ satisfies condition \eqref{E:inf} in cases \textup{(}i\textup{)} and \textup{(}ii\textup{)}, or \eqref{E:infbar} in case \textup{(}iii\textup{)}, then inequality \eqref{E:assumption-radial} holds.
\\ In particular, inequalities \eqref{E:radial-general-0}--\eqref{E:radial-general-bar} hold with $Y=X_{d,n}^m$, where $\|\,\cdot \,\|_{X_{d,n}^m(0,1)}$ is the function norm given by \eqref{E:trace_opt_norm}, and the latter function norm is optimal in the  pertaining inequality if, in addition, $\mu$ satisfies condition \eqref{E:inf} in cases \textup{(}i\textup{)} and \textup{(}ii\textup{)}, or \eqref{E:infbar} in case \textup{(}iii\textup{)}.
\end{theorem}

\begin{proof}[Proof of Theorem \ref{T:sufficiency-sub-main}]
We focus on part  (iii), the proofs of parts (i) and (ii) being analogous. To begin with, we claim  that, under assumption \eqref{E:two-operators2}, the trace operator $T_\mu$ is well defined on $W^mX(\o)$. This follows from the fact that, if  inequality \eqref{E:two-operators2} is fulfilled for some rearrangement-invariant function norm $Y(0,1)$, then, by property \eqref{l1linf}, it is also satisfied with $Y(0, 1)=L^1(0,1)$. An application of  inequality \eqref{E:two-operators2} with this choice of $Y(0, 1)$ and with $f=f^*$, and the use of Fubini's theorem tell us that
$X(0,1) \to L^{\frac{n-d}m,1}(0,1)$, whence  $W^mX(\o) \to W^m L^{\frac{n-d}m,1}(\o)$.
Therefore, our claim follows from Theorem \ref{L:endpointsub}, part (iii).
\\ Next, by the same
 theorem
we have that
\begin{equation}\label{E:left-trace-embedding}
W\sp mL\sp{\frac{n-d}{m},1}(\Omega)\to L\sp{\frac{n-d}{m}}(\overline{\Omega},\mu).
\end{equation}
Moreover, embedding~\eqref{E:right-trace-embedding} still holds. From the endpoint estimates~\eqref{E:left-trace-embedding} and~\eqref{E:right-trace-embedding}, and property \eqref{K} applied when $T$ is the trace operator $T_\mu$, one can deduce that
\begin{equation}\label{gen25}
K(u,t;L\sp{\frac{n-d}{m}}(\overline{\Omega},\mu),L\sp{\infty}(\overline{\Omega},\mu))
\lesssim
K(u,Ct;W\sp mL\sp{\frac{n-d}{m},1}(\Omega),W\sp mL\sp{\frac{n}{m},1}(\Omega))\quad\text{for $t\in(0,\infty)$,}
\end{equation}
for some constant $C$ and  every $u\in W\sp{m}L\sp{\frac{n-d}{m},1}(\Omega)$. Here, and in the remaining part of this proof, the constants in the relation $\lesssim$, and all other constants, depend only on $n,m,d$, $\mu$ and $\o$.
On
applying formulas~\eqref{E:fx2.14},~\eqref{E:fx2.13} and a change of variables, we infer from inequality \eqref{gen25} that
\begin{align*}
 \left(\frac{1}{t}\int_0^{ct} u^*_\mu(\mu(\overline{\Omega})s)^{\frac{n-d}{m}}\,ds\right)^{\frac{m}{n-d}}
 &\lesssim \int_{t^{\frac{n}{d}}}^\infty |D^mu|^*(|\Omega| s) s^{-1+\frac{m}{n}}\,ds\\
 &+t^{-\frac{m}{n-d}} \int_0^{t^{\frac{n}{d}}} |D^mu|^*(|\Omega| s) s^{-1+\frac{m}{n-d}}\,ds \quad \textup{for $t\in (0,\infty)$,}
\end{align*}
for some constant $c$.
Since $|D^mu|^*$ vanishes outside $(0,|\Omega|)$, we hence deduce that
\begin{align}\label{E:pointwise}
 \left(\frac{1}{t}\int_0^{ct} u^*_\mu(\mu(\overline{\Omega})s)^{\frac{n-d}{m}}\,ds\right)^{\frac{m}{n-d}}
 &\lesssim \int_{t^{\frac{n}{d}}}^1 |D^mu|^*(|\Omega| s) s^{-1+\frac{m}{n}}\,ds\\
 \nonumber
 & \quad +t^{-\frac{m}{n-d}} \int_0^{t^{\frac{n}{d}}} |D^mu|^*(|\Omega| s) s^{-1+\frac{m}{n-d}}\,ds \quad \textup{for $t\in (0,1)$.}
\end{align}
Define the operators $S_1$ and $S_2$ as
\begin{equation}\label{S12}
S_1f(t)=\int_{t\sp{\frac{n}{d}}}\sp{1}f(s)s\sp{-1+\frac{m}{n}}\,ds \quad\text{and}\quad
S_2f(t)=t\sp{-\frac{m}{n-d}}\int_0\sp{t\sp{\frac{n}{d}}}f(s)s\sp{-1+\frac{m}{n-d}}\,ds\quad\text{for $t\in(0,1)$,}
\end{equation}
for $f\in\Mpl(0,1)$. From inequality~\eqref{E:pointwise} and the boundedness of the dilation operator on rearrangement-invariant spaces, one obtains that
$$
\left\|(u^*_{\mu}(\mu(\overline{\Omega})\, \cdot \, )\sp{\frac{n-d}{m}})\sp{**}(t)\sp{\frac{m}{n-d}}\right\|_{Y(0,1)}
\lesssim \left\|(S_1+S_2)(|D\sp mu|\sp*(|\Omega|t))\right\|_{Y(0,1)}.
$$
Now, thanks to inequalities  \eqref{E:two-operators1} and \eqref{E:two-operators2},
\[
\left\|(u^*_{\mu}(\mu(\overline{\Omega}) \, \cdot \,)\sp{\frac{n-d}{m}})\sp{**}(t)\sp{\frac{m}{n-d}}\right\|_{Y(0,1)}
\lesssim
\left\|D\sp mu\right\|_{X(\Omega)}
\]
for $u\in W\sp{m}L\sp{\frac{n-d}{m},1}(\Omega)$, whence inequality~\eqref{E:sobolev_embeddingsub-iii} follows, via the very definition of the norm $\|\cdot\|_{Y^{\langle\frac{n-d}{m}\rangle}(0,1)}$.
\\
The  assertion about the replacement of
$Y^{\langle\frac{n-d}{m}\rangle}$ with $Y$ is a consequence of inequality \eqref{gen36}.
%
%
\end{proof}

\smallskip

\begin{proof}[Proof of Theorem \ref{T:product-of-operators}] As in the proof of Theorem \ref{T:sufficiency-sub-main}, we  carry out the argument in case (iii). As a preliminary step, let us show that the trace operator $T_\mu$ is actually well defined on $W^mX(\o)$ under assumption \eqref{E:assumptionsub2}. If the latter assumption is fulfilled for some rearrangement-invariant norm $\|\cdot\|_{Z (0,1)}$, then, by \eqref{l1linf}, it is a fortiori satisfied with $\|\cdot\|_{Z (0,1)}=\|\cdot\|_{L^1 (0,1)}$. This choice yields $\|\cdot\|_{Z^{\{\frac{n-d}m\}} (0,1)} = \|\cdot\|_{L^{\frac{n-d}m} (0,1)}$, and inequality \eqref{E:assumptionsub2}, after the use of Fubini's theorem, reads
\begin{align}\label{gen37}
\bigg(\int _0^1 f^*(r) r^{-1+\frac mn}\int _0^{r^{\frac dn}}t^{-1+\frac m{n-d}}\bigg(\int _0^{t^{\frac nd}}f^*(s)s^{-1+\frac m{n-d}}\, ds\bigg)^{-1+\frac{n-d}m }\, dt\, dr\bigg)^{\frac m{n-d}} \leq C_1 \|f\|_{X(0,1)}
\end{align}
for $f \in\Mpl(0,1)$. On the other hand,
\begin{align}\label{gen38}
\bigg(\int _0^1& f^*(r) r^{-1+\frac mn}\int _0^{r^{\frac dn}}t^{-1+\frac m{n-d}}\bigg(\int _0^{t^{\frac nd}}f^*(s)s^{-1+\frac m{n-d}}\, ds\bigg)^{-1+\frac{n-d}m}\, dt\, dr\bigg)^{\frac m{n-d}}
\\ \nonumber & \geq \bigg(\int _0^1 f^*(r) r^{-1+\frac mn}\int _ {(r/2)^{\frac dn}}^{r^{\frac dn}}t^{-1+\frac m{n-d}}\bigg(\int _0^{t^{\frac nd}}f^*(s)s^{-1+\frac m{n-d}}\, ds\bigg)^{-1+\frac{n-d}m}\, dt\, dr\bigg)^{\frac m{n-d}}
\\ \nonumber & \geq \bigg(\int _0^1 f^*(r) r^{-1+\frac m{n}-\frac dn +\frac{dm}{n(n-d)}}\big(r^{\frac  dn}  - (r/2)^{\frac dn}\big) \bigg(\int _0^{r/2}f^*(s)s^{-1+\frac m{n-d}}\, ds\bigg)^{-1+\frac{n-d}m}\, dr\bigg)^{\frac m{n-d}}
\\ \nonumber & \approx \bigg(\int _0^1 f^*(r) r^{-1+\frac m{n-d}}  \bigg(\int _0^{r/2}f^*(s)s^{-1+\frac m{n-d}}\, ds\bigg)^{-1+\frac{n-d}m}\, dr\bigg)^{\frac m{n-d}}
  \approx \|f\|_{L^{\frac{n-d}m,1}(0,1)}.
\end{align}
 Throughout this proof, the relations $\lq\lq \lesssim "$ and $\lq\lq \approx "$ hold up to constants depending on $n,m,d$, $\mu$ and $\o$.
Coupling inequalities \eqref{gen37} and \eqref{gen38} tells us that $X(0,1) \to L^{\frac{n-d}m,1}(0,1)$. Thus, $W^mX(\o) \to W^m L^{\frac{n-d}m,1}(\o)$, and hence $T_\mu$ is well defined by Theorem \ref{L:endpointsub}.
\\
Our point of departure in the proof of inequality \eqref{E:sobolev-alternate-3} is equation~\eqref{E:pointwise}.
Define the operator $P$  as
\begin{equation}\label{P}
Pf(t)= t\sp{-\frac{m}{n-d}}\int_0\sp t s\sp{-1+\frac{m}{n-d}}f(s)\,ds \qquad \hbox{for $t \in (0,1)$,}
\end{equation}
for $f \in \Mpl(0,1)$. Let $S_1$ and $S_2$ be the operators defined as in \eqref{S12}.
By Fubini's theorem, one can verify that
\begin{equation}\label{E:fubini-1}
P S_1 f(t)=\frac{n-d}{m}(S_1+S_2)f(t) \quad\text{for  $t\in(0,1)$,}
\end{equation}
for every $f\in\M_+(0,1)$.
Now, fix $u\in W^mX(\Omega)$ and set $g(t)=|D^mu|^*(|\Omega|t)$ for $t\in (0,1)$. A combination of equations ~\eqref{E:pointwise} and~\eqref{E:fubini-1} yields
\begin{equation}\label{E:average-right}
\left(\frac{1}{t}\int_0^{ct} u^*_\mu(\mu(\overline{\Omega})s)^{\frac{n-d}{m}}\,ds\right)^{\frac{m}{n-d}}
\lesssim
P S_1 g(t) \quad\text{for  $t\in(0,1)$,}
\end{equation}
for some constant $c$ independent of $u$.
Next, define the function $h\in\Mpl(0,1)$ by
\begin{equation}\label{E:g}
h(t)=\kappa(P S_1 g(t))\sp{1-\frac{m}{n-d}} (S_1 g(t))\sp{\frac{m}{n-d}} \quad\text{for $t\in(0,1)$},
\end{equation}
where $\kappa=(\frac{n-d}{m})\sp{\frac{m}{n-d}}$. Since the function on the right-hand side of equation \eqref{E:g} is non-increasing,  we have that $h=h^*$.
Raising both sides of   equality  \eqref{E:g} to the power $\frac{n-d}{m}$ and integrating the resultant equality yield
\[
\int_0\sp{t}h (s)\sp{\frac{n-d}{m}}\,ds
=
\left(\int_0\sp{t}
s\sp{-1+\frac{m}{n-d}}
\int_{s\sp{\frac{n}{d}}}\sp{1}g(r)r\sp{-1+\frac{m}{n}}\,dr\,ds\right)\sp{\frac{n-d}{m}} \quad \text{for $t\in(0,1)$}.
\]
The latter equality can be rewritten in the form
\begin{equation}\label{E:principle}
(h\sp{\frac{n-d}{m}})\sp{**}(t)\sp{\frac{m}{n-d}}=P S_1g(t) \quad\text{for $t\in(0,1)$}.
\end{equation}
Coupling equation
\eqref{E:average-right} with~\eqref{E:principle} tells us that
\[
\left(\frac{1}{t}\int_0^{ct} u^*_\mu(\mu(\overline{\Omega})s)^{\frac{n-d}{m}}\,ds\right)^{\frac{m}{n-d}}
\lesssim
(h\sp{\frac{n-d}{m}})\sp{**}(t)\sp{\frac{m}{n-d}} \quad\text{for $t\in(0,1)$},
\]
or, equivalently,
\begin{equation}\label{E:u-g-int}
\int_0\sp{ct}u\sp*_{\mu}(\mu(\overline{\Omega})s)\sp{\frac{n-d}{m}}\,ds
\lesssim
\int_0\sp{t}h(s)\sp{\frac{n-d}{m}}\,ds \quad\text{for $t\in(0,1)$}.
\end{equation}
Owing to inequality~\eqref{E:u-g-int}, property~\eqref{E:HLP} and the definition of the functional $\|\cdot\|_{Z^{\{\frac{n-d}{m}\}}(0,1)}$,
\[
\|u\sp*_{\mu}(\mu(\overline{\Omega})t)\|_{Z^{\{\frac{n-d}{m}\}}(0,1)}
\lesssim
\|h\|_{Z^{\{\frac{n-d}{m}\}}(0,1)}.
\]
By equation~\eqref{E:g}, this means that
\begin{equation}\label{gen30}
\|u\sp*_{\mu}(\mu(\overline{\Omega})t)\|_{Z^{\{\frac{n-d}{m}\}}(0,1)}
\lesssim
\|(P S_1g)\sp{1-\frac{m}{n-d}} (S_1 g)\sp{\frac{m}{n-d}}\|_{Z^{\{\frac{n-d}{m}\}}(0,1)}.
\end{equation}
Thanks to equation ~\eqref{E:fubini-1} again,
\begin{align}\label{gen31}
(P S_1 g(t))\sp{1-\frac{m}{n-d}} (S_1g(t))\sp{\frac{m}{n-d}}
&\approx
((S_1+S_2)g(t))\sp{1-\frac{m}{n-d}}(S_1g(t))\sp{\frac{m}{n-d}} \approx
S_1g(t)+(S_1g(t))\sp{\frac{m}{n-d}}(S_2g(t))\sp{1-\frac{m}{n-d}}
\end{align}
for $t\in(0,1)$. Assumptions \eqref{E:assumptionsub1} and \eqref{E:assumptionsub2}, combined with equations \eqref{gen30}, \eqref{gen31} and \eqref{gen21}, imply that
\begin{align*}
\|u \|_{Z^{\{\frac{n-d}{m}\}}(\overline \o, \mu)} & \approx \|u\sp*_{\mu}(\mu(\overline{\Omega})t)\|_{Z^{\{\frac{n-d}{m}\}}(0,1)} \lesssim
\|S_1g+(S_1g)\sp{\frac{m}{n-d}}(S_2g)\sp{1-\frac{m}{n-d}}\|_{Z^{\{\frac{n-d}{m}\}}(0,1)}\\ &  \lesssim
\| |D\sp mu|\sp*\|_{X(0,1)} \approx \| D^m u\|_{X(\o)} \approx \|u\|_{W^mX(\o)}.
\end{align*}
Hence, inequality  \eqref{E:sobolev-alternate-3} follows.
\end{proof}

\begin{proof}[Proof of Theorem \ref{T:radial}]
We begin by showing that  functional $\|\cdot\|_{Z(0,1)}$, defined by
\begin{equation}\label{E:Z}
\|f\|_{Z(0,1)}=\|f^*(t\sp{\frac nd})\|_{Y(0,1)}
\end{equation}
for $f \in \Mpl(0,1)$,
is a rearrangement-invariant function norm. In order to prove the triangle
inequality, fix any  $f,g \in \Mpl(0,1)$. The following chain holds:
\begin{align*}
\|f+g\|_{Z(0,1)}
&=\|(f+g)^*(t\sp{\frac{n}{d}})\|_{Y(0,1)}
=\sup_{\|h\|_{Y'(0,1)}\leq 1} \int_0^1 (f+g)^*(t\sp{\frac{n}{d}}) h^*(t)\,dt\\
&= {\frac{d}{n}}  \sup_{\|h\|_{Y'(0,1)}\leq 1} \int_0^1 (f+g)^*(t) h^*(t\sp{\frac{d}{n}}) t\sp{-1+\frac{d}{n}}\,dt\\
&\leq  {\frac{d}{n}} \sup_{\|h\|_{Y'(0,1)}\leq 1} \int_0^1 f^*(t) h^*(t\sp{\frac{d}{n}}) t\sp{-1+\frac{d}{n}}\,dt +  {\frac{d}{n}} \sup_{\|h\|_{Y'(0,1)}\leq 1} \int_0^1 g^*(t) h^*(t\sp{\frac{d}{n}}) t\sp{-1+\frac{d}{n}}\,dt\\
&= \sup_{\|h\|_{Y'(0,1)}\leq 1} \int_0^1 f^*(t\sp{\frac{n}{d}}) h^*(t)\,dt
+\sup_{\|h\|_{Y'(0,1)}\leq 1} \int_0^1 g^*(t\sp{\frac{n}{d}}) h^*(t)\,dt\\
&=\|f^*(t\sp{\frac{n}{d}})\|_{Y(0,1)} +\|g^*(t\sp{\frac{n}{d}})\|_{Y(0,1)}
=\|f\|_{Z(0,1)} + \|g\|_{Z(0,1)}.
\end{align*}
Note that
 the second equality is due to equation \eqref{X''}, and the inequality   to the fact that the functional
$$
f \mapsto \int_0^1 f^*(t) h^*(t\sp{\frac{d}{n}}) t\sp{-1+\frac{d}{n}}\,dt
$$
is subadditive for any fixed $h\in \mathcal M_+(0,1)$. The latter property is in its turn a consequence of properties \eqref{subadd} and \eqref{E:hardy-lemma}.
%
The remaining properties in \textup{(P1)}, as well as properties \textup{(P2)}, \textup{(P3)}, \textup{(P4)} and \textup{(P6)} of the definition of a rearrangement-invariant function norm, follow easily from the definition of $\|\cdot\|_{Z(0,1)}$. Finally, property  \textup{(P5)} holds since
%
%
$$
\int_0^1 f^*(t)\,dt \lesssim \|f^*\|_{Y(0,1)} \leq \|f^*(t\sp{\frac{n}{d}})\|_{Y(0,1)} =\|f\|_{Z(0,1)}
$$
for $f \in \Mpl(0,1)$, where the first inequality is due to property  \textup{(P5)} for the function norm $ \|\cdot\|_{Y(0,1)}$, and the second one to the inequality $t^{\frac nd} \leq t$ for $t \in (0,1)$.
Throughout this proof, the constants involved in the relations $\lq\lq \approx "$ and $\lq\lq \lesssim "$, as well as  all other constants, depend on $n,m,d$, $\mu$ and $\o$.
The fact that
 $\|\cdot\|_{Z(0,1)}$ is a rearrangement-invariant function norm is thus established.
\\ From now on, we focus on part (iii), the proof of parts (i) and (ii) being analogous, and even simpler. We claim that
\begin{equation}\label{E:weighted_embedding}
\|u\|_{Y\left(\overline \Omega,\mu\right)}\lesssim\|u\|_{Z(\Omega)}
\end{equation}
for every $u \in Z(\Omega)$.
To prove equation~\eqref{E:weighted_embedding}, first notice that
\begin{equation}\label{E:measure-estimate}
\mu(E)\lesssim|E|\sp{\frac{d}{n}}
\end{equation}
for every Borel set   $E\subset\o$. Indeed, on denoting by $B$ a ball
 centered at the point $x_0$ and satisfying $|E|=|B|$, one has that
\[
\mu(E)
=
\int_Eg(|x-x_0|)\,dx
\leq
\int_Bg(|x-x_0|)\,dx
\lesssim
\int_{B\cap \overline \o}g(|x-x_0|)\,dx
\lesssim |B\cap \overline \o|^{\frac dn}
\leq |B|\sp{\frac{d}{n}}
=|E|\sp{\frac{d}{n}},
\]
where the first inequality holds by the Hardy-Littlewood inequality \eqref{E:HL} and the monotonicity of the function $g$, the second one since $x_0\in \overline \o$ and $\o$ is a Lipschitz domain, and the third one by assumption \eqref{E:supbar}.
 Hence,
\begin{align*}
\mu(\{x\in \Omega: |u(x)|>\varrho\})
\lesssim |\{x\in \Omega: |u(x)|>\varrho\}|\sp{\frac{d}{n}} \qquad \hbox{for $\varrho \in (0, \infty)$,}
\end{align*}
for $u \in \Mpl(\overline \o, \mu)$.
Thus,
\begin{align}\label{E:rearrangement}
u^*_\mu(t)
\leq u^*(c t\sp{\frac{n}{d}}) \qquad \hbox{for $t \in (0, \infty)$,}
\end{align}
for some constant $c$ and for every $u \in \Mpl(\overline \o, \mu)$. From inequality
 \eqref{E:rearrangement} and the boundedness of the dilation operator on rearrangement-invariant spaces we deduce  that
\begin{align*}
\|u\|_{Y(\overline \Omega,\mu)}
&=\|u^*_\mu(\mu(\Omega)t)\|_{Y(0,1)}
\leq \| u^*(c (\mu(\Omega) t)\sp{\frac{n}{d}})\|_{Y(0,1)}
\lesssim \|u^*\|_{Z(0,1)} \lesssim\|u^*(|\Omega|t)\|_{Z(0,1)}
=\|u\|_{Z(\Omega)}.
\end{align*}
This establishes~\eqref{E:weighted_embedding}.
\\
Finally, by the very definition of the function norm $ \|\cdot \|_{Z(0,1)}$,
\begin{equation}\label{E:comparison-of-norms}
\left\|\int_{t}^1 f(s) s^{-1+\frac{m}{n}}\,ds\right\|_{Z(0,1)}
=\left\|\int_{t\sp{\frac{n}{d}}}^1 f(s) s^{-1+\frac{m}{n}}\,ds\right\|_{Y(0,1)}
\end{equation}
for $f\in\Mpl(0,1)$.
Owing to assumption \eqref{E:assumption-radial} and Theorem \ref{T:sufficiency} applied to the case when $\mu$ is Lebesgue measure, and hence $d=n$,
\begin{equation}\label{E:KP}
\|u\|_{Z(\Omega)} \lesssim \|u\|_{W^mX(\Omega)}
\end{equation}
for every $u \in W^mX(\o)$. Coupling inequalities \eqref{E:weighted_embedding}  and~\eqref{E:KP} yields  embedding~\eqref{E:radial-general-bar}.
\\ The necessity of condition \eqref{E:assumption-radial}
 under assumption \eqref{E:infbar} holds (for any $d$-Frostman measure and any $d\in (0, n]$) by
 Theorem~\ref{T:necessity}.
\\ With  the sufficiency and necessity of condition \eqref{E:assumption-radial} at disposal, the assertions about  the function norm $\|\cdot\|_{X_{n,d}^m(0,1)}$ follow via the same argument as in the proof of Theorem \ref{T:optimal}.
\end{proof}

\begin{proof}[Proof of Proposition
\ref{P:radial-improvement}] Owing to Theorem \ref{T:radial}, it suffices to observe that, if $X(0,1)= L^{\frac{n-d}m,1}(0,1)$, then $X_{n,d}^m(0,1)= L^{\frac{n-d}m,1}(0,1)$ as well, up to equivalent norms. This equality is a special case of the last but one equation in the proof of \cite[Theorem 5.1]{CP-Trans}.
\end{proof}

\section{Optimal embeddings for classical Sobolev spaces}\label{S:last}

In this final section    results established in the previous sections are exploited  to prove an embedding for the standard Sobolev space $W^{m,p}_0(\o)$, or $W^{m,p}(\o)$, into an optimal rearrangement-invariant target space with respect to a $d$-Frostman measure. This is the content of the next theorem, which enhances classical embeddings for measures by Adams \cite{Adams2, Adams3} for $p>1$, and Maz'ya \cite{Mazya543, Mazya548} for $p=1$, and carries over to the case of Frostman measures  results by O'Neil \cite{Oneil} and Peetre \cite{Peetre} for $1\leq p < \frac nm$ and by Br\'ezis-Wainger \cite{BW} for $p= \frac nm$. A related result, under  assumptions of a different nature on the measure, can be found in \cite{costea-mazya}.
%
%

\begin{theorem}{\rm{\bf [Optimal range space for classical Sobolev spaces]}}\label{T:main-example} Let $\Omega$ be a   domain with finite measure in $\Rn$, $n\geq 2$, and let $m\in \N$  with
$m<n$. Assume that either $d\in[n-m,n]$ and $p\in[1,\infty]$, or $d\in(0,n-m)$ and $p\in(\frac{n-d}{m},\infty]$.
\\
\textup{(}i\textup{)} Assume that $\mu$ is a finite Borel measure on ${\Omega}$ fulfilling \eqref{E:sup}. Then
\begin{equation}\label{E:optimal-leb-emb-i}
W^{m,p}_0(\Omega) \to
\begin{cases}
    L\sp {\frac{dp}{n-mp},p}(\Omega, \mu) &\text{if $p<\frac{n}{m}$}\\
    L\sp {\infty,p,-1}(\Omega, \mu) &\text{if $p=\frac{n}{m}$}\\
    L\sp {\infty}(\Omega, \mu) &\text{if $p>\frac{n}{m}$}.
  \end{cases}
\end{equation}
\\
\textup{(}ii\textup{)} Assume, in addition, that $\Omega$ has the cone property. Let
$\mu$ be a Borel measure on ${\Omega}$ fulfilling \eqref{E:sup}.
Then embedding~\eqref{E:optimal-leb-emb-i} holds with $W^{m,p}_0(\Omega)$ replaced by $W^{m,p}(\Omega)$.
\\
\textup{(}iii\textup{)} Assume that $\Omega$ is a bounded Lipschitz domain.  Let $\mu$ be a Borel measure on $\overline{\Omega}$ fulfilling \eqref{E:supbar}.
Then embedding~\eqref{E:optimal-leb-emb-i} holds with $W^{m,p}_0(\Omega)$ replaced by $W^{m,p}(\Omega)$ and $\Omega$ replaced by $\overline{\Omega}$ on the right-hand side.
\\
Moreover, if~\eqref{E:inf} is satisfied, then the target spaces in \textup{(}i\textup{)} and \textup{(}ii\textup{)}
 are optimal among all rearrangement-invariant spaces. If~\eqref{E:infbar} is satisfied, then the target spaces   in \textup{(}iii\textup{)} are optimal among all rearrangement-invariant spaces.
\end{theorem}


\begin{proof}
Assume first that $d\in[n-m,n]$, and   fix $p\in[1,\infty]$. By Theorem \ref{T:optimal}, embedding \eqref{E:optimal-leb-emb-i} will follow if we show that
\begin{equation}\label{gen40}
(L^p)_{n,d}^m(0,1)=
\begin{cases}
    L\sp {\frac{dp}{n-mp},p}(0,1) &\text{if $p<\frac{n}{m}$}\\
    L\sp {\infty,p,-1}(0,1) &\text{if $p=\frac{n}{m}$}\\
    L\sp {\infty}(0,1) &\text{if $p>\frac{n}{m}$},
  \end{cases}
\end{equation}
up to equivalent norms,
where $(L^p)_{n,d}^m(0,1)$ denotes the space associated with $L^p(0,1)$ as in \eqref{E:trace_opt_norm}. If $p\in[1,\frac{n}{m}]$, then
%
%
\begin{align}\label{E:S1-char}
 \|f\|_{((L^{p})^{m}_{d,n})'(0,1)}
    =\bigg\|t\sp{-1+\frac mn}\int _0^{t^{\frac d n}} f\sp{*}(s)ds\bigg\|_{L^{p'}(0,1)}
= \|f\|_{L^{(\frac{dp}{dp-(n-mp)},p')}(0,1)}.
\end{align}
Hence, the cases when $1 \leq p<\frac{n}{m}$ and $p=\frac{n}{m}$ in \eqref{gen40} follow via
equations~\eqref{E:lorentz-assoc} and~\eqref{v}.
If $p>\frac{n}{m}$, then
\begin{equation}\label{gen43}
  \|f\|_{((L^{p})^{m}_{d,n})'(0,1)}=\bigg\|t\sp{-1+\frac mn}\int _0^{t^{\frac d n}} f\sp{*}(s)ds\bigg\|_{L^{p'}(0,1)}
  \leq
  \|f\|_{L^{1}(0,1)}\Big\|t\sp{-1+\frac{m}{n}}\Big\|_{L^{p'}(0,1)}
  \approx \|f\|_{L^{1}(0,1)},
\end{equation}
whence $L^{1}(0,1)\to((L^{p})^{m}_{d,n})'(0,1)$, where the relation $\lq\lq \approx "$ holds up to constants depending on $n,m,d,p$. Owing to~\eqref{l1linf}, this implies that $L^{1}(0,1)=((L^{p})^{m}_{d,n})'(0,1)$, whence, by \eqref{X''},
 $(L^{p})^{m}_{d,n}(0,1)=L^{\infty}(0,1)$. The case when $p>\frac nm$ in \eqref{gen40} is thus also settled.
\\ Assume now that
 $d\in(0,n-m)$. Let $S_1$ and $S_2$ be the operators defined by \eqref{S12}. Observe that
$$S_1' f(t) = t\sp{-1+\frac mn}\int _0^{t^{\frac d n}} f(s)ds \quad \hbox{for $t \in (0,1)$,}$$
for $f \in \Mpl(0,1)$.
If $p\in[1,\frac{n}{m}]$,  then the second equality in \eqref{E:S1-char} and inequality~\eqref{E:HL}  tell us that
\begin{equation}\label{gen42}
S_1' : L^{(\frac{dp}{dp-(n-mp)},p')}(0,1) \to L^{p'}(0,1).
\end{equation}
If $p > \frac nm$, then equation \eqref{gen43} yields
\begin{equation}\label{gen44}
S_1' : L^1(0,1) \to L^{p'}(0,1).
\end{equation}
From equations \eqref{gen42} and \eqref{gen44} we infer, via
 \eqref{E:novabis}, \eqref{iv},~\eqref{E:lorentz-assoc} and~\eqref{v} that
\begin{equation}\label{E:S-1-bded}
S_1\colon L^{p}(0,1)\to
    \begin{cases}
        L^{\frac{dp}{n-mp},p}(0,1) &\text{if $p\in[1,\frac{n}{m})$},
             \\
        L^{\infty,p;-1}(0,1) &\text{if $p=\frac{n}{m}$},
            \\
            L^{\infty}(0,1) &\text{if $p\in(\frac{n}{m},\infty]$}.
    \end{cases}
\end{equation}
Let $p\in(\frac{n-d}{m},\frac{n}{m})$, and let $P$ be the operator defined by \eqref{P}. By~\cite[Theorem~3.5.15, Lemma~4.4.5 and Theorem~4.4.6]{BS},
\begin{equation*}
P \colon L^{\frac{dp}{n-mp},p}(0,1)\to L^{\frac{dp}{n-mp},p}(0,1),
\end{equation*}
and therefore, by \eqref{E:S-1-bded} and \eqref{E:fubini-1},
\begin{equation*}
S_2\colon L^{p}(0,1)\to L^{\frac{dp}{n-mp},p}(0,1).
\end{equation*}
Next, let $p\geq\frac{n}{m}$. Note that
\begin{align*}
\|S_2f\|_{L\sp{\infty}(0,1)}
&\leq
\|S_2f\sp*\|_{L\sp{\infty}(0,1)}
=
\sup_{t \in (0,1)}t\sp{-\frac{m}{n-d}}\int_0\sp{t\sp{\frac{n}{d}}}
f\sp*(s)s\sp{\frac{m}{n}}s\sp{-1-\frac{m}{n}+\frac{m}{n-d}}\,ds\\
&\leq
\sup_{s\in (0,1)}f\sp*(s)s\sp{\frac{m}{n}}
\sup_{t \in (0,1)}t\sp{-\frac{m}{n-d}}\int_0\sp{t\sp{\frac{n}{d}}}
s\sp{-1-\frac{m}{n}+\frac{m}{n-d}}\,ds\approx\|f\|_{L\sp{\frac{n}{m},\infty}(0,1)},
\end{align*}
whence, $S_2\colon L\sp{\frac{n}{m},\infty}(0,1)\to L\sp{\infty}(0,1)$. Combining this piece of information with the embeddings $L\sp{\infty}(0,1)\to L^{\infty,p;-1}(0,1)$, which holds by~\eqref{l1linf}, and $L\sp{\frac{n}{m}}(0,1)\to L\sp{\frac{n}{m},\infty}(0,1)$, which holds  by~\eqref{ii}, tells us that $S_2\colon L\sp{\frac{n}{m}}(0,1)\to L^{\infty,p;-1}(0,1)$. When $p>\frac{n}{m}$, then obviously $L^{p}(0,1)\to L^{\frac{n}{m}}(0,1)\to L^{\frac{n}{m},\infty}(0,1)$, and consequently $S_2\colon L\sp{p}(0,1)\to L^{\infty}(0,1)$.
Altogether, we have shown that
\begin{equation}\label{E:S-2-bded}
S_2\colon L^{p}(0,1)\to
    \begin{cases}
        L^{\frac{dp}{n-mp},p}(0,1) &\text{if $p\in(\frac{n-d}{m},\frac{n}{m})$}
             \\
        L^{\infty,p;-1}(0,1) &\text{if $p=\frac{n}{m}$}
            \\
            L^{\infty}(0,1) &\text{if $p\in(\frac{n}{m},\infty]$}.
    \end{cases}
\end{equation}
Embedding \eqref{E:optimal-leb-emb-i}, and the parallel embeddings stated in parts (ii) and (iii),
follow from Theorem~\ref{T:sufficiency-sub-main}, via equations~\eqref{E:S-1-bded} and~\eqref{E:S-2-bded}.
\\  It remains to establish the optimality of the target spaces in the relevant embeddings under the additional assumption \eqref{E:inf} or \eqref{E:infbar}. To this purpose,
suppose that $W\sp{m,p}_0(\Omega)\to Z(\Omega, \mu)$ for some rearrangement-invariant function norm $\|\cdot\|_{Z(0,1)}$. Assume that $p \in [1,  \frac{n}{m}]$ and, if $d \in (0,n-m)$, also  $p>\frac{n-d}{m}$.
 By Theorem~\ref{T:necessity}, $S_1\colon L^{p}(0,1)\to Z(0,1)$, whence,
by~\eqref{E:novabis}, $S_1'\colon Z'(0,1)\to L^{p'}(0,1)$. On the other hand, equation~\eqref{E:S1-char} tells us that $\|S_1'f^{*}\|_{L^{p'}(0,1)}=\|f\| _{L^{(\frac{dp}{dp-(n-mp)},p')}(0,1)}$, and hence $Z'(0,1)\to L^{(\frac{dp}{dp-(n-mp)},p')}(0,1)$. By~\eqref{emb}, the latter embedding yields $L^{\frac{dp}{n-mp},p}(0,1)\to Z(0,1)$ if $p\in(\frac{n-d}{m},\frac{n}{m})$, and $L^{\infty,p;-1}(0,1)\to Z(0,1)$ if $p=\frac{n}{m}$,  thus proving the optimality of the target space in both  cases. Finally, when $p>\frac{n}{m}$, the target space is trivially optimal, since, by~\eqref{l1linf}, $L^{\infty}(0,1)\to Z(0,1)$.
\end{proof}

\section*{Compliance with Ethical Standards}\label{conflicts}

\smallskip
\par\noindent
{\bf Funding}. This research was partly funded by:
\\ (i)
Italian Ministry of University and Research (MIUR), Research Project Prin 2015 \lq\lq  Partial differential equations and related analytic-geometric inequalities"  (grant number 2015HY8JCC);
\\ (ii) GNAMPA   of the Italian INdAM - National Institute of High Mathematics (grant number not available);
\\  (iii) Grant Agency of the Czech Republic (grant numbers
P201-13-14743S and P201-18-00580S);
\\  (iv) Czech Ministry of Education (grant number 8X17028);
\\ (v) Charles University (project GA UK No. 62315).

\smallskip
\par\noindent
{\bf Conflict of Interest}. The authors declare that they have no conflict of interest.

\end{document}